\documentclass[a4paper, 11pt, reqno]{smfbook}
\title{Propri\'et\'es dynamiques g\'en\'eriques des hom\'eomorphismes conservatifs}
\date{\today}

\usepackage[latin1]{inputenc}
\usepackage[french]{babel}
\usepackage[T1]{fontenc} 
\usepackage{amsfonts}
\usepackage{amsmath}
\usepackage{amssymb}

\usepackage{stmaryrd}
\usepackage{amsthm}
\usepackage{color}
\usepackage{tikz}
\usepackage{enumerate}

\author{Pierre-Antoine~Guih\'eneuf}

\newtheorem{lemme}{Lemme}[chapter]
\newtheorem{theoreme}[lemme]{Th\'eor\`eme}
\newtheorem{prop}[lemme]{Proposition}
\newtheorem{propdef}[lemme]{Proposition et définition}
\newtheorem{coro}[lemme]{Corollaire}
\theoremstyle{definition}
\newtheorem{definition}[lemme]{D\'efinition}
\newtheorem{conv}[lemme]{Convention}
\theoremstyle{remark}
\newtheorem{rem}[lemme]{Remarque}
\newtheorem{app}[lemme]{Application}
\numberwithin{equation}{chapter}
\numberwithin{figure}{chapter}

\newcommand{\B}{\mathcal{B}}
\newcommand{\C}{\mathbf{C}}
\newcommand{\D}{\mathcal{D}}

\newcommand{\F}{\mathcal{F}}
\newcommand{\Hi}{\mathcal{H}}
\newcommand{\N}{\mathbf{N}}
\newcommand{\Pb}{\mathcal{P}}
\newcommand{\R}{\mathbf{R}}
\newcommand{\Sn}{\mathfrak{S}}
\newcommand{\T}{\mathbf{T}}

\newcommand{\V}{\mathcal{V}}
\newcommand{\Q}{\mathbf{Q}}
\newcommand{\Z}{\mathbf{Z}}

\definecolor{bleu}{rgb}{0,0,0.3}

\begin{document}
\frontmatter

\begin{abstract}
Ce mémoire porte sur l'étude des propriétés dynamiques génériques des homéomorphismes conservatifs de variétés compactes. Nous y présentons les principales techniques utilisées pour aborder cette question et mettons l'accent, entre autres, sur l'importance du rôle joué par les permutations dyadiques, ainsi que sur l'utilisation, assez étonnante à première vue, de la topologie faible et du plongement de l'espace des homéomorphismes dans celui des automorphismes bi-mesurables.
\end{abstract}

\begin{altabstract}
This memoir is concerned with the generic dynamical properties of conservative homeomorphisms of compact manifolds. Several important techniques allowing to prove genericity results are presented : we emphasize the important role played by periodic approximations of homeomorphisms, and by the embedding of the space of homeomorphisms in the space of bi-measurable automorphisms.
\end{altabstract}

\email{pguihene@ens-cachan.fr}
\address{Laboratoire de mathématiques CNRS UMR 8628\\Université Paris-Sud 11, Bât. 425\\91405 Orsay Cedex FRANCE}
\altkeywords{Measure-preserving homeomorphism, generic property, dyadic permutation}

\maketitle

\tableofcontents

\mainmatter

\chapter*{Introduction}

Le but de ce mémoire est d'étudier les \emph{propriétés dynamiques génériques des homéo\-mor\-phismes conservatifs des variétés compactes}. Le mot \og conservatif \fg~signifie que nous nous fixerons une variété $X$ munie d'une \emph{bonne} mesure de probabilité $\mu$ et que l'on s'intéressera aux homéomorphismes de $X$ préservant $\mu$. Par \emph{bonne} mesure nous entendons une mesure de probabilité borélienne sans atome, qui charge tout ouvert et qui ne charge pas le bord de la variété. La variété $X$ sera supposée compacte, éventuellement à bord et de dimension supérieure ou égale à deux\footnote{L'étude des homéomorphismes conservatifs en dimension 1 est très réduite : il est connu que les seules variétés compactes de dimension 1 sont, à homéomorphisme près, le segment $[-1,1]$ et le cercle $\R/2\pi\Z$ ; on peut de plus supposer que la mesure $\mu$ est la mesure de Lebesgue. Il y a seulement deux homéomorphismes sur $[-1,1]$ : l'identité et son opposé. La dynamique sur le cercle n'est pas très compliquée elle non plus : les seuls homéomorphismes du cercle $\R/2\pi\Z$ préservant la mesure de Lebesgue sont les rotations.}, sans perdre en généralité nous la supposerons connexe. Nous nous restreindrons au cas où $X$ est une variété \emph{différentielle} pour des raisons de simplicité \footnote{On a alors une triangulation de notre variété, l'ensemble des homéomorphismes est dans ce cas toujours séparable.} ; il est vraisemblable que les résultats que nous allons montrer restent vrais dans le cadre plus naturel des variétés topologiques. Par \emph{propriété dynamique}, nous entendons notamment une propriété qui est stable par conjugaison dans le groupe des homéomorphismes de $X$ préservant la mesure $\mu$. Les propriétés dynamiques peuvent être classées en deux grandes familles : les propriétés topologiques, et les propriétés ergodiques. Ces dernières sont non seulement stables par conjugaison topologique, mais aussi par conjugaison mesurable, autrement dit par conjugaison par tout automorphisme ; nous désignerons par le terme \emph{automorphisme} toute bijection de $X$ bi-mesurable et préservant $\mu$. Rappelons enfin qu'une propriété est dite \emph{générique}, si elle est vérifiée sur au moins une intersection dénombrable d'ouverts denses ; d'après le théorème de Baire, une telle intersection dénombrable d'ouverts denses est elle-même dense. Ici, il s'agira bien sûr d'ouverts denses pour la topologie naturelle sur l'espace des homéomorphismes de~$X$, c'est-à-dire la topologie de la convergence uniforme.

Une caractéristique fondamentale des propriétés génériques est leur stabilité par produit fini (ou dénombrable) : étant donné un nombre fini de propriétés génériques, l'ensemble des homéomorphismes qui satisfont simultanément toutes ces propriétés est encore une intersection dénombrable d'ouverts denses. Ainsi, par abus de langage, nous pourrons parler d'\og homéomorphismes génériques\fg~et lister leurs propriétés dynamiques, telles que la transitivité, l'ergodicité etc.

Comme c'est souvent le cas, les résultats de généricité sont non seulement intéressants en tant que tels, mais fournissent aussi des preuves d'existence d'homéomorphismes vérifiant une propriété donnée. \`A première vue, il n'est pas évident de trouver un exemple explicite d'homéomorphisme conservatif du carré muni de la mesure de Lebesgue, qui soit topologiquement fortement mélangeant. Mais nous montrerons qu'un homéomorphisme conservatif générique de cet espace est topologiquement fortement mélangeant.

Rappelons toutefois que la notion de généricité a une portée limitée : le contraire d'une propriété générique $P$ peut être lui-même vérifié sur un ensemble dense d'homéomorphismes (ce sera d'ailleurs le cas pour chacune des propriétés dynamiques génériques que nous montrerons). Pire encore, les homéomorphismes vérifiant $P$ peuvent être rares du point de vue de la mesure, dans le sens qu'étant donnée une famille continue générique à un nombre fini de paramètres d'homéomorphismes, l'ensemble des paramètres correspondant à des homéomorphismes vérifiant $P$ peut être de mesure de Lebesgue nulle.

\paragraph*{Une br\`eve histoire des homéomorphismes conservatifs génériques}

La question de la généricité de propriétés dynamiques a été posée pour la première fois par G.~Birkhoff et E. Hopf lorsqu'ils conjecturèrent que l'ergodicité --- qui était connue à l'époque sous le nom de transitivité \emph{métrique}, par opposition à la notion actuelle de transitivité \emph{topologique} --- était le cas \emph{général} dans l'ensemble des homéomorphismes conservatifs, en un sens qui n'était pas encore bien défini. Un premier pas en direction de cette conjecture fut effectué par J. Oxtoby en 1937, lorsqu'il montra que la transitivité (topologique) est générique parmi les homéomorphismes conservatifs (pour la topologie uniforme) \cite{Oxtotrans} ; cela ouvrit la voie à l'étude \emph{topologique} de ce que G. Birkhoff et E. Hopf avaient appelé le cas général, à savoir l'étude des propriétés dynamiques vérifiées sur des $G_\delta$ denses. J. Oxtoby et S. Ulam résolurent la conjecture quatre ans plus tard, en montrant la généricité de l'ergodicité parmi les homéomorphismes conservatifs \cite{Oxtoby}. Ils établirent à cette occasion le théorème des mesures homéomorphes (théorème \ref{mesures-homéo}), ce qui leur a permis de se ramener au cas du cube ; déjà dans leur preuve ils utilisaient les subdivisions dyadiques pour approcher tout homéomorphisme conservatif par un autre dont les orbites sont bien distribuées dans les différents cubes de la subdivision. Parallèlement, P. Halmos résolut la conjecture de G. Birkhoff et E. Hopf dans le cas des automorphismes (pour la topologie faible) dans son article fondateur \cite{Halm44}, où il établit par ailleurs la densité des permutations dyadiques. La même année, il démontra que le mélange faible est générique parmi les automorphismes \cite{HalmMix} ; c'est à cette occasion qu'il prouva un lemme qui s'est avéré crucial par la suite, et qui affirme que la classe de conjugaison de tout automorphisme apériodique (c'est-à-dire dont l'ensemble des points périodiques est de mesure nulle) est dense. Dans son article, il soulignait néanmoins le fait que les preuves pour les homéomorphismes et pour les automorphismes étaient plutôt différentes ; il laissait penser qu'il n'y avait aucun lien logique entre les résultats obtenus. Quatre ans plus tard, V. Rokhlin prouva le premier résultat de non généricité dans les automorphismes, à savoir que génériquement, un automorphisme n'est pas fortement mélangeant \cite{Rok}.

Il a fallu attendre la fin des années 60 pour voir de nouveaux développements dans l'obtention de propriétés génériques. A. Katok et A. Stepin, entre autres, ont alors développé la théorie des approximations périodiques. Leur idée était de regarder à quelle vitesse un automorphisme est approché par les permutations dyadiques en topologie faible ; si cette vitesse est assez grande on peut en déduire certaines propriétés dynamiques sur l'automorphisme de départ. Le lien entre cette théorie et notre étude était fait par le théorème qui assure que les automorphismes approchés à une vitesse arbitrairement fixée sont génériques (pour la topologie faible). Cela permit aux auteurs non seulement de retrouver les résultats précédemment obtenus (la généricité de l'ergodicité, du mélange faible, et celle du non mélange fort), mais aussi d'en établir de nouveaux, comme par exemple la généricité de l'entropie métrique nulle ou bien celle de le simplicité du spectre.

A. Katok et A. Stepin établirent en 1970 un théorème similaire de généricité des approximations en distance faible à une vitesse fixée, mais pour les homéomorphismes en topologie uniforme cette fois : les homéomorphismes approchés en topologie \emph{faible} à une vitesse arbitrairement fixée sont génériques en topologie \emph{uniforme}. Ainsi, tous les résultats d'approximation obtenus quelques années auparavant par la technique des approximations périodiques s'appliquaient directement aux homéomorphismes ; la totalité des résultats de généricité connus pour les automorphismes (en topologie faible) devenaient aussi vrais pour les homéomorphismes en topologie uniforme. Ils remarquèrent par ailleurs que la seule chose qui manquait pour faire le lien direct entre les propriétés génériques des homéomorphismes et des automorphismes était un analogue topologique du lemme d'Halmos sur la densité des classes de conjugaison des automorphismes.

Un premier pas dans ce sens fut franchi en 1973, lorsque S. Alpern adapta le théorème de Lax \cite{Lax}, établi en 1970 et initialement motivé par des questions d'analyse numérique pour les homéomorphismes conservatifs, pour redémontrer la généricité de la transitivité parmi les homéomorphismes, replaçant par là la preuve de J. Oxtoby dans un contexte plus général. À défaut d'un réel analogue topologique du lemme d'Halmos\footnote{Qui se serait énoncé \og la classe de conjugaison de tout homéomorphisme apériodique est dense dans l'ensemble des homéomorphismes\fg.}, il montra en 1978 une variante dans l'ensemble des automorphismes mais pour la topologie forte : l'ensemble des homéomorphismes est inclus dans l'adhérence (pour la topologie forte) de la classe de conjugaison (dans l'ensemble des automorphismes) d'un automorphisme apériodique. Ceci lui permit d'établir son très beau théorème de transfert \cite{AlpernGene}, \cite{AlpernTopo}, qui affirme que toute propriété ergodique générique parmi les automorphismes pour la topologie faible l'est aussi parmi les homéomorphismes pour la topologie forte. L'étude de la dynamique générique des homéomorphismes conservatifs de variétés non compactes a été amorcée l'année suivante, lorsque V. Prasad a montré que l'ergodicité est générique dans l'espace $\R^n$ muni de la mesure de Lebesgue \cite{Prasad}.

Par la suite, la recherche sur le sujet a été semble-t-il moins active. Cependant, E. Glasner et J. King établirent en 1993 la loi du 0-1 dans l'ensemble des automorphismes \cite{4} : pour toute propriété ergodique $(P)$ portant sur l'ensemble des automorphismes, soit $(P)$ est générique, soit son contraire l'est. Par théorème de transfert, on obtient une loi du 0-1 pour les propriétés ergodiques dans l'ensemble des homéomorphismes. Cette loi du 0-1 permit à E. Glasner et B. Weiss de montrer en 2008 que toute classe de conjugaison est maigre \cite{5}.

Contrairement aux propriétés ergodiques, les propriétés dynamiques \emph{topologiques} des homéomorphismes conservatifs génériques n'ont apparemment pas été explorées de manière systématique. Hormis l'article pionnier dans lequel J. Oxtoby montre la généricité de la transitivité topologique \cite{Oxtotrans}, et l'article dans lequel S. Alpern donne une nouvelle preuve de cette  propriété~\cite{AlpernOuique}, la seule publication que nous connaissons sur le sujet est un article de l'an 2000 \cite{6}, dans lequel F. Daalderop et R. Fokkink montrent la généricité du \emph{chaos maximal} au sens de R. Devaney (voir la définition 8.5 page 50 de \cite{Dev}). \`A notre connaissance, la généricité de l'entropie topologique infinie, et du mélange topologique fort, dont nous donnons des preuves basées sur des techniques classiques dans les chapitre 1 et 2, n'avaient jamais été énoncées explicitement pour les homéomorphismes conservatifs.

L'article de J. Choksi et V. Prasad \cite{Cho} présente un survol historique plus précis (jusqu'en 1982) ; les auteurs y dégagent trois périodes importantes : la première comporte les travaux fondateurs de J. Oxtoby et S. Ulam, ainsi que les études des espaces d'automorphismes faites par P. Halmos et V. Rokhlin, la seconde est portée par les travaux de A. Katok et A. Stepin sur les vitesses d'approximation, et la troisième commence lorsque S. Alpern établit le lien entre les dynamiques génériques des automorphismes et des homéomorphismes.

\paragraph*{Organisation générale du mémoire}

Par souci de pédagogie, nous avons choisi de présenter les résultats dans un ordre logique plutôt qu'historique. Nous commençons chaque chapitre par la présentation d'une nouvelle technique, qui permet ensuite d'obtenir des résultats de généricité. Les techniques présentées dans les deux premiers chapitres sont assez naturelles ; il apparaît néanmoins assez vite que l'ensemble des homéomorphismes muni de la topologie uniforme est trop petit, trop rigide ; on parvient à trouver des résultats de généricité plus forts non seulement en s'intéressant à l'inclusion de l'ensemble des homéomorphismes dans celui des automorphismes, mais aussi, et c'est plus étonnant, en munissant ces ensembles d'une topologie auxiliaire, appelée topologie \emph{faible} (par opposition à la topologie \emph{forte}), dans laquelle deux automorphismes sont proches s'ils diffèrent très peu sur un ensemble de mesure très grande. Par souci de pédagogie, nous prouvons directement les généricités de propriétés dynamiques simples, sans attendre de les obtenir comme corollaires de propriétés plus fortes.

Il est possible d'étendre la notion de \emph{subdivision dyadique}, naturelle pour un cube de dimension $n\ge 2$, à toute variété $X$. Cela permet de définir les \emph{permutations dyadiques} comme des automorphismes préservant une subdivision dyadique, et agissant comme une translation sur chaque cube de la subdivision. L'idée d'approximation, aussi bien en topologie forte qu'en topologie faible, par des permutations, souvent cycliques, joue alors un rôle central dans l'étude des propriétés génériques des homéomorphismes. La quasi totalité des preuves de généricité commence par une approximation par une permutation. 

Le premier résultat présenté est l'approximation uniforme des homéomorphismes par des permutations cycliques, connu sous le nom de théorème de Lax. Ce théorème, combiné avec un résultat d'extension des applications finies, donne un premier résultat : la généricité de la transitivité ; en travaillant un peu plus on obtient celle du mélange faible topologique.

Dans la seconde partie, elle aussi purement topologique, nous donnons un énoncé de \og modification locale \fg~: sous certaines conditions il est possible de remplacer localement un homéomorphisme par un autre. En combinant cette technique avec l'approximation des homéomorphismes en topologie uniforme, nous montrons que l'ensemble des applications ayant un ensemble de points périodiques dense est générique. À l'aide de la notion d'intersection markovienne, nous prouvons ensuite la généricité de l'entropie topologique infinie et celle du mélange fort.

Nous commençons à aborder les propriétés ergodiques dans la troisième partie. Des considérations sur la vitesse d'approximation par des permutations cycliques (mesurée par la distance faible), introduites par A. Katok et A. Stepin, conduisent à bon nombre de résultats de généricité parmi les homéomorphismes conservatifs\footnote{Qui sont automatiquement aussi vrais pour les automorphismes.}, dont la généricité de l'ergodicité (théorème d'Oxtoby-Ulam), du mélange faible, du non mélange fort, de l'entropie métrique nulle etc. Nous consacrons la fin de ce chapitre au \emph{type spectral} associé à un automorphisme : après l'avoir défini, nous montrons que cette mesure est génériquement sans atome et étrangère à la mesure de Lebesgue.

Dans le début de la quatrième partie, nous justifions plus précisément l'étude de l'ensemble des automorphismes muni de la topologie faible en énonçant un résultat de densité des homéomorphismes parmi les automorphismes. Il y a, à ce moment, encore du travail à faire pour établir le théorème de transfert, qui affirme que toute propriété générique parmi les automorphismes munis de la topologie faible l'est aussi parmi les homéomorphismes munis de la topologie forte, mais on peut d'ores et déjà en énoncer une forme faible, qui fournit une nouvelle preuve de la généricité de l'ergodicité (théorème d'Oxtoby-Ulam).

Dans la cinquième partie, nous nous intéressons plus précisément aux liens qu'il peut y avoir entre les propriétés génériques des homéomorphismes préservant la mesure et celles des automorphismes préservant la mesure. On y montre le théorème \emph{de transfert}, dû à S. Alpern ; ce théorème permet de \og transférer \fg~les résultats de généricité (pour les propriétés ergodiques) de l'espace des automorphismes muni de la topologie faible vers l'espace des homéomorphismes muni de la topologie forte. Cela justifie la similitude des propriétés génériques des homéomorphismes et des automorphismes établie auparavant. On en profite pour donner de nouvelles preuves de la généricité du mélange faible, et du non-mélange fort, et exposer une technique d'obtention de propriétés génériques à l'aide du type spectral, technique que l'on illustre ensuite par une preuve du fait que génériquement, le type spectral est étranger à une mesure donnée.

Finalement, la sixième et dernière partie est consacrée à la loi du 0-1 pour l'ensemble des automorphismes préservant la mesure, due à E. Glasner et J. King \cite{4} : chaque propriété ergodique est soit générique soit maigre. On en déduit que toute classe de conjugaison dans les automorphismes, et donc dans les homéomorphismes, est maigre.

\paragraph*{Une liste de propriétés génériques des homéomorphismes conservatifs}

Dans ce mémoire, plutôt que de chercher à établir une liste la plus exhaustive possible des propriétés dynamiques génériques des homéomorphismes conservatifs, nous avons préféré présenter les diverses techniques qui permettent de trouver de telles propriétés. Voici néanmoins une liste de propriétés que nous allons montrer. Un homéomorphisme conservatif générique~:
\begin{itemize}
\item est topologiquement transitif (théorème \ref{transitivité}), et même topologiquement faiblement mélangeant (théorème \ref{mélange topo faible}), et encore mieux topologiquement fortement mélangeant (théorème \ref{th-mel-fort}),
\item est apériodique (proposition \ref{aper}), et même ergodique (théorème d'Oxtoby-Ulam, théorèmes \ref {OxtoUlApprox} et \ref{ergo}), encore mieux faiblement mélangeant (théorème \ref{mélfaiblgéné} et corollaire~\ref{mélange faible}), encore mieux $\alpha$-mélangeant pour tout $\alpha\in[0,1]$ (partie \ref{alphamél}), mais pas fortement mélangeant (théorème \ref{MélApprox} et corollaire \ref{mélange pas}),
\item possède un ensemble dense de points périodiques (théorème \ref{per-resid}),
\item a une entropie topologique infinie (théorème \ref{entropie-topologique}), mais une entropie métrique nulle (corollaire \ref{entrmetnullgéné}),
\item est rigide\footnote{C'est-à-dire qu'il existe une sous-suite de ses itérés tendant faiblement vers l'identité.}, (proposition \ref{rigid}),
\item est de rang un\footnote{On dit qu'un homéomorphisme $f$ est \emph{de rang un} si toute partition mesurable finie de $X$ peut être approchée par une partition engendrée par des itérés par $f$ deux à deux disjoints d'une seule partie $A$ de $X$ ; pour une définition plus précise voir la définition \ref{defrg}.} (théorème \ref{rang1F}), donc à spectre simple (théorème \ref{specsimpl}), et standard\footnote{Un homéomorphisme $f$ est dit \emph{standard} s'il existe un borélien $A$ tel que l'application induite $f_A$ soit métriquement conjuguée à un odomètre ; pour une définition plus précise voir la définition \ref{Enzo}.} (théorème \ref{pffff}),
\item a un type spectral sans autre atome que 1 (proposition \ref{speccon}) et étranger à la mesure de Lebesgue (proposition \ref{singLeb}), et même étranger à toute mesure donnée (théorème \ref{specortho}),
\item est en dehors d'une classe de conjugaison donnée (corollaire \ref{homeomaigr}).
\end{itemize}

\paragraph*{Sur les propriétés dynamiques génériques dans différents contextes}

La g\'en\'ericit\'e ou non de nombreuses propri\'et\'es dynamiques d\'epend du cadre dans lequel on se place, par exemple de la pr\'esence ou non d'une mesure pr\'eserv\'ee (syst\`emes conservatifs \emph{versus} systèmes dissipatifs), ou de la r\'egularit\'e des transformations consid\'er\'ees (automorphismes, hom\'eomorphismes, diff\'eomorphismes de classe $C^1$, diff\'eomorphismes de classe $C^2$, etc.\footnote{\`A chaque fois on munit l'espace de transformations choisi de la topologie \og naturelle\fg~qui en fait un espace de Baire.}). Prenons un exemple : la transitivité topologique.
\begin{itemize}
\item Comme nous l'avons déjà dit, J. Oxtoby a montré qu'un hom\'eomorphisme conservatif g\'en\'erique (pour la topologie uniforme) est transitif (th\'eor\`eme~\ref{transitivité}).
\item Beaucoup plus récemment, C. Bonatti et S. Crovisier ont r\'eussi \`a montrer dans \cite{Bonatti} qu'il en est de m\^eme pour un diff\'eomorphisme de classe $C^1$ conservatif g\'en\'erique (pour la topologie $C^1$).
\item Par contre, le th\'eor\`eme KAM implique que, sur toute surface compacte, il existe des ouverts de diff\'eomorphismes conservatifs de classe $C^r$ pour $r\geq 4$ (pour la topologie $C^r$) sur lesquels les éléments ne sont pas transitifs (voir par exemple la section 4 de \cite{Herman}).
\item Enfin, dans le cas dissipatif (c'est-à-dire si on considère des systèmes dynamiques qui ne préservent pas nécessairement une mesure fixée), il est tr\`es facile de voir qu'il existe un ouvert dense d'hom\'eomorphismes (et donc \emph{a fortiori} un ouvert dense de diff\'eomorphismes $C^r$ pour tout $r$) qui poss\`edent des orbites périodiques attractives, et ne peuvent donc pas \^etre topologiquement transitifs.
\end{itemize}
Les r\'esultats de g\'en\'ericit\'e obtenus dans les diff\'erents cadres sont en g\'en\'eral logiquement ind\'ependants les uns des autres\footnote{\`A l'exception notoire des r\'esultats o\`u l'on r\'eussit \`a prouver qu'une propri\'et\'e est v\'erifi\'ee sur un ouvert dense. Par exemple, tout propri\'et\'e satisfaite par un ouvert dense de diff\'eomorphismes de classe $C^r$ (pour la topologie $C^r$) est \'egalement satisfaite par un ouvert dense de diff\'eomorphismes de classe $C^{r'}$ (pour la topologie $C^{r'}$) pour tout $r'\geq r$, car la topologie $C^{r'}$ est plus fine que la topologie $C^{r}$.}. Il est n\'eanmoins toujours enrichissant de comparer ces r\'esultats, ne serait-ce que pour relativiser la port\'ee de chacun d'entre eux. Voici quelques indications bibliographiques tr\`es sommaires relatives aux différents contextes~:
\begin{itemize}
\item Dans \cite{Akin}, E. Akin, M. Hurley et J. Kennedy ont effectu\'e une \'etude assez compl\`ete des propri\'et\'es dynamiques g\'en\'eriques des hom\'eomorphismes dissipatifs d'une vari\'et\'e compacte. Cette \'etude systématique s'av\`ere cependant un peu d\'ecevante, dans la mesure o\`u les r\'esultats vont tous dans la m\^eme direction et proc\`edent d'une m\^eme heuristique~: la dynamique d'un hom\'eomorphisme dissipatif g\'en\'erique \og contient\fg~simultan\'ement tous les comportements sauvages que l'on peut imaginer\footnote{À ce sujet, les auteurs déclarent dans l'introduction : \og At this point we made a discovery which astonished us until it was interpreted for us by our elderly, imaginary, topologically inclined aunt: ``Let your homeomorphisms be wild. It will make them stable.''\fg}. Leur mémoire ne concerne que la dynamique \emph{topologique} des hom\'eomorphismes~; tr\`es r\'ecemment, F. Abdenur et  M. Anderson se sont intéress\'es aux propri\'et\'es \emph{ergodiques} g\'en\'eriques des hom\'eomorphismes dissipatifs~: il s'agit de consid\'erer le comportement des sommes de Birkhoff le long de l'orbite d'un point typique pour la mesure de Lebesgue pour un hom\'eomorphisme lui-m\^eme g\'en\'erique (voir~\cite{Anderson}).
\item L'\'etude des propri\'et\'es g\'en\'eriques des diff\'eomorphismes de classe $C^1$ est un sujet tr\`es actif ---~aussi bien dans le cadre conservatif que dans le cadre dissipatif~--- et est devenu de ce fait extr\^emement touffu. Dans le cadre dissipatif, les travaux sont guid\'es par un ensemble de conjectures dues \`a J. Palis qui d\'ecrivent ce que pourrait \^etre la dynamique, d'un point de vue topologique et d'un point de vue ergodique, d'un diff\'eomorphisme de classe $C^1$ \og typique\fg. On pourra se faire une id\'ee de la richesse de ce sujet en consultant le séminaire Bourbaki~\cite{Bbk} ou le compte-rendu de conférence de l'ICM \cite{Bona} de C. Bonatti, le m\'emoire d'habilitation~\cite{Crovisier}, ou bien l'article de survol \cite{Gambaudo} de S. Crovisier.
\item Si on la compare \`a la situation pour les diff\'eomorphismes de classe $C^1$, l'\'etude des propri\'et\'es dynamiques g\'en\'eriques des diff\'eomorphismes de classe $C^r$ avec $r\geq 2$ est bloqu\'ee par l'absence d'un \emph{closing lemma} en topologie $C^r$ pour $r\geq 2$ (concernant le closing lemma en topologie $C^1$, voir le livre de M.-C. Arnaud~\cite{Arnaud}). À notre connaissance, le seul cas dans lequel on ait des r\'esultats autres qu'anecdotiques, outre le cas trivial de la dimension~1, est celui des diff\'eomorphismes conservatifs des surfaces~; on consultera \`a ce sujet l'article de J. Franks et P. Le Calvez~\cite{Franks}.
\item Les propriétés génériques des homéomorphismes d'ensembles moins réguliers que les variétés tels que l'ensemble de Cantor ou le cube de Hilbert  ont fait l'objet de quelques articles tels que ceux de M. Hochman \cite{Hochman} ou de E. Akin, E. Glasner et B.Weiss \cite{Weiss}. Là aussi, les résultats diffèrent notablement de ceux qui seront exposés durant ce mémoire ; par exemple le mélange fort (dans son sens ergodique) est générique sur l'ensemble de Cantor.
\item Enfin, comme nous l'avons déjà dit, l'étude de la dynamique générique des automorphismes conservatifs est maintenant assez riche, notamment grâce à la technique d'approximation périodique de A. Katok et A. Stepin \cite{KatokS1} qui s'est révélé être un outil extrêmement puissant (voir \cite{Katok}, \cite{KatokT}). Cette dynamique est intimement liée à celle des homéomorphismes, à cause du théorème de transfert de S. Alpern ; c'est pourquoi nous nous y intéresserons de près.
\end{itemize}

Notons que les résultats de généricité que nous obtenons peuvent également être faux si on enlève l'hypothèse de compacité : par exemple l'ergodicité n'est pas générique dans certaines variétés non compactes (pour un exemple, voir le chapitre 13 du livre de S. Alpern et V. Prasad \cite{2}).

\paragraph*{Motivations de l'écriture du présent mémoire}

Ce survol est une version largement modifiée du mémoire de stage de fin de M2, que j'ai effectué sous la direction de François Béguin en 2010. Nous avions décidé, avec François, de le reprendre pour en faire un survol plus complet sur le sujet ; j'y ai consacré les trois premiers mois de ma thèse.

Les lecteurs pourraient trouver incongru que nous ayons éprouvé le besoin d'écrire un survol sur les propriétés génériques des homéomorphismes conservatifs, alors que l'auteur d'un des plus beau théorème du sujet, S. Alpern, a déjà publié avec V. Prasad un petit livre intitulé \og\emph{Typical dynamics of volume preserving homeomorphisms}\fg\cite{2}. En fait, contrairement \`a ce que son titre sugg\`ere, ce livre n'est pas réellement un survol des propriétés génériques des homéomorphismes conservatifs, mais plut\^ot une pr\'esentation tr\`es compl\`ete du th\'eor\`eme de transfert (voir le chapitre~\ref{partie 3} du présent mémoire), et de ses diverses généralisations (en particulier aux variétés non-compactes). Même les conséquences du théorème de transfert y sont très peu explorées. Le court article de J. Choksi et V. Prasad \cite{Cho} est, quant \`a lui, un véritable survol historique des propriétés génériques des homéomorphismes conservatifs, mais celui-ci est très concis ; c'est pourquoi il nous a semblé intéressant d'écrire un survol plus complet.

\paragraph*{Remerciements}

Je tiens bien évidemment à remercier François Béguin pour son investissement constant durant ces quelques mois de rédaction. Il a fait preuve de beaucoup de résistance face aux nombreuses relectures et corrections successives de ce survol qui s'avère finalement assez long. Je ne cache pas que certaines parties du mémoire sont, peu ou prou, de sa plume. Je voudrais également remercier Sylvain Crovisier et Frédéric Le Roux pour les quelques discussions mathématiques fructueuses que nous avons eues ensemble, ainsi qu'Emmanuel Militon pour ses nombreuses remarques pertinentes.

\chapter{Définitions, notations et remarques importantes}

Nous commençons ici par définir les objets que nous utiliserons tout au long du mémoire.
Une fois les définitions des espaces d'homéomorphismes $\mathrm{Homeo}(X,\mu)$ et d'automorphismes $\mathrm{Auto}(X,\mu)$ bien posées, nous énonçons un résultat de M. Brown, J. Oxtoby et S. Ulam qui affirme l'existence d'un homéomorphisme \og persque injectif\fg~du cube $[-1,1]^n$ vers la variété $X$, qui envoie la mesure de Lebesgue sur la mesure $\mu$. Ceci nous permet de définir des notions de subdivision dyadique et de permutation dyadique sur $X$. On pourra trouver les propriétés importantes des espaces $\mathrm{Homeo}(X,\mu)$ et $\mathrm{Auto}(X,\mu)$ dans l'annexe \ref{AAA}.

\paragraph*{La vari\'et\'e $X$ et la mesure $\mu$.}

Dans toute la suite, $X$ désignera une variété différentielle compacte connexe éventuellement à bord, de dimension $n\ge 2$. Nous la supposons  munie d'une métrique riemannienne, qui induit une distance notée $\mathrm{dist}$.

Sur cette variété $X$, nous considérons une bonne mesure de probabilité borélienne~$\mu$ :

\begin{definition}\label{bonne mesure}
Une mesure de probabilité borélienne $\mu$ sur $X$ est appelée une \emph{bonne mesure} si elle vérifie les conditions suivantes :
\begin{enumerate}[(i)]
\item elle est sans atome (ne charge pas les points)
\item elle est de support total (strictement positive sur les ouverts non vides)
\item elle est nulle sur le bord de $X$.
\end{enumerate}
\end{definition}
\bigskip

\emph{On fixe une fois pour toute une variété $X$ et une bonne mesure $\mu$ sur $X$.}
\bigskip

D'autre part, on munit une fois pour toutes l'ensemble $\B$ des boréliens de $X$ de la métrique $d$ définie par 
\[d(A,B) = \mu(A\Delta B).\]
Celle-ci rend l'ensemble $\B$ complet, et polonais pour la topologie associée à $d$ (voir le lemme \ref{polonais}).

\paragraph*{Les espaces $\mathrm{Homeo}(X,\mu)$ et $\mathrm{Auto}(X,\mu)$, les topologies faibles et fortes.}

Lorsque ces conditions sont vérifiées, on note $\mathrm{Homeo}(X,\mu)$ l'ensemble des homéomor\-phis\-mes de $X$ préservant la mesure $\mu$, et $\mathrm{Auto}(X,\mu)$ l'ensemble des bijections de $X$, bi-mesurables et préservant la mesure $\mu$, où l'on identifie deux bijections qui coïncident presque partout. L'ensemble $\mathrm{Homeo}(X,\mu)$ s'injecte alors naturellement dans l'ensemble $\mathrm{Auto}(X,\mu)$ (puisque deux homéomorphismes qui coïncident presque partout sont égaux).

Nous munissons l'espace $\mathrm{Auto}(X,\mu)$ de deux topologies distinctes : une dite \emph{forte}, ou \emph{uniforme}, et une dite \emph{faible}. La distance forte sur $\mathrm{Auto}(X,\mu)$ est définie comme la distance uniforme associée au supremum essentiel découlant de $\mathrm{dist}$ :
\[d_{\mathit{forte}}(f,g) = \underset{x\in X}{\sup\mbox{ess}}\ \mathrm{dist}(f(x),g(x)).\]
La topologie faible est elle définie par la distance :
\[d_{\mathit{faible}}(f,g) = \inf\big\{\alpha \mid \mu\{x\mid \mathrm{dist}( f(x),g(x))>\alpha\}<\alpha\big\}.\]
Comme son nom l'indique, elle est plus faible que la topologie uniforme. Ces deux distances induisent deux topologies sur l'espace $\mathrm{Auto}(X,\mu)$ ; une dite forte (ou uniforme) et une dite faible.

Par restriction, ceci définit une distance forte et une distance faible, une topologie uniforme et une topologie faible sur $\mathrm{Homeo}(X,\mu)$. Remarquons que sur $\mathrm{Homeo}(X,\mu)$, le supremum essentiel dans la définition de la distance forte est bien sûr un vrai supremum.

La topologie \og naturelle \fg~sur $\mathrm{Homeo}(X,\mu)$ est bien sûr la topologie forte, de même la topologie \og naturelle \fg~sur $\mathrm{Auto}(X,\mu)$ est la topologie faible ; nous verrons que $\mathrm{Homeo}(X,\mu)$ muni de la topologie forte, et $\mathrm{Auto}(X,\mu)$ muni de la topologie faible, sont des espaces de Baire (voir le paragraphe suivant et l'appendice \ref{AAA}). Lorsque rien ne sera précisé, l'espace $\mathrm{Homeo}(X,\mu)$ sera par défaut muni de la distance forte et $\mathrm{Auto}(X,\mu)$ sera muni de la distance faible. Il est à première vue beaucoup moins naturel de considérer l'ensemble $\mathrm{Homeo}(X,\mu)$ muni de la topologie faible, ou l'ensemble $\mathrm{Auto}(X,\mu)$ muni de la topologie forte ; ces deux espaces topologiques joueront néanmoins un rôle très important dans la suite.


\paragraph*{Espaces de Baire}

Profitons-en pour rappeler la terminologie concernant les espaces de Baire. Un espace topologique est dit \emph{de Baire} si le théorème de Baire y est vrai\footnote{Le théorème de Baire établit que dans un espace complet, une intersection dénombrable d'ouverts denses est elle-même dense. Un espace est dit \emph{de Baire} si la conclusion de ce théorème reste vraie.}. Le fait que les espaces $\mathrm{Homeo}(X,\mu)$ et $\mathrm{Auto}(X,\mu)$ soient de Baire\footnote{Nous le montrons dans l'appendice \ref{AAA}.} est fondamental pour notre étude, puisqu'on veut que la propriété de densité se conserve par intersection dénombrable d'ouverts. Une intersection dénombrable d'ouverts sera appelée un $G_\delta$. Un ensemble contenant un $G_\delta$ dense sera dit \emph{résiduel}, \emph{gras} ou de \emph{seconde catégorie} ; nous utiliserons systématiquement le terme \og gras \fg, qui nous semble plus parlant. Selon la terminologie usuelle, une propriété vérifiée sur un ensemble gras sera dite \emph{générique}. Le complémentaire d'un $G_\delta$ dense sera appelé un $F_\sigma$ d'intérieur vide, et un ensemble contenu dans un $F_\sigma$ d'intérieur vide sera dit \emph{maigre} (terme que nous utiliserons systématiquement) ou de \emph{première catégorie}.

\`A chaque fois que nous parlerons d'une propriété générique ou d'un ensemble ouvert, fermé, gras ou maigre de $\mathrm{Homeo}(X,\mu)$ sans préciser la topologie, il s'agira de la topologie uniforme. Au contraire, à chaque fois que nous parlerons d'une propriété générique ou d'un ensemble ouvert, fermé, gras ou maigre de $\mathrm{Auto}(X,\mu)$ sans préciser la topologie, il s'agira de la topologie faible.

\paragraph*{Du cube muni de la mesure de Lebesgue \`a la vari\'et\'e $X$ munie de la mesure $\mu$.}

On notera par la suite $I=[0,1]$ le segment unité et $\mathrm{Leb}$ la mesure de Lebesgue sur le cube $I^n$.

Nous allons voir que les travaux de M. Brown, J. Oxtoby et S. Ulam montrent que la variété $X$ munie de la mesure $\mu$ est \og presque l'image homéomorphe\fg~du cube $I^n$ muni de la mesure de Lebesgue $\mathrm{Leb}$. Plus précisément on a le théorème :

\begin{theoreme}[Brown]\label{Brown}
Il existe une application continue $\phi : I^n\to X$ telle que :
\begin{itemize}
\item $\phi$ est surjective,
\item $\phi_{| \mathring{I^n}}$ est un homéomorphisme sur son image,
\item $\phi (\partial I^n)$ est un sous-ensemble fermé d'intérieur vide de $X$, disjoint de $\phi (\mathring{I^n})$.
\end{itemize}
\end{theoreme}

Pour une preuve de ce théorème, on pourra se référer à l'article de Brown \cite{13}. L'application $\phi$ de ce théorème a l'inconvénient de ne pas transformer la mesure de Lebesgue $\mathrm{Leb}$ en la mesure $\mu$. Cette propriété s'obtient à l'aide du théorème des mesures homéomorphes :

\begin{theoreme}[des mesures homéomorphes, Oxtoby-Ulam, \cite{Oxtoby}]\label{mesures-homéo}
Si $\nu$ est une bonne mesure borélienne sur $I^n$, alors elle est est homéomorphe à la mesure de Lebesgue, dans le sens qu'il existe un homéomorphisme $h$ tel que $h_*(\nu) = \mathrm{Leb}$. De plus, étant donné un homéomorphisme $g$ de $I^n$, $h$ peut être choisi comme étant égal à $g$ sur $\partial I^n$.
\end{theoreme}

Ce théorème est montré dans l'appendice 2 de \cite{2}. Comme corollaire on obtient une extension du théorème de Brown :

\begin{coro}[Oxtoby, Ulam, \cite{Oxtoby}]\label{Brown-mesure}
Il existe une application $\phi : I^n\to X$ telle que :
\begin{enumerate}
\item $\phi$ est surjective
\item $\phi_{|\mathring{I^n}}$ est un homéomorphisme sur son image
\item $\phi(\partial I^n)$ est un sous-ensemble fermé d'intérieur vide de $X$, disjoint de $\phi (\mathring{I^n})$
\item $\mu(\phi(\partial I^n))=0$
\item $\phi^*(\mu) = \mathrm{Leb}$.
\end{enumerate}
\end{coro}

\bigskip

\textit{On fixe une fois pour toute une application continue surjective $\phi:I^n\to X$ donn\'ee par le corollaire~\ref{Brown-mesure}.}

\paragraph*{Subdivisions et permutations dyadiques.}

Grâce à l'application $\phi$, on peut définir la notion de \emph{subdivision dyadique} pour des variétés quelconques. Informellement, la subdivision dyadique d'ordre $m$ de la variété $X$ sera l'image par l'application $\phi$ de la subdivision dyadique usuelle du cube $I^n$.

\begin{definition}[Subdivisions dyadiques]\label{subdiv}
Nous appellerons \emph{cube dyadique d'ordre $m$} de $X$ toute partie de $X$ qui est l'image par $\phi$ d'un cube du type $\prod_{i=1}^n [\frac{k_i}{2^m},\frac{k_i+1}{2^m}]$ avec $0\leq k_i\leq 2^m-1$ pour tout $i$. Nous appellerons \emph{centre} du cube dyadique $\phi\big(\prod_{i=1}^n [\frac{k_i}{2^m},\frac{k_i+1}{2^m}]\big)$ le point $\phi\big(\frac{k_1+1/2}{2^m},\dots,\frac{k_m+1/2}{2^m}\big)$. Nous appellerons \emph{subdivision dyadique d'ordre $m$} de la variété $X$, et nous noterons $\D_m$, la collection des cubes dyadiques d'ordre $m$ de $X$. Nous noterons $q_m$ le nombre de cubes de la subdivision dyadique $\D_m$ ; si la variété $X$ est de dimension $n$, on aura $q_m = 2^{nm}$.
\end{definition}

\begin{rem}
La suite $(\D_m)_{m\in\N}$ des subdivisions dyadiques de $X$ satisfait les propriétés suivantes :
\begin{itemize}
\item Pour tout $m$, les cubes constituant $\D_m$ sont des domaines (c'est-à-dire des adhérences d'ouverts) connexes.
\item Pour tout $m$, $\D_m$ est un recouvrement de $X$ par un nombre fini de cubes ayant tous la même mesure, et d'intérieurs deux à deux disjoints.
\item Pour tout $m$, la subdivision $\D_{m+1}$ raffine la subdivision $\D_m$.
\item La mesure des cubes constituant $\D_m$, ainsi que le maximum des diamètres de ces cubes, tendent vers 0 lorsque $m$ tend vers l'infini.
\end{itemize}
\end{rem}

\begin{rem} \label{Hadi}
En fait, on peut définir de la même manière, pour tout entier strictement positif $q$, une subdivision $q$-adique d'ordre $m$ de $X$ et une notion de permutation $q$-adique d'ordre $m$ de $X$. Nous nous concentrons ici sur les subdivisions et les permutations \emph{dyadiques} par souci de simplicit\'e, et parce que ces notions nous suffirons presque partout (en fait partout sauf dans la preuve du lemme \ref{lem mélange topo faible}).
\end{rem}


\begin{definition}\label{translat}
Soient $U,V$ deux parties de $X$ telles que $U,V\subset \phi(\mathring{I^n})$ et $f\in \mathrm{Auto}(X,\mu)$ un automorphisme qui envoie $U$ sur $V$. Nous dirons que $f$ est \emph{une translation en restriction \`a $U$} si l'application
\[\phi^{-1}\circ f\circ \phi\ :\ \phi^{-1}(U)\longrightarrow \phi^{-1}(V)\]
est la restriction d'une translation de $\R^n$.
\end{definition}

Autrement dit, $f$ est une translation en restriction à $U$ si c'est, dans la carte $\phi$, la restriction d'une translation de $\R^n$. Notons que l'hypoth\`ese $U,V\subset \phi(\mathring{I^n})$ permet de consid\'erer l'inverse de $\phi$ sur $U$ et $V$.

\begin{definition}[Permutation dyadique]\label{PerDy}
Nous appellerons \emph{permutation dyadique d'ordre $m$} toute automorphisme $f\in Auto(X,\mu)$ qui permute les cubes de la subdivision dyadique $\D_m$, et qui est une translation (au sens de la définition \ref{translat}) en restriction \`a l'int\'erieur de chaque cube de $\D_m$.
\end{definition}

\begin{rem}
\begin{itemize}
\item On voit les permutations dyadiques d'ordre $m$ comme des automorphismes~; on n'a donc pas \`a se préoccuper de leur définition sur les bords des cubes de $\D_m$, qui sont de mesure nulle.
\item Les permutations dyadiques d'ordre $m$ sont des automorphismes périodiques de période au plus $q_m$.
\end{itemize}
\end{rem}
\bigskip

\paragraph*{L'espace $\mathrm{Auto}(X,\mu)$ ne d\'epend pas de $X$ et $\mu$.}

Enfin, on a une propriété plus forte que le corollaire \ref{Brown-mesure} pour les automorphismes : les ensembles $\mathrm{Auto}(X,\mu)$ sont tous isomorphes.

\begin{theoreme}\label{leb-stand}
Il existe une application $\widetilde{\phi} : (X,\mu) \to (I,\mathrm{Leb})$ bijective bi-mesurable préservant la mesure. La conjugaison par $\widetilde{\phi}$ induit un homéomorphisme entre les espaces $\mathrm{Auto}(X,\mu)$ et $\mathrm{Auto}(I,\mathrm{Leb})$ munis de leurs topologies faibles respectives.
\end{theoreme}

\begin{proof}[Preuve du théorème \ref{leb-stand}] Ceci découle immédiatement du théorème 2.4.1 de \cite{Ito} et de la seconde caractérisation de la convergence en topologie faible.
\end{proof}

Moralement, cette propriété dit que la généricité des propriétés ergodiques des ensembles $\mathrm{Auto}(X,\mu)$ est indépendante de l'espace $X$ et de la mesure $\mu$.

\chapter[Approximations en topologie uniforme]{Approximations par des permutations en topologie uniforme}

Comme on l'a déjà dit dans l'introduction, une méthode fondamentale pour l'étude des propriétés génériques des homéomorphismes est l'approximation par des permutations ; ce chapitre est centré sur le résultat le plus simple allant dans ce sens, dû à P. Lax et S. Alpern\footnote{Pour plus de précisions, voir le début de la partie \ref{secLax}.}, qui affirme que tout homéomorphisme conservatif peut être approché, en topologie uniforme, par une permutation dyadique cyclique (au sens de la définition \ref{PerDy}). Notons qu'un tel résultat nous fait sortir du monde des homéomorphismes puisque les permutations dyadiques sont des applications discontinues ; c'est pourquoi il faudra le combiner à une proposition, dite d'extension des applications finies, qui permet de lisser les permutations en des homéomorphismes. Ainsi, on commencera par \og casser nos homéomorphismes en petits morceaux \fg, en les approchant par des permutations dyadiques, grâce au théorème de Lax, puis on \og recollera les morceaux \fg, à l'aide de la proposition d'extension des applications finies. La clef de cette méthode est que l'on contrôle parfaitement l'approximation dyadique qui approche notre homéomorphisme, et notamment l'orbite des points : les permutations cycliques fournissent des approximations d'homéomorphismes aux orbites $\delta$ denses, pour $\delta$ arbitrairement petit. On conserve ce contrôle de certaines orbites lors du retour dans l'ensemble des homéomorphismes conservatifs effectué par le théorème d'extension des applications finies.

Le théorème de Lax, ainsi que la proposition d'extension des applications finies, seront utilisés tout au long de ce mémoire. Cependant nous en déduisons, dès ce deuxième chapitre, deux résultats de généricité dans $\mathrm{Homeo}(X,\mu)$ : la généricité de la transitivité topologique et celle du mélange topologique faible.

Ce chapitre est très largement inspiré du cours de F. Le Roux \cite{1}, lui-même issu du livre de S. Alpern er V. Prasad \cite{2} ; on y a ajouté un énoncé de généricité du mélange topologique faible (théorème \ref{mélange topo faible}) et un résultat de densité qui servira dans la partie suivante (proposition \ref{densité}).

\section{Théorème de Lax}\label{secLax}

Comme on l'a déjà dit, nous commençons ce chapitre par une présentation du théorème de Lax, qui permet d'approcher tout élément de $\mathrm{Homeo}(X,\mu)$ par une permutation dyadique en topologie uniforme. Sa démonstration, mise au point par P. Lax et S. Alpern, est élémentaire et assez jolie ; elle découle facilement de deux lemmes combinatoires. Historiquement, ce n'est pas le premier résultat d'approximation par des permutations ; on trouvait déjà cette idée dans l'article pionnier de J. Oxtoby et S. Ulam \cite{Oxtoby}, mais aussi, et de manière plus forte, chez A. Katok et A. Stepin \cite{KatokS1} ; la nouveauté introduite par ce théorème réside semble-t-il dans le fait que l'approximation se fait ici en topologie \emph{uniforme}\footnote{Même si cela n'intervient pas de façon centrale dans les applications dans ce théorème. En fait, l'intérêt du théorème de Lax réside en grande partie dans la simplicité et la beauté de sa démonstration.}.

P. Lax a énoncé son théorème en 1971 dans l'optique de l'approximation numérique des homéomorphismes conservatifs \cite{Lax}, mais c'est S. Alpern qui a eu l'idée d'une part de le raffiner en précisant que la permutation peut être choisie cyclique, et d'autre part de l'utiliser pour donner une preuve moderne de la généricité de la transitivité \cite{AlpernThez}, \cite{AlpernOuique}.

Rappelons que l'on s'est fixé une application $\phi : I^n\to X$ donnée par le corollaire \ref{Brown-mesure}. Celle-ci permet de définir une suite $(\D_m)_{m\in\N}$ de subdivisions dyadiques de $(X,\mu)$, ainsi qu'une notion de permutation dyadique sur ces subdivisions. Le théorème de Lax affirme alors que les permutations (qui, rappelons-le, permutent les cubes de la subdivision en agissant comme des \og translations \fg~sur ces cubes) cycliques (tel que l'orbite de tout cube recouvre toute notre variété $X$) approchent uniformément les éléments de $\mathrm{Homeo}(X,\mu)$.

\begin{theoreme}[Lax, Alpern]\label{Lax}
Soient $f\in\mathrm{Homeo}(X,\mu)$ et $\varepsilon>0$. Alors il existe un entier $m$ et une permutation dyadique cyclique d'ordre $m$ notée $f_m$ tels que $d_{\mathit{forte}}(f,f_m)<\varepsilon$.
\end{theoreme}

Pour prouver ce théorème nous aurons besoin du \og lemme de mariage\fg :

\begin{lemme}[Lemme de mariage]\label{mariage}
Soient $E$ et $F$ deux ensembles finis, et $\approx$ une relation entre les éléments de $E$ et de $F$. On suppose que :
\[\forall E'\subset E,\quad\# E'\leq \# \left\{f\in F\mid \exists e \in E' : e\approx f \right\},\]
autrement dit l'ensemble des éléments de $F$ associés à un sous-ensemble $E'$ de $E$ est de cardinal plus grand que celui de $E'$. Alors il existe une application injective $\Phi : E\to F$ telle que pour tout $e\in E$, $e\approx \Phi(e)$.
\end{lemme}

Ainsi que d'un lemme purement combinatoire :

\begin{lemme}[Approximations cycliques dans $\Sn_q$]\label{Pioure}
Soient $q\in\N^*$ et $\sigma \in \mathfrak{S}_q$ (on voit $\Sn_q$ comme le groupe des permutations de $\Z/q\Z$). Alors il existe $\tau \in \mathfrak{S}_q$ telle que $|\tau(k)-k|\leq 2$ pour tout $k$ (où $|.|$ désigne la distance dans $\Z/q\Z$) et telle que la permutation $\tau\sigma$ soit cyclique.
\end{lemme}

\begin{proof}[Preuve du lemme \ref{mariage}] Ce lemme se montre par récurrence sur le cardinal de $E$. La propriété est évidente pour $\# E=1$. Supposons qu'elle soit vraie pour tout ensemble de cardinal $q$, et soit $E$ un ensemble de cardinal $q+1$. On a alors deux cas :
\begin{itemize}
\item[\emph{Premier cas :}] Pour tout sous-ensemble strict $E'$ de $E$,
\[\# E'<\# \left\{f\in F\mid\exists e \in E' : e\approx f \right\}.\]
Choisissons alors un élé\-ment $e_0\in E$ et un élément $f_0\in F$ en relation avec $e_0$. Posons $\Phi(e_0) = f_0$. On peut appliquer l'hypothèse de récurrence à $E\setminus\{e_0\}$ et à $F\setminus\{f_0\}$, ce qui permet de définir $\Phi$ sur $E\setminus \{e_0\}$ ; on a alors défini $\Phi$ sur $E$ tout entier.
\item[\emph{Second cas :}] Il existe un sous-ensemble strict $E'$ de $E$ tel que
\[\# E'\!=\# \!\left\{f\in F\mid\exists e \in E' : e\approx f \right\}.\]
Dans ce cas on peut appliquer l'hypothèse de récurrence aux ensembles $E'$ et $F'= \left\{f\in F\mid \exists e \in E' : e\approx f \right\}$, mais aussi à $E\setminus E'$ et à $F\setminus F'$; en effet, on a pour tout $G$ inclus dans $E\setminus E'$ : 
\begin{eqnarray*}
\# (G\cup E') & \leq & \# \left\{f\in F\mid\exists e \in (G\cup E') : e\approx f \right\}\\
              & \leq & \# \left\{f\in F'\mid\exists e \in (G\cup E') : e\approx f \right\}\\
              &      & + \,\# \left\{f\in (F\setminus F')\mid\exists e \in (G\cup E') : e\approx f \right\},
\end{eqnarray*}
d'où :
\[\# G +\# E' \leq \# F' + \# \left\{f\in (F\setminus F')\mid\exists e \in G : e\approx f \right\},\]
et donc :
\[\# G \leq \# \left\{f\in (F\setminus F')\mid\exists e \in G : e\approx f \right\}.\]
On définit ainsi $\Phi$ sur $E'$ et $E\setminus E'$, donc sur $E$ tout entier.
\end{itemize}
On a bien défini une application $\Phi$ dans les deux cas, ceci termine la récurrence.
\end{proof}

\begin{proof}[Preuve du lemme \ref{Pioure}] Posons $k = \lceil q/2\rceil$. Pour tout $1\leq i\leq k$, on considère la permutation $r_i$ définie par $r_i = (2i\!-\!1\ 2i)$ si $2i\!-\!1$ et $2i$ sont dans deux cycles différents (de longueurs éventuellement 1) de $\sigma$, $r_i = \mathrm{Id}$ sinon. Remarquant que :
\[(2i\!-\!1\ 2i)\,(a_1\cdots a_p\ 2i\!-\!1\ a_{p+2}\cdots a_r)\,(b_1\cdots b_q\ 2i\ b_{q+2}\cdots b_s)=\]
\begin{equation}\label{Formule1}
(a_1\cdots a_p\ 2i\ b_{q+2}\cdots b_s\ b_1\cdots b_q\ 2i\!-\!1\ a_{p+2}\cdots a_r)
\end{equation}
et posant $\tau_1 = r_1\cdots r_k$, on se rend compte que l'on a réuni deux par deux dans $\tau_1\sigma$ tous les cycles contenant $2i\!-\!1$ et $2i$. On recommence l'opération, mais avec cette fois-ci les couples $2i$ et $2i\!+\!1$. On obtient alors une permutation $\tau$ qui vérifie les conclusions du lemme.
\end{proof}

\begin{proof}[Preuve du théorème \ref{Lax}] Soit $f\in\mathrm{Homeo}(X,\mu)$. Considérons un entier $m$ tel que les cubes de la subdivision $\D_m$, ainsi que leurs images\footnote{Cela est possible par uniforme continuité de $f$.} par $f$, aient un diamètre plus petit que $\varepsilon$. Pour tout couple $(C,C')$ de cubes de $\D_m$, on définit la relation $\approx$ entre les éléments de $\D_m$ comme suit : $C\approx C'$ si et seulement si l'image de $C$ par $f$ intersecte $C'$ non trivialement. Puisque $f$ préserve le volume, l'image de l'union de $k$ cubes intersecte au moins $k$ cubes, ce qui fait que l'on se retrouve dans les hypothèses du lemme de mariage : il existe une application injective $\sigma_m$ de $\D_m$ dans $\D_m$ (donc bijective) telle que, posant $f_m$ la permutation dyadique d'ordre $m$ associée à $\sigma_m$ (au sens de la définition \ref{PerDy}), pour tout cube $C$, $f_m (C)$ intersecte $f(C)$ non trivialement. On a alors :
\[d_{\mathit{forte}}(f,f_m) \leq \sup_{C\in\D_m}\big(\mathrm{diam}(C)\big)+\sup_{C\in\D_m}\big(\mathrm{diam}(f(C))\big)\le 2\varepsilon\]
Reste à montrer que l'on peut prendre pour $f_m$ une permutation cyclique. Numérotant les cubes de manière à ce que deux cubes consécutifs soient adjacents, on utilise le lemme \ref{Pioure}, qui nous donne une permutation cyclique $f_m'$ qui est $2\,\sup(\mathrm{diam}(C))$-proche de $f_m$. On a donc trouvé un entier $m$ et une permutation cyclique $f_m'$ de $\D_m$, qui est $4\varepsilon$-proche de $f$ en distance uniforme.
\end{proof}

Le théorème \ref{Lax} nous sert principalement à montrer la généricité des propriétés du type \og quasi-périodicité \fg, comme la transitivité ou bien le non mélange fort ergodique. Il existe une variante du théorème \ref{Lax} dans laquelle la permutation cyclique est remplacée par une \emph{permutation bicyclique} et qui permet d'obtenir une preuve de la généricité du mélange topologique faible. On pourrait sûrement montrer de nombreuses autres variantes du théorème \ref{Lax} où les permutations cycliques seraient remplacées par d'autres types de permutations, mais seule la suivante nous sera utile par la suite.

\begin{definition}[Permutation bicyclique]\label{vélo}
Nous appelerons \emph{permutation bicylique} d'ordre $m$ de $X$ une permutation dyadique d'ordre $m$ ayant exactement deux cycles, et telle que les longueurs de ces deux cycles soient premières entre elles.
\end{definition}

\begin{coro}[Variante du théorème de Lax]\label{Laxbis}
Soit un homéomorphisme $f\in\mathrm{Homeo}(X,\mu)$ et $\varepsilon>0$. Alors il existe un entier $m$ et une permutation dyadique bicyclique $f_m$ telle que $d_{\mathit{forte}}(f,f_m)<\varepsilon$.
\end{coro}

\begin{proof}[Preuve du corollaire \ref{Laxbis}] Soit $f\in\mathrm{Homeo}(X,\mu)$ et $\varepsilon>0$. Le théorème \ref{Lax} donne un entier $m$ et une permutation dyadique cyclique $f_m$ d'ordre $m$ à distance au plus $\varepsilon$ de~$f$.

Montrons tout d'abord par l'absurde qu'il existe deux cubes adjacents $C_1$ et $C_2= f_m^k(C_1)$ de $\D_{m}$ tels que le temps de transition $k$ entre ces deux cubes soit impair. Soit $C_0$ un cube quelconque et posons $C'_0=f_m(C_0)$. Puisque la permutation $f_m$ est d'ordre pair, les entiers $k$ tels que $C'_0 = f_m^k(C_0)$ sont tous impairs. Considérons une suite de cubes adjacents allant de $C_0$ à $C'_0$. Si tous les temps de transition entre les cubes adjacents étaient pairs, alors on trouverait un temps de transition de $C_0$ à $C'_0$ pair, ce qui entrainerait une contradiction. Donc il existe deux cubes adjacents $C_1$ et $C_2$ dont le temps de transition est impair.

Soit $\tau$ la transposition dyadique\footnote{Autrement dit une permutation dyadique qui échange deux cubes de la subdivision, et laisse invariants tous les autres.} qui permute les deux cubes $C_1$ et $C_2$. Posons $f_m' = \tau\circ f_m$, alors $d_{\mathit{forte}}(f_m',f)<2\varepsilon$. De plus, la permutation $f_m'$ se décompose en deux cycles disjoints d'ordres $p$ et $q$ impairs (comme dans la formule \ref{Formule1}). Puisque $p+q$ est une puissance de 2, les entiers $p$ et $q$ sont premiers entre eux.
\end{proof}

\begin{rem}
Comme on l'a dit dans la remarque \ref{Hadi}, les preuves du théorème \ref{Lax} et du corollaire \ref{Laxbis} s'adaptent facilement à des subdivisions $p$-adiques avec $p$ premier.
\end{rem}

\section{Extension des applications finies}

Nous venons d'établir le théorème de Lax (théorème \ref{Lax}), qui nous permet d'approcher un homéomorphisme préservant la mesure par une permutation cyclique d'une subdivision dyadique. Or le but que nous nous sommes fixés est d'obtenir des propriétés sur $\mathrm{Homeo}(X,\mu)$, si bien qu'il faut maintenant trouver un moyen de \og lisser \fg~notre permutation pour en déduire des propriétés de généricité parmi les homéomorphismes. L'idée est alors de ne s'occuper que des images des points voisins des centres des cubes et de laisser libre les points situés près du bord des cubes, de manière à avoir assez de marge de man\oe uvre pour obtenir un homéomorphisme préservant la mesure. On ajoute à cela une condition sur la distance l'extension à l'identité, afin de ne pas perdre la qualité de l'approximation que l'on avait obtenue dans le théorème de Lax.

\begin{prop}[Extension des applications finies]\label{extension} Soient $(x_1,\dots,x_k)$ et $(y_1,\dots,y_k)$ deux $k$-uplets de points deux à deux distincts de $X\setminus \partial X$. Alors il existe $f\in\mathrm{Homeo}(X,\mu)$ préservant l'orientation tel que pour tout $i\in\{1,\dots,k\}$, $f(x_i) = y_i$. De plus $f$ peut être choisi de telle sorte que si $\max \mathrm{dist}(x_i,y_i)< \delta$, alors $d_{\mathit{forte}}(f,\mathrm{Id})<\delta$.
\end{prop}

On trouvait déjà les idées principales de ce théorème chez J. Oxtoby et S. Ulam lors de la démonstration de leur théorème (\cite{Oxtotrans}, \cite{Oxtoby}) ; nous présentons ici la forme moderne de ce théorème due à S. Alpern (voir \cite{2}).

\begin{proof}[Preuve de la proposition \ref{extension}] Pour plus de simplicité on se place en dimension 2, les dimensions supérieures se traitant de la même manière\footnote{Et même plus simplement.}. 

On commence par relier chaque $x_i$ à chaque $y_i$ par une géodésique minimisante $\gamma_i$. Supposons tout d'abord que les $\gamma_i$ soient disjoints. Alors par compacité, il existe $\delta'$ vérifiant $0<\delta'<\frac{\delta-\max_i \mathrm{dist}(x_i,y_i)}{2}$ tel que les $\delta'$-voisinages des $\gamma_i$ soient disjoints. On envoie alors chaque $x_i$ sur le $y_i$ correspondant à l'aide de transformations à support dans le $\delta'$-voisinage de $\gamma_i$ : pour tout $i$, il existe une suite de points 
\[x_i = z_0, z_1, \dots, z_k = y_i,\]
tous situés sur $\gamma_i$, et tels que pour tout $j\in\{1,\dots,k\}$, les points $z_{j-1}$ et $z_j$ soient situés dans un même ouvert de carte de la variété. Pour tout $j$, il existe un homéomorphisme conservatif envoyant $z_{j-1}$ sur $z_j$, à support de diamètre inférieur à $\delta'$. En effet, donnons-nous un homéomorphisme $h_j$ envoyant un voisinage ouvert connexe de $z_{j-1}$ et $z_j$ sur un ouvert connexe de $\R^2$, et la mesure $\mu$ sur la mesure de Lebesgue\footnote{Ceci est possible grâce au théorème \ref{mesures-homéo}.}, ce qui nous permet de nous placer dans un ouvert $\Omega_j$ de $\R^2$, ce que nous supposons désormais. Il existe $\delta''_j$ tel que si $\|x-y\|_{\R^2}\leq \delta''_j$, alors $\mathrm{dist}(h_j^{-1}(x),h_j^{-1}(y))\leq \delta'$. On construit alors une seconde suite de points,
\[z_{j-1} = \tilde{z}_j^0, \tilde{z}_j^1, \dots, \tilde{z}_j^{k_j} = z_j,\]
telle que pour tout $l$, $\tilde{z}_j^{l}$ soit dans $h_j(\gamma_i)$, et telle que la boule de centre le milieu de $\tilde{z}_j^{l-1}$ et $\tilde{z}_j^{l}$ et de diamètre $4 \|\tilde{z}_j^{l-1}-\tilde{z}_j^{l}\|$ soit incluse dans $\Omega_j$, et de diamètre plus petit que $\delta''_j$. 

Alors pour tout $l$, il existe un homéomorphisme conservatif de $\Omega_j$, envoyant $\tilde z_k^{l-1}$ sur $\tilde z_k^l$, et dont le support est de diamètre plus petit que $\delta''_j$ : on peut par exemple consid\'erer l'hom\'eomorphisme $g$ d\'efini en coordonn\'ees polaires par $g(r\,e^{i\theta}) = r\,e^{i(\theta+f(r))}$, avec $f$ affine par morceaux définie par :
\[f(r) = \left\{\begin{array}{ll}
8\pi r & \quad \text{si}\quad 0\le r\le \frac {1}{8}\\
\pi & \quad\text{si}\quad \frac {1}{8}\le r\le \frac {3}{8}\\
4\pi-8\pi r & \quad \text{si}\quad \frac {3}{8}\le r\le \frac {1}{2},
\end{array}\right.\]
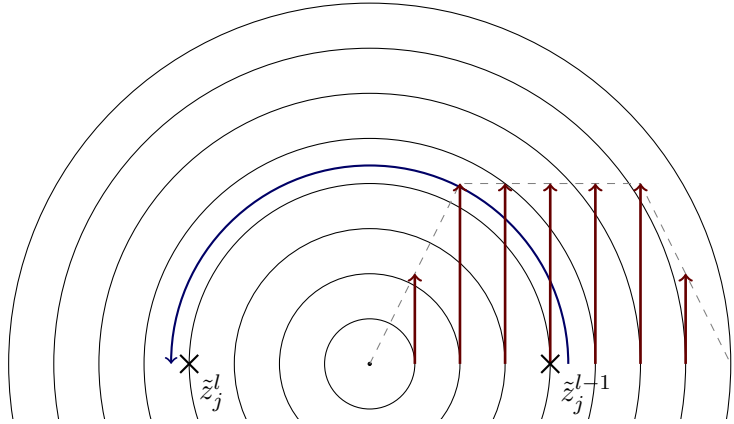
\begin{figure}
\begin{center}
\begin{tikzpicture}[scale=1.2]
\draw[fill=black] (0,0) circle (0.5pt);
\clip (-5,-0.6) rectangle (5,5);
\draw (0.5,0) arc (0:360:0.5);
\draw (1,0) arc (0:360:1);
\draw (1.5,0) arc (0:360:1.5);
\draw (2,0) arc (0:360:2);
\draw (2.5,0) arc (0:360:2.5);
\draw (3,0) arc (0:360:3);
\draw (3.5,0) arc (0:360:3.5);
\draw (4,0) arc (0:360:4);
\draw[gray,dashed] (0,0) -- (1,2) -- (3,2) -- (4,0);
\draw[->][draw=blue!40!black, thick] (2.2,0) arc (0:180:2.2);
\draw[->][line width=1pt,draw=red!40!black] (0.5,0) -- (0.5,1);
\draw[->][line width=1pt,draw=red!40!black] (1,0) -- (1,2);
\draw[->][line width=1pt,draw=red!40!black] (1.5,0) -- (1.5,2);
\draw[->][line width=1pt,draw=red!40!black] (2,0) -- (2,2);
\draw[->][line width=1pt,draw=red!40!black] (2.5,0) -- (2.5,2);
\draw[->][line width=1pt,draw=red!40!black] (3,0) -- (3,2);
\draw[->][line width=1pt,draw=red!40!black] (3.5,0) -- (3.5,1);
\node[below right=0]at (2,0) {$\tilde{z}_j^{l-1}$};
\draw[thick] (1.9,-0.1) -- (2.1,0.1);
\draw[thick] (1.9,0.1) -- (2.1,-0.1);
\node[below right=0] at (-2,0) {$\tilde{z}_j^{l}$};
\draw[thick] (-1.9,-0.1) -- (-2.1,0.1);
\draw[thick] (-1.9,0.1) -- (-2.1,-0.1);
\end{tikzpicture}
\caption{Homéomorphisme conservatif envoyant $\tilde z_k^{l-1}$ sur $\tilde z_k^l$, avec un support de diam\`etre au plus $\delta''$}\label{trajectoire}
\end{center}
\end{figure}
le conjugu\'e de $g$ par la similitude de $\R^2$ envoyant $\tilde z_k^{l-1}$, $\tilde z_k^l$ et le milieu du segment $[\tilde z_k^{l-1},\tilde z_k^l]$ sur respectivement $e^{i0}/4$, $e^{i\pi}/4$ et $0$ convient (voir la figure \ref{trajectoire}). En composant de tels homéomorphismes, on arrive à relier de proche en proche $z_{j-1}$ à $z_j$ \emph{via} les $\tilde{z}_j^{l}$, puis $x_i$ à $y_i$ \emph{via} les $z_j$. L'homéomorphisme obtenu envoie chaque $x_i$ sur $y_i$, et est à une distance de l'identité d'au plus $\delta$.\\

Reste à régler le cas où les chemins se rencontrent. Quitte à les modifier un peu (en gardant des chemins géodésiques par morceaux), on peut supposer que d'une part les intersections entre deux chemins se réduisent à au plus un point, et d'autre part qu'en chaque point d'intersection ne se rencontrent que deux chemins. En minorant les angles que font les chemins entre eux aux voisinages des intersections, ainsi que les distances relatives de ces intersections, et quitte à diminuer $\delta''$, on obtient des chemins dont les $\delta''$-voisinages sont disjoints en dehors des $\delta''/2$-voisinages des points d'intersection. Considérons un point d'intersection $z$ de deux trajectoires $\gamma_i$ et $\gamma_j$ (que l'on peut supposer unique), et quatre points équidistants du cercle de centre $z$ et de rayon $\delta''/4$, notés $\tilde x_i,\tilde x_j,\tilde y_i,\tilde y_j$ ; on choisit les notations de mani\`ere \`a ce que l'ordre des points $\tilde x_i,\tilde x_j,\tilde y_i,\tilde y_j$ sur le cercle soit le m\^eme que celui des points d'intersection du cercle avec l'arc de $\gamma_i$ joignant $x_i$ \`a $z$, l'arc de $\gamma_j$ joignant $x_j$ \`a $z$, l'arc de $\gamma_i$ joignant $z$ \`a $y_i$ et l'arc de $\gamma_j$ joignant $z$ \`a $y_j$. La discussion pr\'ec\'edente nous fournit alors un hom\'eomorphisme conservatif qui envoie $x_i$ sur $\tilde x_i$, $x_j$ sur $\tilde x_j$, $\tilde y_i$ sur $y_i$, et $\tilde y_j$ sur $y_j$. Il reste donc \`a trouver un hom\'eomorphisme qui envoie $\tilde x_i$ sur $\tilde y_i$ et $\tilde x_j$ sur $\tilde y_j$ ; il suffit pour cela de composer par une application du type de $g$ définie ci-dessus, qui envoie bien un point du cercle de rayon $\delta''/4$ sur son symétrique par rapport à $z$ (voir la figure \ref{trajectoirebis}). L'homéomorphisme obtenu par composition vérifie alors les conclusions de la proposition.

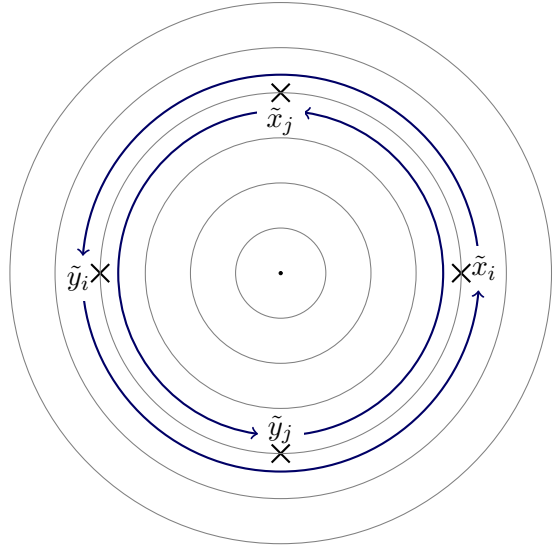
\begin{figure}
\begin{center}
\begin{tikzpicture}[scale=1.2]
\draw[fill=black] (0,0) circle (0.5pt);
\draw[->][draw=blue!40!black, thick] (2.2,0) arc (0:175:2.2);
\draw[->][draw=blue!40!black, thick] (0,1.8) arc (90:262:1.8);
\draw[->][draw=blue!40!black, thick] (0,-1.8) arc (-90:82:1.8);
\draw[->][draw=blue!40!black, thick] (-2.2,0) arc (180:355:2.2);
\node[right, fill=white] at (2,.06) {$\tilde x_i$};
\draw[thick] (1.9,-0.1) -- (2.1,0.1);
\draw[thick] (1.9,0.1) -- (2.1,-0.1);
\node[below, fill=white] at (0,1.95) {$\tilde x_j$};
\draw[thick] (-0.1,1.9) -- (0.1,2.1);
\draw[thick] (0.1,1.9) -- (-0.1,2.1);
\node[left, fill=white] at (-2,-.06) {$\tilde y_i$};
\draw[thick] (-1.9,-0.1) -- (-2.1,0.1);
\draw[thick] (-1.9,0.1) -- (-2.1,-0.1);
\node[above, fill=white] at (0,-2) {$\tilde y_j$};
\draw[thick] (-0.1,-1.9) -- (0.1,-2.1);
\draw[thick] (0.1,-1.9) -- (-0.1,-2.1);\draw[gray] (0.5,0) arc (0:360:0.5);
\draw[gray] (1,0) arc (0:360:1);
\draw[gray] (1.5,0) arc (0:360:1.5);
\draw[gray] (2,0) arc (0:360:2);
\draw[gray] (2.5,0) arc (0:360:2.5);
\draw[gray] (3,0) arc (0:360:3);
\end{tikzpicture}
\caption{Homéomorphisme conservatif envoyant $\tilde x_i$ sur $\tilde y_i$ et $\tilde x_j$ sur $\tilde y_j$}\label{trajectoirebis}
\end{center}
\end{figure}
\end{proof}

\section{Approximations par des homéomorphismes et orbites $\delta$-denses}

La combinaison du théorème de Lax (théorème \ref{Lax}) et de la proposition d'extension des applications finies (proposition \ref{extension}) permet d'approcher tout homéomorphisme par un autre qui va posséder une orbite $\delta$-dense, pour $\delta$ arbitrairement petit (fixé à l'avance). Cette orbite ne sera rien d'autre que la permutation cyclique des centres des cubes d'une subdivision dyadique assez fine. Formalisons cela :

\begin{lemme}\label{lemmetrans}
Pour tout $\delta>0$, l'ensemble des homéomorphismes ayant un point périodique d'orbite $\delta$-dense est dense. On peut de plus supposer que cette orbite est constituée de centres de cubes d'une subdivision dyadique.
\end{lemme}

\begin{proof}[Preuve du lemme \ref{lemmetrans}] Soit $f\in\mathrm{Homeo}(X,\mu)$ et $\varepsilon>0$. Prenons une permutation cyclique d'ordre $m$ donnée par le théorème de Lax, notée $f_m$, avec $m$ assez grand pour que la la taille des cubes soit plus petite que $\delta$, telle que $d_{\mathit{forte}}(f,f_m)< \varepsilon$. Soient $p_1,\dots,p_\ell$ les centres des cubes, ordonnés de telle manière que $f_m(p_k) = p_{k+1}$ pour tout $k$ (pour plus de facilité dans les calculs on indexe les cubes par $\Z/\ell\Z$). On remarque que pour tout $k$, $d_{\mathit{forte}}(f(p_k),p_{k+1})<\varepsilon$. On applique maintenant la proposition d'extension des applications finies aux points $f(p_k)$ et $p_{k+1}$. Cela nous donne un homéomorphisme préservant la mesure $\varphi$ tel que $\varphi(f(p_k)) = p_{k+1}$ et $d_{\mathit{forte}}(\varphi,\mathrm{Id})< \varepsilon$. Posant $f' = \varphi\circ f$, on a $d_{\mathit{forte}}(f,f')\leq \varepsilon$ et $f'(p_k) = p_{k+1}$ pour tout $k$, d'où la conclusion du lemme.
\end{proof}

On a aussi un rafinnement du lemme \ref{lemmetrans} :

\begin{lemme}\label{lem mélange topo faible}
Soient $f\in\mathrm{Homeo}(X,\mu)$, $\varepsilon>0$, $N\in\N$ et $\delta>0$. Alors il existe $f'\in\mathrm{Homeo}(X,\mu)$, avec $d_{\mathit{forte}}(f,f')\le \varepsilon$, des entiers $(q_1,\dots,q_N)$ deux à deux premiers entre eux, et des points $(x_1,\dots,x_N)$ de $X$ tels que chaque $x_i$ soit périodique de période $q_i$ pour $f'$, et que l'orbite de chaque $x_i$ soit $\delta$-dense.
\end{lemme}

\begin{proof}[Preuve du lemme \ref{lem mélange topo faible}] Désignons par $p_i$ le $i$-ème nombre premier. On construit $f'$ par récurrence : ayant construit $f'$ vérifiant la propriété pour un entier $N-1$ et un réel $\varepsilon/2$, on l'approche à une distance au plus $\varepsilon/2$ par une permutation $p_N$-adique cyclique\footnote{Voir la remarque \ref{Hadi}.} dont les cubes sont de taille inférieure à $\delta$, puis on construit $f''\in\mathrm{Homeo}(X,\mu)$, à une distance d'au plus $\varepsilon/2$ de $f'$, qui va permuter les centres de la subdivision choisie (et donc on aura un point périodique de période $p_N^k,\,k\in\N$) et fixer les orbites des points $(x_1,\dots,x_{N-1})$. L'homéomorphisme $f''$ vérifie bien les conclusions du lemme.
\end{proof}

Terminons cette partie par un théorème de densité, dont la preuve est de la même veine que celle du lemme \ref{lemmetrans}. Il sera le point de départ de la preuve de la généricité des homéomorphismes pour lesquels les points périodiques sont denses. 

\begin{prop}\label{densité}
L'ensemble des homéomorphismes ayant un ensemble dense de points p\'eriodiques est dense.
\end{prop}

\begin{proof}[Preuve de la proposition \ref{densité}] Soit $\varepsilon>0$ et $f\in \mathrm{Homeo}(X,\mu)$. L'idée est de trouver une suite $(m_k)_k$ d'entiers naturels, à laquelle est associée une suite d'homéomorphismes $(f_k)_k$ qui est de Cauchy pour la métrique
\[\delta(f,g)= d_{\mathit{forte}}(f,g) + d_{\mathit{forte}}(f^{-1}, g^{-1}),\]
et qui converge donc pour $\delta$ vers un homéomorphisme $f_\infty \in\mathrm{Homeo}(X,\mu)$ (la métrique $\delta$ est complète, voir le lemme \ref{complet}); et telles que l'on ait les propriétés suivantes :
\begin{itemize}
\item $d_{\mathit{forte}}(f,f_\infty)<\varepsilon$,
\item $f_k$ permute cycliquement les centres des cubes de $\D_{m_k}$,
\item $f_{k+1}$ et $f_k$ coïncident sur les centres des cubes de $\D_{m_k}$. 
\end{itemize}
Autrement dit on grossit à chaque fois l'ensemble des points périodiques de notre application de départ par l'ajout d'un ensemble de plus en plus dense. Ainsi $f_\infty$ aura un ensemble dense de points périodiques, constitué au moins de centres de cubes de subdivisions dyadiques d'ordres arbitrairement grands.

Construisons maintenant la suite $(f_k)_k$ par récurrence. On pose $f_0= f$, muni de la subdivision triviale contenant uniquement $X$.

Supposons que l'on ait construit l'homéomorphisme $f_{k-1}$. Le théorème \ref{Lax} nous fournit une subdivision $\D_{m_k}$ plus fine que $D_{m_{k-1}}$, et une permutation dyadique $g_k$ de $\D_{m_k}$ telle que $d_{\mathit{forte}}(f_{k-1},g_k)<\varepsilon/2^k$. Soient $p_k^i$ les centres de la subdivision associée à $g_k$ ordonnés de telle manière que pour tout $i$, $g_k(p_k^i) = p_k^{i+1}$. Ainsi on a $d_{\mathit{forte}}(f_{k-1}(p_k^i),p_k^{i+1})<\varepsilon/2^k$. Le théorème d'extension des applications finies nous donne alors une application $\varphi\in\mathrm{Homeo}(X,\mu)$ vérifiant :
\begin{itemize}
\item pour tout $i$, $\varphi(f_{k-1}(p_k^i)) = p_k^{i+1}$,
\item pour tout $j$ et tout $k'<k$, $\varphi(p_{k'}^j) = p_{k'}^j$,
\item $d_{\mathit{forte}}(\varphi,\mathrm{Id})<\varepsilon/2^k$.
\end{itemize}
En effet, il suffit d'appliquer le théorème aux familles $\{f_{k-1}(p_k^i),p_{k'}^j\}_{i,j}$ et $\{p_k^{i+1},p_{k'}^j\}_{i,j}$, disjointes d'une part parce que les deux familles $\{p_k^i\}_i$ et $\{p_{k'}^j\}_j$ sont disjointes, par définition de la suite de subdivisions dyadiques $\{D_m\}_{m\ge 0}$, et d'autre part par bijectivité de $f_{k-1}$, qui implique que l'image d'une famille disjointe est disjointe.

On pose maintenant $f_k= \varphi\circ f_{k-1}$. On a alors :
\begin{itemize}
\item $f_k(p_k^i) = p_k^{i+1}$,
\item $f_k(p_{k-1}^j) = p_{k-1}^j$,
\item $d_{\mathit{forte}}(f_{k-1},f_k)<\varepsilon/2^k$.
\end{itemize}

Reste à montrer que la dernière des inégalités s'applique aussi à l'inverse des fonctions. Or on sait que $d_{\mathit{forte}}(f_{k-1}^{-1},f_k^{-1}) = d_{\mathit{forte}}(f_{k-1}^{-1},f_{k-1}^{-1}\circ\varphi^{-1})$. Il suffit donc de rendre la distance $d_{\mathit{forte}}(\mathrm{Id},\varphi^{-1})$, par ailleurs égale à $d_{\mathit{forte}}(\mathrm{Id},\varphi)$, plus petite que le module de continuité de $f_{k-1}^{-1}$ associé à la quantité $\varepsilon/2^k$. Mais quitte à prendre $m$ suffisament grand, cette majoration est vérifiée. On a donc obtenu l'homéomorphisme suivant désiré ; la propriété est prouvée.
\end{proof}

\section{Généricité de la transitivité}

Les généricités de la transitivité et du mélange faible découlent facilement des lemmes de la section précédente. La première preuve de la généricité de la transitivité dans $\mathrm{Homeo}(X,\mu)$ est due à J. Oxtoby \cite{Oxtotrans} et suivait à peu près le schéma de la preuve actuelle ; mais c'est S. Alpern qui a clarifié les outils auquels on y fait appel, en adaptant le théorème de Lax \cite{AlpernThez}, \cite{AlpernOuique}.

\begin{theoreme}[Oxtoby]\label{transitivité}
Les homéomorphismes topologiquement transitifs forment un $G_\delta$ dense de l'espace $\mathrm{Homeo}(X,\mu)$.
\end{theoreme}

\begin{proof}[Preuve du théorème \ref{transitivité}] Soit $(C_i)_{i\ge 0}$ la suite des cubes ouverts de la réunion sur $m\in\N$ de toutes les subdivisions dyadiques $\D_m$, rangés par ordre décroissant de taille. La famille $(C_i)$ forme alors une base de la topologie de $X$. Posons, pour $i,j\in \N$,
\[T_{i,j} = \left\{ f\in\mathrm{Homeo}(X,\mu)\mid\exists m>0 : f^m(C_i)\cap C_j\neq \emptyset \right\}.\]
Ces ensembles sont ouverts, et leur intersection est égale à l'ensemble des éléments topologiquement transitifs. Par le théorème de Baire il suffit de montrer que chacun d'eux est dense, mais cela découle immédiatement du lemme~\ref{lemmetrans}. 
\end{proof}

\indent Comme on l'a déjà dit, la technique précédente permet de montrer des résultats un peu plus forts, comme par exemple la généricité du mélange topologique faible :

\begin{definition}
Un homéomorphisme $f$ est dit \emph{topologiquement faiblement mélangeant} si quels que soient les ouverts non vides $U_1,\,V_1,\,U_2$ et $V_2$, il existe un entier $m$ tel que $f^m(U_1)\cap V_1$ et $f^m(U_2)\cap V_2$ soient non vides.
\end{definition}

\begin{theoreme}\label{mélange topo faible}
L'ensemble des homéomorphismes topologiquement faiblement mélangeants est gras dans $\mathrm{Homeo}(X,\mu)$.
\end{theoreme}

Ce théorème découle facilement du lemme \ref{lem mélange topo faible} :

\begin{proof}[Preuve du théorème \ref{mélange topo faible}] Soit $(C_i)_{i\ge 0}$ la suite des cubes ouverts des subdivisions dyadiques de tous ordres déjà considérée. On écrit l'ensemble des homéomorphismes faiblement mélangeants comme l'intersection des $T'_{i,j,k,l}$, où 
\[T'_{i,j,k,l} = \left\{f\in\mathrm{Homeo}(X,\mu)\mid \exists m>0 : f^m(C_i)\cap C_j\neq \emptyset \text{ et } f^m(C_k)\cap C_l\neq \emptyset\right\}.\]
Les ensembles $T'_{i,j,k,l}$ sont ouverts (c'est évident par continuité de $f^m$). Pour terminer la preuve, il suffit de montrer qu'ils sont denses. Fixons les entiers $i,j,k$ et $l$. On se donne $\varepsilon>0$, et $\delta$ plus petit que la moitié du diamètre des cubes de la subdivision à laquelle appartiennent les cubes $C_i$, $C_j$, $C_k$ et $C_l$. On applique alors le lemme \ref{lem mélange topo faible} pour $N=2$. On obtient $f'\in \mathrm{Homeo}(X,\mu)$, $\varepsilon$-proche de $f$, et deux points $x$ et $y$, respectivement $p$ et $q$ périodiques pour $f'$, avec $p$ et $q$ premiers entre eux, dont les orbites sont $\delta$-denses. Par conséquent, il existe $a$, $b$, $a'$ et $b'$ tels que $f'^a(x)\in C_i$, $f'^{a'}(x)\in C_j$, $f'^b(y)\in C_k$ et $f'^{b'}(y)\in C_l$. Or on cherche une puissance de $f'$ qui envoie $f'^a(x)$ dans $C_j$ et $f'^b(y)$ dans $C_l$, il suffit donc de trouver $m$ positif tel que $m+a = a'+\alpha p$ et $m+b = b'+\beta q$. Or, par le théorème de Bézout, il existe bien $\alpha$ et $\beta$ tels que $\alpha p-\beta q = (a-a')-(b-b')$, on pose alors $m = a'-a+\alpha p$ et, quitte à ajouter un multiple de $pq$, $m$ peut être supposé positif. Ainsi $f'^m(C_i)\cap C_j\neq \emptyset$ et $f'^m(C_k)\cap C_l\neq \emptyset$, et donc $T'_{i,j,k,l}$ est dense.
\end{proof}

\chapter{Modifications locales}

Après la proposition d'extension des applications finies, nous présentons ici un autre moyen de construire (plus ou moins) explicitement des homéomorphismes conservatifs. L'idée, finalement assez naturelle, est de remplacer localement, au voisinage d'un point périodique, un homéomorphisme par un autre. On choisit ce dernier de telle manière qu'il possède la propriété dynamique dont on veut obtenir la généricité. Nous appelons \og théorème de modification locale \fg (théorème \ref{extension-sphères}) le résultat technique qui permet d'effectuer cette opération.

Cette technique, basée sur des outils purement topologiques --- on utilisera d'ailleurs un certain nombre de résultats géométriquement \og évidents \fg~---, nous donne ici trois résultats de généricité : la généricité de l'ensemble des homéomorphismes ayant un ensemble dense de points périodiques, établie par F. Daalderop et R. Fokkink en 2000 \cite{6}, celle de l'entropie topologique infinie, ainsi que celle du mélange topologique fort. Il semble que la preuve de cette dernière n'avait jamais été faite.

Comme au chapitre précédent, on se fixe une application $\phi : I^n\to X$ donnée par le corollaire \ref{Brown-mesure}, ce qui permet de choisir une suite $(\D_m)_{m\in\N}$ de subdivisions dyadiques de $(X,\mu)$, et de définir une notion de permutation dyadique sur ces subdivisions.

\section{Modification locale d'un homéomorphisme con\-ser\-va\-tif}

Donnons-nous une petite boule $U$ dans une variété compacte $X$, une boule $V$ incluse dans $U$, un homéomorphisme $f$ préservant la mesure dans cette variété et un homéomorphisme $g$ de $V$ sur $f(V)$ préservant la mesure, orienté de la même manière que $f$. Le bon sens voudrait que l'on puisse construire un homéomorphisme préservant la mesure qui soit égal à $g$ sur $V$ et à $f$ en dehors de $U$. En effet, le cas de la dimension deux se fait aisément : il est très facile de construire un homéomorphisme vérifiant les propriétés désirées, sauf peut-être la préservation de la mesure. On obtient cette dernière propriété par le théorème des mesures homéomorphes (théorème \ref{mesures-homéo}). En dimensions supérieures, des problèmes délicats de topologie surgissent ; mais ceux-ci sont résolus par l'\emph{annulus theorem}, qui affirme que la région de $\R^n$ située entre deux  sphères \emph{épaississables} est homéomorphe à une anneau. C'est un résultat géométriquement évident mais néamoins difficile, dont la preuve diffère selon la dimension de l'espace considéré ; il a été démontré en dimension 2 par T. Rad\'o en 1924 \cite{Rado}, en dimension 3 par E. Moise en 1952 \cite{Moise}, en dimension plus grande que 5 par R. Kirby en 1969 \cite{Cello} et en dimension 4 par F. Quinn en 1982 \cite{Quinn}.

\begin{definition}
Un plongement $i$ d'une variété $\Sigma$ dans une variété $M$ est dit \emph{épaississable} (en anglais \emph{bicollared}) s'il existe un prolongement de $j :\Sigma\times [-1,1] \to M$ tel que $j_{\Sigma\times \{0\}} = i$.
\end{definition}

Cette restriction aux sous-variétés épaississables permet de se défaire des pathologies telles que les sphères cornues. Admettant l'\emph{annulus theorem}, on obtient sans grande difficulté le théorème suivant :

\begin{theoreme}[Modification locale]\label{extension-sphères}
Soient $\sigma_1$, $\sigma_2$, $\tau_1$ et $\tau_2$ quatre plongements épaississables de $\mathbf{S}^{n-1}$ dans $\R^n$, tels que $\sigma_1$  soit dans la composante connexe bornée\footnote{Par théorème de Jordan-Brouwer, le complémentaire d'un ensemble homéomorphe à $\mathbf{S}^{n-1}$ possède bien deux composantes connexes : une bornée et une non.} de $\sigma_2$ et $\tau_1$ dans la composante connexe bornée de $\tau_2$. Soient $A_1$ la composante connexe bornée de $\R^n\setminus\sigma_1$ et $B_1$ la composante connexe bornée de $\R^n\setminus\tau_1$ ; $\Sigma$ la composante connexe de $\R^n\setminus(\sigma_1\cup\sigma_2)$ ayant $\sigma_1 \cup \sigma_2$ pour frontière et $\Lambda$ la composante connexe de $\R^n\setminus(\tau_1\cup\tau_2)$ ayant $\tau_1 \cup \tau_2$ pour frontière ; $A_2$ la composante connexe non bornée de $\R^n\setminus\sigma_2$ et $B_2$ la composante connexe non bornée de $\R^n\setminus\tau_2$ (voir la figure \ref{dessin-extension}).

Supposons que $\mathrm{Leb}(A_1) = \mathrm{Leb}(B_1)$ et $\mathrm{Leb}(\Sigma) = \mathrm{Leb}(\Lambda)$ et donnons-nous deux homéomorphismes préservant la mesure $f_i : A_i\to B_i$, tels que soit chacun préserve l'orientation de $\R^n$, soit chacun la renverse. Alors il existe un homéomorphisme préservant la mesure $f : \R^n \to \R^n$ égal à $f_1$ sur $A_1$ et à $f_2$ sur $A_2$.
\end{theoreme}

\begin{figure}
\begin{center}
\begin{tikzpicture}[scale=1.1]
\draw[color=blue!30!black, fill=gray!8!white, semithick] plot[tension=0.6, smooth cycle] coordinates{(-1.4,1.3) (0,1.6) (1.5,1.7) (1.4,-.6) (1.6,-2) (0.3,-1.4) (-1.5,-1.6) (-1.7,-.8) (-1.6,.5)};
\draw[color=green!30!black, fill=gray!20!white, semithick] plot[tension=0.6, smooth cycle] coordinates{(-1.1,1) (0,1.1) (.8,.6) (.8,-.2) (1,-1) (0.1,-.8) (-.9,-1) (-.8,-.2) (-.9,.1)};
\draw[color=blue!30!black, fill=gray!8!white, semithick] plot[tension=0.6, smooth cycle] coordinates{(5.4,1.3) (6,1.6) (7,1.1) (8.5,1.7) (8.6,-.2) (8.1,-1.5) (7.3,-1.8) (5.5,-1.6) (5.4,-.8) (5.6,.5)};
\draw[color=green!30!black, fill=gray!20!white, semithick] plot[tension=0.6, smooth cycle] coordinates{((6.1,1) (6.8,.9) (7.1,.6) (7.8,1) (8,.9) (7.9,0.5) (7.8,-.6) (6.9,-1) (6.5,-1.1) (6,0)};
\draw[->, semithick] (.6,-.5) .. controls (2.2,-1.5) and (4,-1.5) .. (6.4,-.5);
\draw[->, semithick] (2.2,.5) .. controls (3,1) and (4,1) .. (5.1,.5);
\node at (0,0) {$A_1$};
\node[color=green!30!black] at (-1.05,0) {$\sigma_1$};
\node[color=blue!30!black] at (-1,-1.73) {$\sigma_2$};
\node at (1,1.2) {$\Sigma$};
\node at (0,2.1) {$A_2$};
\node at (3.6,-1.5) {$f_1$};
\node at (3.6,1.1) {$f_2$};
\node at (7,0) {$B_1$};
\node[color=green!30!black] at (8.1,0) {$\tau_1$};
\node[color=blue!30!black] at (8.3,-1.6) {$\tau_2$};
\node at (7,1.7) {$B_2$};
\node at (8.3,1.2) {$\Lambda$};
\end{tikzpicture}\caption{Cadre de la technique de la modification locale}\label{dessin-extension}
\end{center}
\end{figure}
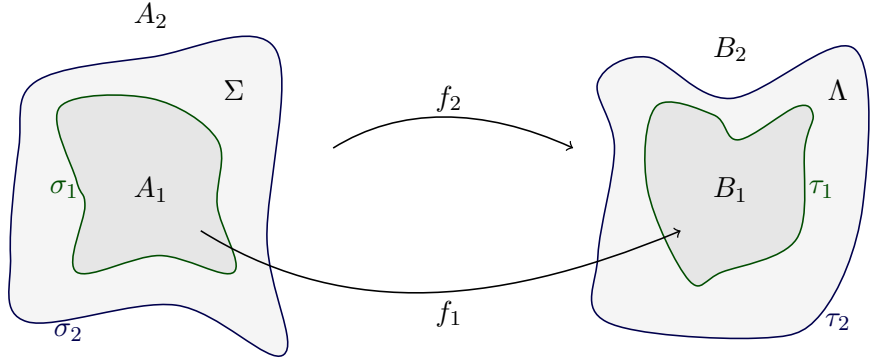

Une preuve de ce théorème, ainsi qu'un énoncé de l'\emph{annulus theorem}, pourront être trouvés dans \cite{6}.

\begin{rem}
Puisque cette propriété ne sera utilisée par la suite que localement, on pourra l'appliquer (à l'aide du théorème \ref{mesures-homéo}) à notre variété $X$ munie de la mesure $\mu$.
\end{rem}

\begin{app}\label{p15}
Nous pouvons d'ores et déjà en déduire quelques résultats sim\-ples de densité dans $\mathrm{Homeo}(X,\mu)$. Par exemple on obtient assez facilement la densité des homéomorphismes dont une puissance coïncide avec l'identité sur un ouvert non vide, ce qui implique la densité des négations des propriétés suivantes : transitivité, mélange topologique, ergodicité, mélange faible ergodique, apériodicité etc. Cela rejoint la remarque faite dans l'introduction affirmant que le contraire de propriétés génériques peut être dense.
\end{app}

\begin{proof}[Preuve de l'application \ref{p15}] Soit $f\in \mathrm{Homeo}(X,\mu)$. Par densité dans l'espace $\mathrm{Homeo}(X,\mu)$ des homéomorphismes ayant un point périodique (lemme \ref{lemmetrans}), on peut supposer que $f$ possède un point périodique $x$ de période $p$. L'idée est de modifier $f$ au voisinage de $f^{p-1}(x)$.

Soit $\varepsilon>0$ et considérons $U$ et $V$ deux boules centrées en $x$ telles que $U\subset V$. En utilisant l'uniforme continuité de $f$, on peut supposer, quitte à rétrécir $V$, puis $U$, que :
\begin{itemize}
\item $f^{p-1}(U)$ et $f^{-1}(V)$ sont dans des ouverts de cartes de la variété,
\item $V,f(V),\dots,f^{p-1}(V)$ sont deux à deux disjoints,
\item le diamètre de $V$ est strictement inférieur à $\varepsilon$,
\item $\overline{f^{p-1}(U)}$ est inclus dans $f^{-1}(V)$.
\end{itemize}
Posons $\varphi = \sigma\circ f^{-(p-1)}$, avec $\sigma$ étant soit l'identité, soit le renversement de la première coordonnée (dans une carte locale), de manière à ce que $\varphi$ préserve (respectivement renverse) l'orientation si $f$ préserve (respectivement renverse) l'orientation. Il ne reste plus qu'à appliquer le théorème \ref{extension-sphères} à : 
\[\begin{array}{rlrl}
A_1 & = f^{p-1}(U)\quad         & A_2 & = f^{-1}(V^\complement)\\
B_1 & = U \quad                 & B_2 & = V^\complement\\
f_1 & =\varphi\quad             & f_2 & = f
\end{array}\]
ce qui nous fournit un homéomorphisme préservant la mesure $g$ tel que $g_{|V^\complement} = f$ et $g_{|U} = \varphi$. Ainsi on a :
\[d_{\mathit{forte}}(f,g)\le \mathrm{diam}(V)\le\varepsilon\]
et 
\[g^p (U) = g(f^{p-1}U) = U.\]
En considérant $g^{2p}$ et $U$ (car sur $U$, $g^p$ est égal à $\sigma$, avec $\sigma^2 = \mathrm{Id}$), on obtient les conclusions de l'application.
\end{proof}

\section[Généricité et points périodiques]{G\'en\'ericit\'e de l'ensemble des homéomorphismes ayant un ensemble dense de points p\'eriodiques}\label{2.2}

Le but de cette section est de prouver le théorème suivant, dû à F. Daalderop et R. Fokkink \cite{6} :

\begin{theoreme}[Daalderop, Fokkink]\label{per-resid}
L'ensemble des homéomorphismes ayant un ensemble dense de points p\'eriodiques est résiduel.
\end{theoreme}

La proposition \ref{densité} montre que l'ensemble des homéomorphismes ayant un ensemble dense de points périodiques forme un ensemble dense, il faut maintenant exprimer cet ensemble comme un $G_\delta$. C'est là qu'intervient la technique de la modification locale, qui permet de créer des points périodiques \emph{persistants} à partir de points périodiques.

\begin{definition}
Soit $f\in\mathrm{Homeo}(X,\mu)$. Un point périodique $p$ de $f$ de période $k$ est dit \emph{persistant} si pour tout voisinage $U$ de $p$, il existe un voisinage $\V$ de $f$ dans $\mathrm{Homeo}(X,\mu)$ tel que tout $\widetilde{f}\in \V$ ait un point périodique $\widetilde{p}\in U$ de période $k$.
\end{definition}

Le c\oe ur de la preuve de généricité se situe dans la densité des points persistants, obtenue par la technique de modification locale :

\begin{prop}\label{persistant}
Soit $f\in\mathrm{Homeo}(X,\mu)$ et $x$ un point périodique de période~$p$ pour~$f$. Alors pour tout $\varepsilon>0$, il existe $g\in\mathrm{Homeo}(X,\mu)$ à distance au plus $\varepsilon$ de~$f$, tel que $x$ soit un point persistant de période $p$ pour $g$.  
\end{prop}

Donnons tout d'abord un exemple d'homéomorphisme préservant la mesure et ayant un point fixe persistant :

\begin{lemme}\label{linéaire}
Toute application linéaire de $\R^n$
\[h = \left(
\begin{array}{ccc}
\lambda_1 &        & (0)      \\
          & \ddots &          \\
(0)       &        & \lambda_n
\end{array}\right),\]
avec $\prod \lambda_i = 1$ et $\lambda_i\neq 1$ pour tout $i$, préserve la mesure et possède un point fixe persistant à l'origine.
\end{lemme}

On trouvera une preuve de ce lemme à la page 319 de \cite{ency}. À l'aide de cet exemple, combiné au théorème \ref{extension-sphères}, nous allons pouvoir remplacer localement un point périodique par un point persistant.

\begin{proof}[Preuve de la proposition \ref{persistant}] Notons $p$ la période de $x$. Soient $U\subset V$ deux boules centrées en $x$. En utilisant l'uniforme continuité de $f$, on peut supposer, quitte à rétrécir $V$, puis $U$, que :
\begin{itemize}
\item $V$ et $f^{-1}V$ sont dans des ouverts de cartes de la variété,
\item $V,f(V)\dots,f^{p-1}(V)$ sont deux à deux disjoints,
\item le diamètre de $V$ est strictement inférieur à $\varepsilon$,
\item $\overline{f^p(U)}$ est inclus dans $V$.
\end{itemize}
Le lemme précédent (lemme \ref{linéaire}) nous donne un homéomorphisme $h : U \to h(U)$ pour lequel $x$ est un point fixe persistant. Posons $\varphi = \sigma\circ h\circ f^{-(p-1)}$, avec $\sigma$ étant soit l'identité, soit le renversement de la première coordonnée, de manière à ce que $\varphi$ et $f^p$ préservent ou renversent simultanément l'orientation. On suppose, quitte à rétrécir de nouveau $U$, que $\overline{\sigma\circ h(U)}$ est contenu dans $V$. On applique le théorème~\ref{extension-sphères} à :
\[\begin{array}{rlrl}
A_1 & = f^{p-1}(U)\quad              & A_2 & = f^{-1}(V)^\complement\\
B_1 & = h(U) \quad                   & B_2 & = V^\complement\\
f_1 & =\varphi\quad                  & f_2 & = f,
\end{array}\]
ce qui nous donne un homéomorphisme $g$ préservant la mesure vérifiant les conclusions de la proposition.
\end{proof}

Nous sommes maintenant en mesure de montrer le théorème \ref{per-resid} :

\begin{proof}[Preuve du théorème \ref{per-resid}] Considérons une base de voisinages ouverts $\{U_k\}_{k\ge 0}$ de $X$ . Pour tout $k$, soit $\mathit{Per}_k$ l'ensemble des homéomorphismes ayant un point périodique dans $U_k$ et $P_k$ son sous-ensemble constitué des applications pour lesquelles ce point est persistant. Par définition, tout élément $f$ de $P_k$ possède un voisinage ouvert dans $\mathrm{Homeo}(X,\mu)$ constitué  uniquement d'éléments de $\mathit{Per}_k$. L'union de tous ces voisinages pour $f$ parcourant $P_k$ forme un ensemble ouvert $\V_k$ vérifiant $P_k\subset\V_k\subset \mathit{Per}_k$. Mais le théorème précédent affirme que $P_k$ est dense dans $\mathit{Per}_k$ et la proposition \ref{densité} dit que $\mathit{Per}_k$ est dense dans $\mathrm{Homeo}(X,\mu)$. Par conséquent $\V_k$ est un ouvert dense de $\mathrm{Homeo}(X,\mu)$. Le théorème de Baire affirme alors que $\bigcap_k \mathit{Per}_k$, qui est exactement l'ensemble des applications ayant un ensemble dense de points p\'eriodiques, contient un $G_\delta$ dense.
\end{proof}

On déduit facilement de ce théorème et du théorème \ref{mélange topo faible} que le \emph{chaos maximal}, tel que défini par Devaney (voir la définition 8.5 page 50 de \cite{Dev}), est générique. Pour plus de précisions, voir la page 27 de \cite{2}.

%

\section{Généricité de l'entropie topologique infinie}

Dans cette partie, nous allons montrer qu'un homéomorphisme conservatif générique a une entropie topologique infinie. En particulier, cela montre qu'un homéomorphisme générique n'est pas topologiquement conjugué à un difféomorphisme. Pour ce faire, nous allons utiliser les notions classiques de rectangle, d'intersection markovienne et de fer à cheval de Smale. Notons que K. Yano avait déjà montré en 1980, à l'aide de  \emph{pseudo fers à cheval}, qu'un homéomoprhisme\footnote{Sans hypothèse de préservation de la mesure.} générique possède une entropie topologique infinie \cite{Yano}.

Commençons par rappeler la définition de l'entropie topologique d'un homéomorphisme d'un espace compact.

\begin{definition}\label{2.9}
Soit $X$ un espace métrique et $f$ un homéomorphisme de $X$. Pour tout entier $m$ non nul, définissons une \emph{distance dynamique} $d_m$ par :
\[d_m(x,y) = \max_{i\le m}\, \mathrm{dist}(f^i(x),f^i(y)).\]
Pour $\varepsilon>0$ donné, notons $C(m,\varepsilon)$ le cardinal maximal d'un sous-ensemble de $X$ dont les points sont à distance au moins $\varepsilon$ les uns des autres pour $d_m$. L'\emph{entropie topologique} de $f$ est alors définie par :
\[h_{top}(f) = \lim_{\varepsilon\to 0}\, \underset{m\to+\infty}{\overline{\lim}} \,\frac{1}{m}\log C(m,\varepsilon).\]
Dans le cas où $X$ est compact, celle-ci est indépendante de la distance définissant la topologie.
\end{definition}

L'entropie peut être vue comme une manière de mesurer le désordre induit par les itérations de notre homéomorphisme.

\begin{rem}
Pour $f\in\mathrm{Homeo}(X,\mu)$, l'inégalité suivante découle facilement de la définition : $m .h_{top}(f)\ge h_{top}(f^m)$ ; c'est même une égalité (voir \cite{ency}, mais nous n'en aurons pas besoin).
\end{rem}

\begin{definition}
Nous appelons un \emph{rectangle} l'image du cube $I^n$ par un homéomorphisme $f : I^n\to f(I^n)\subset X$. Les images par $f$ des faces de $I^n$ sont appelées les \emph{faces} du rectangle. Les faces $f(I^{n-1}\times\{0\})$ et $f(I^{n-1}\times\{1\})$ sont dites \emph{horizontales}, les autres faces sont dites \emph{verticales}. Un sous-ensemble $R_2$ d'un rectangle $R_1$ est appelé \emph{sous-rectangle horizontal strict} (respectivement \emph{sous-rectangle vertical strict}) si c'est un rectangle disjoint des faces horizontales (respectivement verticales) de $R_1$ dont toutes les faces verticales (respectivement horizontales) sont incluses dans celles de $R_1$.
\end{definition}

\begin{rem}
Soit $R_h$ un sous-rectangle horizontal strict de $R$. Alors tout sous-rectangle horizontal strict de $R_h$ est lui-même un sous-rectangle horizontal strict de $R$. De même pour $R_v$ un sous-rectangle vertical strict de $R$, tout sous-rectangle vertical strict de $R_v$ est lui-même un sous-rectangle vertical strict de $R$.
\end{rem}

Les trois lemmes suivants, géométriquement \og évidents \fg, vont préciser la notion de rectangle. Pour les preuves, voir par exemple \cite{Yang}.

\begin{lemme}\label{rectangles}
Soit $R$ un rectangle de $X$, $R_h$ un sous-rectangle horizontal strict de $R$ et $R_v$ un sous-rectangle vertical strict de $R$. Alors $R_h \cap R_v$ est non-vide ; de plus l'une des composantes connexes de $R_h \cap R_v$ contient à la fois un sous-rectangle vertical strict de $R_h$ et un sous-rectangle horizontal strict de $R_v$.
\end{lemme}

\begin{definition}
Soit $f$ un homéomorphisme de $X$, $R_1$ et $R_2$ deux rectangles dans $X$ et $V$ une composante connexe de $f(R_1)\cap R_2$. On dit que $V$ est une \emph{composante markovienne} si c'est à la fois un sous-rectangle vertical strict de $R_2$ et l'image par $f$ d'un sous-rectangle horizontal strict de $R_1$. On dit que l'intersection $f(R_1)\cap R_2$ est \emph{markovienne} si au moins l'une de ses composantes connexes est markovienne.
\end{definition}

\begin{lemme}\label{markov-récur}
Soient $f$ et $g$ deux homéomorphismes de $X$ et $R_1$, $R_2$ et $R_3$ trois rectangles de $X$. Supposons que les intersections $f(R_1)\cap R_2$ et $g(R_2)\cap R_3$ aient respectivement $k$ et $\ell$ composantes markoviennes. Alors l'intersection $(g\circ f)(R_1)\cap R_3$ possède au moins $k\ell$ composantes markoviennes (voir la figure \ref{dessin-markov}).
\end{lemme}

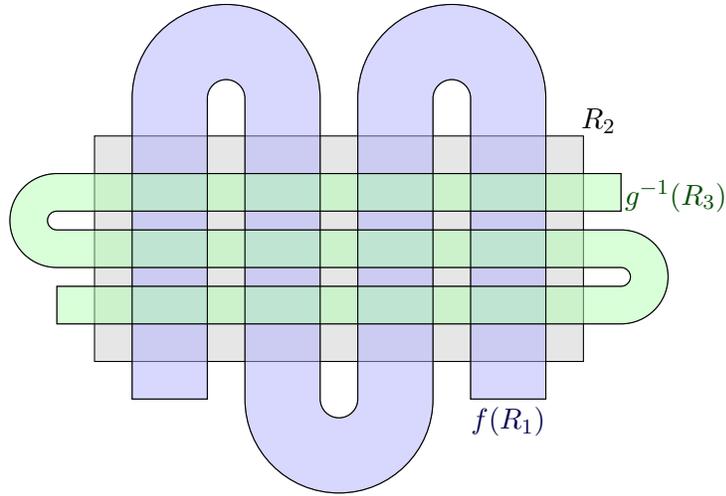
\begin{figure}
\begin{center}
\begin{tikzpicture}
\draw[fill=gray!20!white] (-3,-1.5) rectangle (3.5,1.5);
\draw[fill=blue!30!white, opacity = 0.5] (-2.5,-2) -- (-2.5,2) arc (180:0:1.25) -- (0,-2) arc (180:360:0.25) -- (0.5,2) arc (180:0:1.25) -- (3,2) -- (3,-2) -- (2,-2) -- (2,2) arc (0:180:0.25) -- (1.5,-2) arc (360:180:1.25) -- (-1,2) arc (0:180:0.25) -- (-1.5,-2) -- cycle;
\draw[fill=green!30!white, opacity=0.5] (-3.5,-1) -- (4,-1) arc (-90:90:0.625) -- (-3.5,.25) arc (270:90:0.125) -- (4,.5) -- (4,1) -- (-3.5,1) arc (90:270:0.625) -- (4,-.25) arc (90:-90:0.125) -- (-3.5,-.5) -- cycle;
\draw (-2.5,-2) -- (-2.5,2) arc (180:0:1.25) -- (0,-2) arc (180:360:0.25) -- (0.5,2) arc (180:0:1.25) -- (3,2) -- (3,-2) -- (2,-2) -- (2,2) arc (0:180:0.25) -- (1.5,-2) arc (360:180:1.25) -- (-1,2) arc (0:180:0.25) -- (-1.5,-2) -- cycle;
\draw (-3.5,-1) -- (4,-1) arc (-90:90:0.625) -- (-3.5,.25) arc (270:90:0.125) -- (4,.5) -- (4,1) -- (-3.5,1) arc (90:270:0.625) -- (4,-.25) arc (90:-90:0.125) -- (-3.5,-.5) -- cycle;
\node at (3.7,1.7) {$R_2$};
\node[color=green!30!black] at (4.73,0.7) {$g^{-1}(R_3)$};
\node[color=blue!30!black] at (2.5,-2.3) {$f(R_1)$};
\end{tikzpicture}
\caption{Intersection de $f(R_1)$ et $g^{-1}(R_3)$}\label{dessin-markov}
\end{center}
\end{figure}

\begin{lemme}\label{markov-voisin}
Soit $f$ un homéomorphisme de $X$, et $R_1$ et $R_2$ deux rectangles. Si l'intersection $f(R_1)\cap R_2$ a au moins $k$ composantes markoviennes, alors cela est encore vrai sur tout un voisinage de $f$ dans $\mathrm{Homeo}(X,\mu)$.
\end{lemme}

Maintenant que les notions de rectangle et d'intersection markovienne sont bien posées, on peut montrer le théorème suivant :

\begin{theoreme}\label{entropie-topologique}
Un élément générique de $\mathrm{Homeo}(X,\mu)$ a une entropie topologique infinie.
\end{theoreme}

L'idée de la preuve est d'approcher tout homéomorphisme par un autre dont un itéré va posséder un \og $k$-fer à cheval de Smale \fg. Ceci va donner une entropie aussi grande que l'on veut à notre approximation. Commençons par étudier l'entropie du fer à cheval lui-même :

\begin{definition}
Pour $k\in\N^*$, on dit qu'on homéomorphisme conservatif $f$ contient un \emph{$k$-fer à cheval de Smale} s'il existe un rectangle $R$ dont l'intersection avec son image par $f$ contient au moins $k$ composantes markoviennes disjointes.
\end{definition}

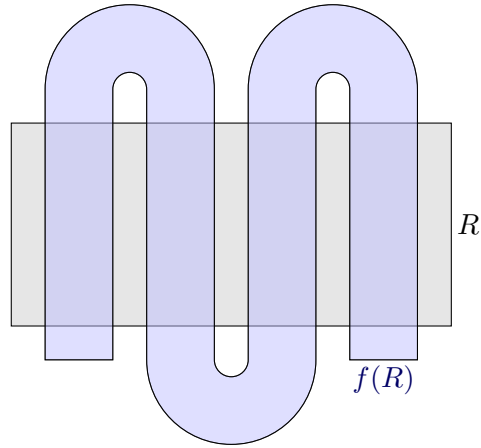
\begin{figure}
\begin{center}
\begin{tikzpicture}[scale=0.9]
\draw[fill=gray!20!white] (-3,-1.5) rectangle (3.5,1.5);
\draw[fill=blue!25!white, opacity = 0.5] (-2.5,-2) -- (-2.5,2) arc (180:0:1.25) -- (0,-2) arc (180:360:0.25) -- (0.5,2) arc (180:0:1.25) -- (3,2) -- (3,-2) -- (2,-2) -- (2,2) arc (0:180:0.25) -- (1.5,-2) arc (360:180:1.25) -- (-1,2) arc (0:180:0.25) -- (-1.5,-2) -- cycle;
\draw (-2.5,-2) -- (-2.5,2) arc (180:0:1.25) -- (0,-2) arc (180:360:0.25) -- (0.5,2) arc (180:0:1.25) -- (3,2) -- (3,-2) -- (2,-2) -- (2,2) arc (0:180:0.25) -- (1.5,-2) arc (360:180:1.25) -- (-1,2) arc (0:180:0.25) -- (-1.5,-2) -- cycle;
\node at (3.75,0) {$R$};
\node[color=blue!40!black] at (2.5,-2.3) {$f(R)$};
\end{tikzpicture}\caption{4-fer à cheval de Smale}\label{dessin-fer}
\end{center}
\end{figure}

\begin{rem}\label{34}
\begin{itemize}
\item La terminologie choisie fait référence au célèbre exemple de difféomorphisme de $\mathbf S^2$, découvert par S. Smale au début des années 60 et désormais connu sous le nom de fer à cheval de Smale. Avec notre terminologie, cet exemple est un 2-fer à cheval de Smale.
\item Par le lemme \ref{markov-voisin}, le fait de posséder un $k$-fer à cheval de Smale est stable par perturbation dans $\mathrm{Homeo}(X,\mu)$.
\item Par le lemme \ref{markov-récur}, si $f$ est un $k$-fer à cheval de Smale, alors $f^m$ est un $k^m$-fer à cheval de Smale.
\end{itemize}
\end{rem}

\begin{lemme}\label{entr-fer}
Tout $k$-fer à cheval de Smale a une entropie topologique supérieure à $\log k$.
\end{lemme}

\begin{proof}[Preuve du lemme \ref{entr-fer}] Soit $f$ un $k$-fer à cheval et $m\in\N^*$. Soit $R$ le rectangle associé à $f$. Le lemme \ref{markov-récur} assure alors que $f^m(R)\cap R$ possède au moins $k^m$ composantes markoviennes. Soient $x$ et $y$ deux points de cette intersection. Si ces deux points ne sont pas dans une même composante markovienne, il va exister $m'\le m$ tel que $f^{m-m'}(x)$ et $f^{m-m'}(y)$ soient dans des composantes différentes de $f(R)\cap R$. Cela se montre par récurrence sur~$m$ : soit $x$ et $y$ sont dans des composantes différentes de $f^{m-1}(R)\cap R$, si bien qu'on peut appliquer l'hypothèse de récurrence ; soit ils sont dans une même composante de $f^{m-1}(R)\cap R$, et donc $f^{m-1}(x)$ et $f^{m-1}(y)$ seront dans des composantes connexes différentes de $f(R)\cap R$ (en effet, par construction les composantes de $f^m(R)\cap R$ incluses dans une même composante de $f^{m-1}(R)\cap R$ sont des images par $f^{m-1}$ de sous-ensembles dans des composantes différentes de $f(R)\cap R$). La distance $d_m(x,y)$ est alors uniformément minorée par la distance minimale $\varepsilon_0$ entre deux composantes markoviennes de $f(R)\cap R$. 

Minorons maintenant l'entropie de $f$. Soit  $\varepsilon\in]0,\varepsilon_0[$, et $m \in \N^*$, on a :
\[C(m,\varepsilon)\ge k^m\]
($C(m,\varepsilon)$ est l'entier introduit dans la définition \ref{2.9}). En effet, on vient de voir que si on prenait $\varepsilon$ assez petit, les points dans des composantes markoviennes de $f^m(R)\cap R$ différentes sont $\varepsilon_0$-séparés ; on a aussi remarqué que $f^m(R)\cap R$ possède au moins $k^m$ composantes markoviennes distinctes. L'inégalité ci-dessus implique que l'entropie topologique de $f$ est supérieure à $\log k$.
\end{proof}

Finalement, par la technique de la modification locale, on approche tout homéomorphisme par un autre ayant une entropie arbitrairement grande :

\begin{lemme}\label{fer-perturb}
Soit $f\in\mathrm{Homeo}(X,\mu)$ et $\varepsilon>0$. Alors il existe un entier $p$ tel que pour tout entier $k$ non nul, il existe un homéomorphisme $g$ à distance au plus $\varepsilon$ de $f$ tel que $g^p$ possède un $k$-fer à cheval de Smale.
\end{lemme}

\begin{proof}[Preuve du lemme \ref{fer-perturb}] Par densité des homéomorphismes ayant un point périodique, on peut supposer que $f$ possède un point périodique, de période notée $p'$. En procédant comme dans la preuve de l'application \ref{p15}, on approche tout d'abord $f$ par un homéomorphisme $h$ tel que $h^p$ (avec $p=2p'$) soit égal à l'identité sur un voisinage du point périodique. Cet ouvert contient une boule $B$ disjointe de ses $p-1$ premiers itérés sur laquelle on définit un $k$-fer à cheval de Smale $\varphi$, égal à l'identité sur le bord de $B$. On étend $\varphi$ à tout $X$ par $\varphi = \mathrm{Id}$ en dehors de $B$ et pose $g = \varphi\circ h$ ; c'est l'homéomorphisme recherché.
\end{proof}

\begin{proof}[Preuve du théorème \ref{entropie-topologique}] L'ensemble des homéomorphismes ayant une entropie infinie est égal à l'intersection $\bigcap_{m\in \N} H_m$ où $H_m$ est l'ensemble des homéomorphismes ayant une entropie plus grande que $m$. Il suffit donc de montrer que $H_m$ est d'intérieur dense pour tout $m$.

Soit $f \in\mathrm{Homeo}(X,\mu)$ et $\varepsilon>0$. Le lemme \ref{fer-perturb} nous donne un entier $p$, on choisit alors $k \ge e^{mp}$, et on obtient $g\in\mathrm{Homeo}(X,\mu)$ $\varepsilon$-proche de $f$ tel que $g^p$ contienne un $k$-fer à cheval de Smale. Par la proposition \ref{entr-fer}, $g^p$ a une entropie topologique d'au moins $mp$, si bien que $g$ a une entropie d'au moins $m$. Et par la remarque \ref{34}, ceci reste vrai sur tout un voisinage de~$g$.
\end{proof}

\section{Généricité du mélange topologique fort}

Dans cette partie, nous montrons que le mélange topologique fort est générique. La preuve de ce résultat passe de nouveau par le théorème de modification locale (théorème \ref{extension-sphères}) et la notion d'intersection markovienne, qui permet de rendre la propriété \og $f^m(U)\cap V$ est non vide pour tout $m$ assez grand \fg~ouverte. Notons que la généricité du mélange topologique fort est d'autant plus intéressante que nous montrerons au chapitre \ref{chapcycl} que le mélange fort, au sens ergodique, est générique. À notre connaissance, ce résultat n'avait pas été prouvé auparavant.

\begin{definition}
Un homéomorphisme $f$ est dit \emph{topologiquement fortement mélangeant} si quels que soient les ouverts non vides $U$ et $V$, il va exister un entier $N$ tel que pour tout $m\ge N$, l'intersection $f^m (U)\cap V$ soit non vide.
\end{definition}

\begin{theoreme}\label{th-mel-fort}
Le mélange topologique fort est générique dans $\mathrm{Homeo}(X,\mu)$.
\end{theoreme}

Pour cela, il \og suffit \fg~de montrer que quels que soient les ouverts non vides $U$ et $V$, l'ensemble
\[\mathcal{H}(U,V) = \big\{f\in\mathrm{Homeo}(X,\mu) \mid \exists N\in \N : \forall m\ge N, f^m(U)\cap V\neq\emptyset  \big\}\]
est d'intérieur dense pour tous les ouverts non vides $U$ et $V$ ; on conclura en prenant une base dénombrable d'ouverts de $X$.

Le problème est que la propriété \og $f^m(U)$ intersecte $V$ pour tout $m$ assez grand \fg~n'est pas ouverte, autrement dit $\mathcal{H}(U,V)$ n'est pas forcément ouvert. Pour trouver des homéomorphismes dans l'\emph{intérieur} de $\mathcal{H}(U,V)$ qui vont former un ensemble dense, nous allons utiliser la notion d'intersection markovienne. Cette notion permet d'isoler une propriété ouverte, qui ne fait intervenir qu'un nombre fini d'itérés de $f$, mais qui assure que $f^m(U)\cap V$ est non-vide pour tout $m$ assez grand.

\begin{prop}\label{creation-rectangles}
Soit $f\in\mathrm{Homeo}(X,\mu)$, $U$ et $V$ deux ouverts de $X$ et $\varepsilon>0$. Alors il existe deux entiers premiers entre eux $p$ et $q$, un entier $k$, deux rectangles $R_x\subset U$ et $R_y\subset U$ et un homéomorphisme préservant la mesure $g$ vérifiant $d_{\mathit{forte}}(f,g)<\varepsilon$, tels que les intersections $g^p(R_x)\cap R_x$, $g^p(R_x)\cap R_y$, $g^q(R_y)\cap R_x$ et $g^q(R_y)\cap R_y$ soient markoviennes, et que $g^k(R_x\cup R_y)\subset V$.
\end{prop}

Cette proposition donne facilement le théorème \ref{th-mel-fort} :

\begin{proof}[Preuve du théorème \ref{th-mel-fort}] Il suffit de montrer que quels que soient les ouverts non vides $U$ et $V$, l'intérieur de l'ensemble $\mathcal{H}(U,V)$ est dense.
Soient donc deux ouverts non vides $U$ et $V$, $f\in\mathrm{Homeo}(X,\mu)$ et $\varepsilon>0$.

La proposition \ref{creation-rectangles} nous donne deux entiers premiers entre eux $p$ et $q$, un entier $k$, deux rectangles $R_x\subset U$ et $R_y\subset U$, et un homéomorphisme préservant la mesure $g$ tel que $d_{\mathit{forte}}(f,g)<\varepsilon$.

Par le lemme \ref{markov-voisin}, les conclusions de la proposition sont encore vraies sur un voisinage de $g$. Soit donc $h$ dans ce voisinage et $m\ge pq$ ; montrons que $h^m(U)\cap R_x$ est non vide.

On peut écrire $m = \alpha p+\beta q$, avec $\alpha$ et $\beta$ positifs\footnote{\label{note-page}En effet, puisque $p$ et $q$ sont premiers entre eux, le théorème de Bézout affirme qu'il existe $\alpha$ et $\beta$ tels que $m$ s'écrive $m = \alpha p+\beta q$. Si on décompose de plus $\alpha$ en $\alpha = iq+j$, avec $0\le j<q$ et pose $\alpha'=\alpha-iq$ et $\beta'=\beta+ip$, on a toujours $m = \alpha' p+\beta' q$ mais aussi $\alpha'\ge 0$ et $\beta'q>m-pq\ge 0$.}. En appliquant $\beta-1$ fois le lemme \ref{markov-récur} à $R_y$ et $h^q$, on en déduit que l'intersection $(h^q)^{\beta-1}(R_y)\cap R_y$ est markovienne. Cela implique, toujours par le lemme \ref{markov-récur}, que l'intersection $(h^q)^{\beta}(R_y)\cap R_x$ est markovienne. Enfin, on applique $\alpha$ fois le lemme pour pouvoir dire que l'intersection $(h^p)^{\alpha}(h^q)^{\beta}(R_y)\cap R_x = h^m(R_y)\cap R_x$ est markovienne ; en particulier elle est non vide.

Prenons maintenant $m'\ge pq+k$. On pose $m= m'-k$ et on applique ce qu'on a montré à $m$. Alors $h^m(R_y)\cap R_x$ est non vide, et puisque l'intersection $h^k(R_x)\subset V$, $h^{m'}(R_y)\cap V$ est non vide. Finalement, $h^{m'}(U)\cap V$ est non vide ; l'ensemble $\mathcal{H}(U,V)$ est d'intérieur dense.
\end{proof}

Il ne reste plus qu'à prouver la  proposition \ref{creation-rectangles} :

\begin{proof}[Preuve de la proposition \ref{creation-rectangles}] Prenons deux ouverts non vides $U$ et $V$, ainsi qu'un homéomorphisme préservant la mesure $f$. Par densité des homéomorphismes ayant un ensemble dense de points périodiques de période une puissance de $2$ (lemme \ref{lem mélange topo faible}), on peut perturber $f$ en un homéomorphisme que l'on appelle $g$, de telle sorte qu'il existe un point périodique $x\in U$ de période $p$ une puissance de $2$ et un entier $k$ tels que $g^k(x)\in V$. Prenons une boule $B$ non vide de diamètre $\varepsilon$ contenant $x$ et incluse dans $g^{-k}(V)\cap U$. Toujours par le lemme \ref{lem mélange topo faible}, on peut supposer qu'il y a un point périodique $y$ dans $B$ de période $q$ une puissance de $3$. Si on récapitule, on a :
\begin{eqnarray*}
x\in B \quad & g^p(x) = x\\
y\in B \quad & g^q(y) = y\\
B\subset U \quad & g^k(B)\subset V
\end{eqnarray*}

Quitte à réduire $\varepsilon$ dès le début, on peut supposer que $B$ est dans un ouvert trivialisant de la variété et donc découper $B$ par des subdivisions. Si la subdivision est assez fine, les points de l'orbite de $x$, et ceux de l'orbite de $y$, seront tous dans des cubes différents et qui ont leurs frontières deux à deux disjointes. Ainsi il existe un ouvert $D$ homéomorphe à une boule dont la frontière est épaississable, et contenant $x$ et $y$, mais à une distance strictement positive des itérés de $x$ : $(g(x),\dots,g^{p-1}(x))$ et de ceux de $y$ : $(g(y),\dots,g^{q-1}(y))$ (en effet il suffit de relier $x$ et $y$ par un chemin ne passant par aucun des cubes contenant les itérés indésirables, ce qui est possibles car ces cubes sont isolés ; $D$ est alors l'union des cubes par lesquels passe ce chemin, voir la figure \ref{cubcub}).

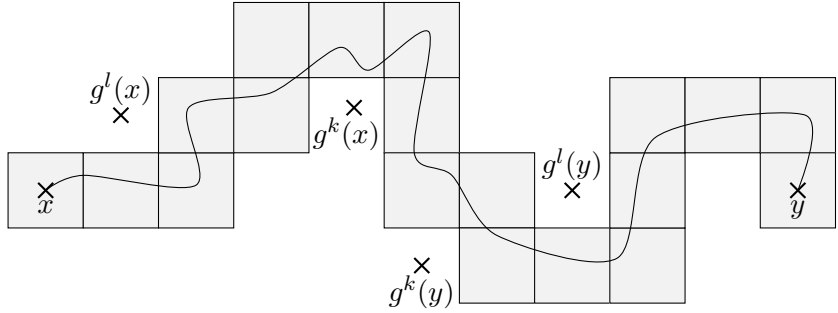
\begin{figure}
\begin{center}
\begin{tikzpicture}[scale=1]
\draw[fill=gray!10!white] (0,0) rectangle (1,1);
\draw[fill=gray!10!white] (1,0) rectangle (2,1);
\draw[fill=gray!10!white] (2,0) rectangle (3,1);
\draw[fill=gray!10!white] (2,1) rectangle (3,2);
\draw[fill=gray!10!white] (3,1) rectangle (4,2);
\draw[fill=gray!10!white] (3,2) rectangle (4,3);
\draw[fill=gray!10!white] (4,2) rectangle (5,3);
\draw[fill=gray!10!white] (5,2) rectangle (6,3);
\draw[fill=gray!10!white] (5,1) rectangle (6,2);
\draw[fill=gray!10!white] (5,0) rectangle (6,1);
\draw[fill=gray!10!white] (6,0) rectangle (7,1);
\draw[fill=gray!10!white] (6,-1) rectangle (7,0);
\draw[fill=gray!10!white] (7,-1) rectangle (8,0);
\draw[fill=gray!10!white] (8,-1) rectangle (9,0);
\draw[fill=gray!10!white] (8,0) rectangle (9,1);
\draw[fill=gray!10!white] (8,1) rectangle (9,2);
\draw[fill=gray!10!white] (9,1) rectangle (10,2);
\draw[fill=gray!10!white] (10,1) rectangle (11,2);
\draw[fill=gray!10!white] (10,0) rectangle (11,1);
\draw plot[smooth] coordinates{(0.5,0.5) (1,0.7) (2.5,0.57) (2.4,1.6) (3.5,1.8) (4.4,2.4) (4.8,2.1) (5.6,2.6) (5.4,1) (5.9,0.7) (6.5,-.1) (8.1,-.4) (8.6,1.2) (10.6,1.5) (10.5,0.5)};
\draw[thick] (0.4,0.6) -- (0.6,0.4);
\draw[thick] (0.4,0.4) -- (0.6,0.6);
\node[below] at (0.5,0.5) {$x$};
\draw[thick] (1.4,1.6) -- (1.6,1.4);
\draw[thick] (1.4,1.4) -- (1.6,1.6);
\node[above] at (1.5,1.5) {$g^l(x)$};
\draw[thick] (4.5,1.7) -- (4.7,1.5);
\draw[thick] (4.5,1.5) -- (4.7,1.7);
\node[below] at (4.5,1.6) {$g^k(x)$};
\draw[thick] (5.4,-0.6) -- (5.6,-0.4);
\draw[thick] (5.4,-0.4) -- (5.6,-0.6);
\node[below] at (5.5,-0.5) {$g^k(y)$};
\draw[thick] (7.4,0.6) -- (7.6,0.4);
\draw[thick] (7.4,0.4) -- (7.6,0.6);
\node[above] at (7.5,0.5) {$g^l(y)$};
\draw[thick] (10.4,0.4) -- (10.6,0.6);
\draw[thick] (10.4,0.6) -- (10.6,0.4);
\node[below] at (10.5,0.5) {$y$};
\end{tikzpicture}
\caption{Construction de l'ensemble $D$}\label{cubcub}
\end{center}
\end{figure}

Alors il existe (par uniforme continuité) un petit rectangle $R_x$ contenant $x$ et inclus dans $D$, et un petit rectangle $R_y$ contenant $y$ et inclus dans $D$, tels que $g^p(R_x)\subset D$ et $g^q(R_y)\subset D$, mais que les itérés $(g(R_x),\dots ,g^{p-1}(R_x))$ et $(g(R_y),\dots ,g^{q-1}(R_y))$ soient disjoints de $D$. On suppose de plus que chacun de $g^p(R_x)$ et $g^q(R_y)$ est de diamètre inférieur à la demi distance de $x$ à $y$, si bien que l'on peut partager $D$ en trois ensembles disjoints (que l'on supposera unions de cubes de la subdivision) dont les frontières sont des sphères épaississables $D_x$, $D_y$ et $D'$, tels que $R_x\cup g^p(R_x) \subset D_x$, $R_y\cup g^q(R_y) \subset D_y$, et que $D_x$ et $D_y$ n'aient pas de frontière en commun. Il est évident que $R_x$ et $R_y$ sont eux-mêmes épaississables.

Plaçons-nous maintenant dans l'ouvert $g^{-1}(D)$. Par hypothèse, il contient $g^{p-1}(R_x)$ et $g^{q-1}(R_y)$, mais pas les itérés précédents de $R_x$ et de $R_y$. On utilise alors le théorème \ref{extension-sphères} deux fois. On l'applique tout d'abord à
\[\begin{array}{rlrl}
A_1 & = g^{p-1}(R_x)\quad          & A_2 & = g^{-1}(D_x)^\complement\\
B_1 & = R_x \quad                  & B_2 & = D_x^\complement\\
f_1 & =\sigma\circ g^{-(p-1)}\quad & f_2 & = g
\end{array}\]
avec $\sigma$ un renversement éventuel de l'orientation. On fait de même pour le rectangle $R_y$ (en remplaçant $x$ par $y$ et $p$ par $q$). On se retrouve maintenant avec une application $g'$, que l'on peut prendre aussi proche que l'on veut de $g$, vérifiant $g'^p(R_x) = R_x$ et $g'^q(R_y) = R_y$.

Regardons maintenant ce qui se passe sur $D$. Il est facile de trouver un homéomorphisme $h$ de $D$ égal à l'identité sur $\partial D$, tel que $h(R_x)\cap R_x$ et $h(R_y)\cap R_x$ soient des sous-rectanges horizontaux de $R_x$ et que $h(R_x)\cap R_y$ et $h(R_y)\cap R_y$ soient des sous-rectanges horizontaux de $R_y$ (voir figure \ref{construction-g}).

\begin{figure}
\begin{center}
\begin{tikzpicture}
\draw[fill=gray!10!white] (1,-1.5) rectangle (4,1.5);
\draw[fill=gray!10!white] (-4,-1.5) rectangle (-1,1.5);
\draw[fill=gray!20!white, opacity = 0.5] (-4.5,0.25) rectangle (4.5,1.25);
\draw[fill=gray!20!white, opacity = 0.5] (-4.5,-0.25) rectangle (4.5,-1.25);
\draw (-4.5,0.25) rectangle (4.5,1.25);
\draw (-4.5,-0.25) rectangle (4.5,-1.25);
\draw (0.9,-0.1) -- (1.1,0.1);
\draw (1.1,-0.1) -- (0.9,0.1);
\draw (-4.4,0.65) -- (-4.6,0.85);
\draw (-4.6,0.65) -- (-4.4,0.85);
\draw (-4.1,0.1) -- (-3.9,0);
\draw (-4.1,0) -- (-3.9,-0.1);
\draw (-4.6,-0.75) -- (-4.4,-0.85);
\draw (-4.6,-0.65) -- (-4.4,-0.75);

\node at (0,1.5) {$h(R_y)$};
\node at (0,-1.55) {$h(R_x)$};
\node at (2.5,1.75) {$R_y$};
\node at (-2.5,1.75) {$R_x$};
\end{tikzpicture}
\caption{Construction de l'homéomorphisme $h$}\label{construction-g}
\end{center}
\end{figure}
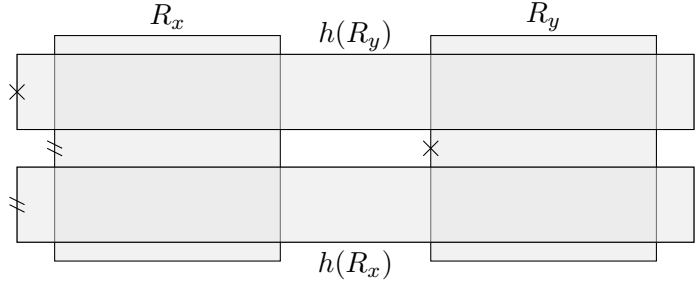

Posons maintenant $g''=h\circ g'$. Alors les intersections $g''^p(R_x)\cap R_x$, $g''^p(R_x)\cap R_y$, $g''^q(R_y)\cap R_x$ et $g''^q(R_y)\cap R_y$ sont markoviennes et $g''^k(R_x\cup R_y)$ est inclus dans $V$. Les conclusions de la proposition sont démontrées.
\end{proof}

\chapter[Approximations en topologie faible dans $\mathrm{Homeo}(X,\mu)$]{Approximations des homéomorphismes par des permutations en topologie faible}\label{chapcycl}

Dans les chapitres précédents, nous avons observé que l'approximation des homéomorphismes par des permutations joue un rôle important dans certaines preuves de généricité. Cette idée d'approximation est apparue dès les années~40, lorsque P. Halmos l'a utilisée pour montrer la généricité du mélange faible \cite{HalmMix}. Une vingtaine d'années plus tard, A. Katok et A. Stepin, entre autres, se sont rendu compte que l'étude systématique de la \emph{vitesse} de cette approximation, mesurée à l'aide de la distance faible cette fois-ci, permet d'obtenir bien d'autres résultats de généricité dans $\mathrm{Homeo}(X,\mu)$ \cite{KatokS3}. De la même manière que la vitesse d'approximation d'un réel par les rationnels fournit des indications sur ce réel, la vitesse d'approximation d'un automorphisme (en particulier d'un homéomorphisme) conservatif par les permutations cycliques peut être exploitée en vue de son étude dynamique. Ce lien assez fort avec l'approximation diophantienne est d'ailleurs sous-jacent dans certaines preuves (voir \cite{Yuz}). Le théorème fondamental pour notre étude est le fait que génériquement, un homéomorphisme admet des approximations par des permutations à une vitesse donnée arbitraire (théorème \ref{approxmeo}). Ce théorème a d'abord été prouvé pour l'ensemble des automorphismes en 1968 \cite{KatokS1}, puis pour l'ensemble des homéomorphismes en 1970 par A. Katok et A. Stepin \cite{KatokS3}. Les résultats du type \og si un homéomorphisme est approché par des permutations de tel type (souvent cycliques) à une vitesse au moins $\vartheta$, alors il possède telle propriété \fg~se transforment alors automatiquement en résultats de généricité\footnote{Néanmoins, comme il se passe exactement la même chose dans $\mathrm{Auto}(X,\mu)$ qui \emph{contient} $\mathrm{Homeo}(X,\mu)$, nous énoncerons ces résultats pour les automorphismes.} dans $\mathrm{Homeo}(X,\mu)$, qui est muni rappelons-le de la topologie forte.

Nous avons choisi de présenter, outre le théorème \ref{approxmeo}, bon nombre des résultats de généricité obtenus par cette méthode\footnote{La plupart sont dus à A. Katok et A. Stepin.} : un automorphisme admettant une approximation cyclique à une vitesse assez grande, donc un homéomorphisme générique (voir \cite{Katok}, \cite{KatokS1}, \cite{KatokS2} et \cite{KatokT}) :
\begin{itemize}
\item est ergodique (théorème \ref{OxtoUlApprox}), cela donne une nouvelle preuve du théorème d'Oxtoby-Ulam,
\item est non fortement mélangeant (théorème \ref{MélApprox}),
\item est rigide (théorème \ref{rigidgéné}),
\item est faiblement mélangeant (théorème \ref{mélfaiblgéné}),
\item a une entropie métrique nulle (théorème \ref{entrmetnullgéné}),
\item est de rang un (théorème \ref{rang1F}),
\item est à spectre simple (théorème \ref{specsimplgéné}),
\item est standard (théorème \ref{pffff}),
\item a son type spectral sans atome (théorème \ref{speccon}) et singulier par rapport à la mesure de Lebesgue (théorème \ref{singLeb}).
\end{itemize}
Notons que puisque la vitesse d'approximation est mesurée à l'aide de la distance faible, on obtient naturellement des résultats sur les propriétés ergodiques\footnote{Bien que certains résultats topologiques puissent être obtenus ainsi, par exemple la transitivité topologique, qui de toute façon découle de l'ergodicité.}.

\section[Généricité et approximation à vitesse fixée]{Généricité des homéomorphismes admettant une approximation à vitesse fixée}

On se donne $\vartheta : \N\to ]0,+\infty[$ une fonction décroissante tendant vers 0 en~$+\infty$ qui sera la \emph{vitesse} d'approximation. Comme aux chapitres précédents, on fixe une application $\phi : I^n\to X$ donnée par le corollaire \ref{Brown-mesure}, ce qui permet de choisir une suite $(\D_m)_{m\in\N}$ de subdivisions dyadiques de $(X,\mu)$ (image par $\phi$ d'une suite de subdivisions dyadiques de $I^n$) et de définir une notion de permutation dyadique sur ces subdivisions.

\begin{definition}\label{defappcycl}
Soit $f\in \mathrm{Auto}(X,\mu)$. On dit que $f$ admet une \emph{approximation cyclique à vitesse $\vartheta$} s'il existe une suite strictement croissante d'entiers positifs $(m_k)_{k\ge 0}$ et pour tout $k\ge 0$ une permutation cyclique dyadique d'ordre $m_k$ notée $f_k$ telles que, notant $(C_{k,i})_{1\le i \le q_{m_k}}$ les cubes de la subdivision dyadique $\D_{m_k}$, on ait :
\[\sum_{i=1}^{q_{m_k}} \mu\big(f(C_{k,i})\Delta f_k(C_{k,i})\big)\le\vartheta(q_{m_k}).\]
\end{definition}

On peut maintenant énoncer le théorème principal de cette partie :

\begin{theoreme}[Katok, Stepin]\label{approxmeo}
L'ensemble des éléments de $\mathrm{Homeo}(X,\mu)$ admettant une approximation cyclique à la vitesse $\vartheta$ est un $G_\delta$ dense.
\end{theoreme}

La preuve de ce théorème, tirée de \cite{KatokS3}, est à rapprocher de celle, plus simple, de la généricité des nombres de Liouville. \label{Liouv}L'ensemble des réels $\alpha$ approchés à une vitesse $\vartheta$ par des rationnels sous forme irréductible $p/q$ est l'ensemble
\[\bigcap_{m\in\N}\bigcup_{q\ge m}\bigcup_{p\in\N}\left\{\alpha\mid \left| \frac{p}{q}-\alpha\right|<\vartheta(q)\right\},\]ce qui l'exprime sous la forme d'un $G_\delta$ dense\footnote{Pour plus de précisions sur les normbres de Liouville, l'approximation diophantienne et les espaces de Baire, voir \cite{7}.}. Ici, la propriété de densité des rationnels parmi les réels est remplacée par le lemme \ref{4.2}, obtenu par A. Katok et A.~Stepin dans \cite{KatokS3} :

\begin{lemme}\label{4.2}
Soit $f\in \mathrm{Homeo}(X,\mu)$ et $\varepsilon>0$. Alors il existe un entier $m\ge 0$ tel que si on note $C_i$ les cubes de la subdivision dyadique $\D_m$, quel que soit $\delta>0$, il existe $g\in\mathrm{Homeo}(X,\mu)$, ainsi que des cubes $c_i\subset C_i$ de taille constante, tels que :
\begin{enumerate}[(1)]
\item $d_{\mathit{forte}}(f,g)\le\varepsilon$,
\item pour tout $i$, $g_{|c_i}$ est égal à la composée de la translation\footnote{Voir la définition \ref{translat}.} allant de $c_i$ vers $c_{i+1}$ avec $\sigma$ (où $\sigma$ est soit l'identité, soit le renversement d'une coordonnée, selon que $f$ préserve l'orientation ou non),
\item pour tout $i$, $\mu(c_i)\ge\frac{1-\delta}{q_m}$, autrement dit $c_i$ remplit une proportion d'au moins $1-\delta$ de $C_i$.
\end{enumerate}
\end{lemme}

Le lemme \ref{4.2} exprime que tout homéomorphisme est approché par un autre homéomorphisme qui est égal à une permutation cyclique sur un ensemble de mesure arbitrairement grande. Nous présentons une preuve de ce lemme, due à S. Alpern et utilisant l'annulus theorem \cite{AlpernOuique}, qui nous semble plus claire et plus simple que la preuve originale de A. Katok et A. Stepin.

\begin{proof}[Preuve du lemme \ref{4.2}] Soit $f\in \mathrm{Homeo}(X,\mu)$ et $\varepsilon>0$. Le lemme \ref{lemmetrans} nous donne un homéomorphisme $h$ à distance au plus $\varepsilon/2$ de $f$ ayant une orbite cyclique constituée de centres $x_i$ de cubes $C_i$ d'une subdivision dyadique $\D_m$. On peut supposer que la taille des cubes de $\D_m$ n'excède pas $\varepsilon/2$. On remplace alors localement $h$ au voisinage de chaque $x_i$ par la composée d'une translation (définie on le rappelle à l'aide de l'application $\phi$ du corollaire \ref{Brown-mesure}) avec $\sigma$ (qui est l'identité ou bien le renversement d'une coordonnée, de telle manière que soit $h$ et $\sigma$ préservent simultanément l'orientation, soit la renversent simultanément). Pour cela, on utilise le théorème~\ref{extension-sphères}, qui permet de remplacer $h$, au voisinage de chaque $x_i$, par la translation allant de $x_i$ à $x_{i+1}$, translation qui est bien définie parce que l'application $\phi$ du corollaire~\ref{Brown-mesure} est un homéomorphisme sur $I^n\setminus \partial I^n$, et parce que les points $x_i$ sont tous dans $X\setminus \phi(\partial I^n)$. On obtient alors un homéomorphisme $h'$ qui est $\varepsilon/2$-proche de $h$ et qui vérifie simultanément les points (1) et (2) du lemme sur des cubes $c_i^0$.

Reste à obtenir le point (3), autrement dit à remplacer les cubes $c_i^0$ par des cubes $c_i$ vérifiant l'inégalité $\mu(c_i)\ge\frac{1-\delta}{q_m}$. Pour ce faire, on prend des cubes $c_i$ comme définis dans l'énoncé du lemme, avec $\mu(c_i)\ge\frac{1-\delta}{q_m}$ et on construit un homéomorphisme $\Psi : X\to X$ qui fixe globalement chaque cube $C_i$ et qui envoie chaque petit cube $c_i^0$ sur le cube $c_i$. Pour définir l'homéomorphisme $\Psi$, on commence par se placer dans le cas du cube $I^n$ muni de la mesure de Lebesgue à l'aide de l'application $\phi$ du corollaire \ref{Brown-mesure}, ce qui ne posera pas de problème par la suite vu que la fonction $\Psi$ sera égale à l'identité au voisinage du bord de tout cube $C_i$ ; on pourra donc supposer $\phi$ uniformément continue, et les inégalités des points (1) et (3) du lemme seront conservées\footnote{À proprement parler elles ne seront pas conservées en elles-mêmes, mais les conclusions du lemme resteront vraies.}. Quitte à composer par un homéomorphisme conservatif\footnote{Un tel homéomorphisme est donné par exemple par l'application du théorème \ref{extension-sphères}.} envoyant $C_1$ et $c_1^0$ sur des boules concentriques $B_1$ et $b_1^0$, on peut définir notre application sur $B_1$ plutôt que sur $C_1$. On définit alors la fonction $\widetilde\Psi : B_1\to B_1$, continue et de classe $C^1$ sur $B_1\setminus \partial b_1^0$, et telle que (voir la figure \ref{Jacques}) :
\begin{itemize}
\item sur $b_1^0$, $\widetilde\Psi$ soit une homothétie de centre $x_1$ qui envoie $b_1^0$ sur la boule $(1-\delta)B_1$ (c'est à dire l'image de $B_1$ par une homothétie de centre $x_1$ de rapport $1-\delta$),
\item $\widetilde\Psi$ ait un jacobien constant sur $B_1\setminus b_1^0$,
\item $\widetilde\Psi$ soit l'identité sur la frontière de $B_1$.
\end{itemize}

\begin{figure}
\begin{center}
\begin{tikzpicture}[scale=1.1]
\draw (0,0) circle (2);
\draw[fill=gray!25!white,draw=black] (0,0) circle (1);
\draw (7,0) circle (2);
\draw[fill=gray!10!white,draw=black] (7,0) circle (1.8);
\draw[->, thick, color=red!30!black] (.4,0) -- (1.4,0);
\draw[->, thick, color=red!30!black] (0,.4) -- (0,1.4);
\draw[->, thick, color=red!30!black] (-.4,0) -- (-1.4,0);
\draw[->, thick, color=red!30!black] (0,-.4) -- (0,-1.4);
\draw[->, thick, color=red!30!black] (.283,.283) -- (.99,.99);
\draw[->, thick, color=red!30!black] (-.283,.283) -- (-.99,.99);
\draw[->, thick, color=red!30!black] (.283,-.283) -- (.99,-.99);
\draw[->, thick, color=red!30!black] (-.283,-.283) -- (-.99,-.99);
\draw[->, semithick] (2.2,0) .. controls (3.2,.3) and (3.8,.3) .. (4.7,0);
\node at (3.5,.6) {$\widetilde\Psi$};
\node at (0,0) {$b_1^0$};
\node at (1.5,-.5) {$B_1$};
\node at (7,0) {$(1-\delta)B_1$};
\node at (9.25,0) {$B_1$};
\end{tikzpicture}
\caption{L'application $\widetilde\Psi$}\label{Jacques}
\end{center}
\end{figure}

On considère l'application associée $\Psi : C_1\to C_1$ que l'on étend de manière naturelle à $X$ entier en considérant ses conjugués par les translations du cube $C_1$ vers les autres cubes $C_i$. On pose alors $g = \Psi h' \Psi^{-1}$ et on a :
\begin{itemize}
\item $g$ préserve le volume car d'une part $\widetilde\Psi$ a un jacobien constant en dehors des $c_i$ et d'autre part $h'$ est une translation sur les $c_i^0$ et $\widetilde\Psi$ est invariant par translation ; par conséquent $g\in\mathrm{Homeo}(X,\mu)$
\item $g$ vérifie le point (2) du lemme.
\end{itemize}
En d'autres termes cette manipulation nous a permis de grossir les cubes $c_i$ de manière à ce qu'ils soient presque aussi gros que les $C_i$ ; ils en remplissent une proportion d'au moins $1-\delta$. On a de plus
\[d_{\mathit{forte}}(g,h')\le d_{\mathit{forte}}(\Psi,Id)+d_{\mathit{forte}}(h'\circ\Psi^{-1},h').\]
Le premier terme est majoré par $\varepsilon$ et le second tend vers 0 lorsque $m$ tend vers $+\infty$ (le module de continuité de $h'$ est uniformément majoré, en fait de plus en plus proche de celui de $f$, et la distance $d_{\mathit{forte}}(\Psi^{-1},Id)$ tend vers 0). Ainsi on peut prendre $d_{\mathit{forte}}(f,g)\le \varepsilon$.
\end{proof}

\begin{proof}[Preuve du théorème \ref{approxmeo}] Notons $\mathcal P(\vartheta,m)$ l'ensemble des homéo\-mor\-phismes vérifiant les points (2) et (3) du lemme \ref{4.2} avec $\delta = \frac{\vartheta(q_m)}{2}$ et la subdivision $\D_m$. Ce même lemme affirme que l'ensemble $\bigcup_m \mathcal P(\vartheta,m)$ est dense dans $\mathrm{Homeo}(X,\mu)$.

Soit $g\in \mathcal P(\vartheta,m)$. On choisit $\varepsilon_m$ tel que le $\varepsilon_m$-voisinage de $c_i$ soit inclus dans $C_i$ et considère la boule $B(g,\varepsilon_m)$ dans $\mathrm{Homeo}(X,\mu)$, de centre $g$ et de rayon $\varepsilon_m$ pour la distance forte. Alors l'ensemble 
\[\mathcal L (\vartheta,m) = \bigcup_{g\in P(\vartheta,m)}B(g,\varepsilon_m)\]
est un ouvert de $\mathrm{Homeo}(X,\mu)$ si bien que, puisque $\bigcup_m \mathcal P(\vartheta,m)$ est dense,
\[\mathcal L(\vartheta) = \bigcap_{M\ge 0}\bigcup_{m>M} \mathcal L (\vartheta,m) = \overline{\lim_{m}} \ \mathcal L (\vartheta,m)\]
est un $G_\delta$ dense de $\mathrm{Homeo}(X,\mu)$.

Montrons maintenant que chaque homéomorphisme $f\in \mathcal L(\vartheta)$ admet une approximation cyclique à la vitesse $\vartheta$. Par hypothèse, il existe alors une suite d'entiers $m_k$ strictement croissante et une suite d'homéomorphismes $g_k\in P(\vartheta,m_k)$ tels que $d_{\mathit{forte}}(f,g_k)\le\varepsilon_{m_k}$, auxquels sont associés des cubes $C_{k,i}$ et $c_{k,i}$. On a alors
\[f(g_k^{i-1}(c_{k,1}))\subset B(g_{k}^{i}(c_{k,1}),\varepsilon_{m_k}) = B(c_{k,i+1},\varepsilon_{m_k}) \subset C_{k,i+1}.\]
Par conséquent, les ensembles $f(g_k^i(c_1^k))$ sont deux à deux disjoints et on a :
\begin{eqnarray*}
\sum_{i=1}^{q_{m_k}} \mu\big(f(C_{k,i})\Delta C_{k,i+1}\big) & = & \sum_{i=1}^{q_{m_k}} \mu\left(f\big(g_k^{i-1}(c_{k,1})\cup g_k^{i-1}(C_{k,1}\setminus c_{k,1})\big)\Delta C_{k,i+1}\right)\\
                      & \le & 2\mu\left(\left(\bigcup_{i=1}^{q_{m_k}} c_{k,i}\right)^\complement\right) \le \vartheta(q_{m_k}).
\end{eqnarray*}
Posons $f_k\in\mathrm{Auto}(X,\mu)$ la permutation cyclique des $C_{k,i}$, alors
\[\sum_{i=1}^{q_{m_k}} \mu(f(C_{k,i})\Delta f_k(C_{k,i})) \le \vartheta(q_{m_k}).\] 
Par conséquent la suite $f_k$ vérifie les conclusions du théorème.
\end{proof}

On peut tout à fait remplacer dans la démonstration précédente le théorème de Lax par sa version traitant de l'approximation par des permutations bicycliques (définition \ref{vélo} et corollaire \ref{Laxbis}) ; cela donne le corollaire suivant :

\begin{coro}\label{approxmeobis}
L'ensemble des éléments de $\mathrm{Homeo}(X,\mu)$ admettant une approximation à la vitesse $\vartheta$ par des permutations dyadiques bicycliques est un $G_\delta$ dense.
\end{coro}

\begin{rem}
On a déjà vu l'approximation par des permutations cycliques et des permutations ayant deux orbites de longueurs premières entre elles. De la même manière, toute variante du théorème de Lax du type corollaire \ref{Laxbis} donne automatiquement une variante du théorème \ref{approxmeo} du type corollaire~\ref{approxmeobis}. Une théorie des approximations par des permutations autres que cycliques a été développée pour l'ensemble des automorphismes, voir la partie \ref{alphamél}.
\end{rem}

Enfin, donnons un lemme calculatoire qui sera utile par la suite :

\begin{lemme}\label{itéré}
Soit $f$ un automorphisme admettant une approximation cyclique à vitesse $\vartheta$. En reprenant les notations de la définition \ref{defappcycl}, on a pour tout entier $p$ :
\[\sum_{i=1}^{q_{m_k}} \mu (f^p (C_{k,i})\Delta f_k^p (C_{k,i}))\le p\vartheta(q_{m_k}).\]
\end{lemme}

\begin{proof}[Preuve du lemme \ref{itéré}] En effet, en utilisant d'abord inégalité triangulaire, puis le fait que $f$ préserve la mesure et le fait que $f_k$ permute circulairement les $C_{k,i}$, on obtient :
\begin{eqnarray*}
\sum_{i=1}^{q_{m_k}} \mu (f^p (C_{k,i})\Delta f_k^p (C_{k,i})) & \le & \sum_{i=1}^{q_{m_k}}\sum_{j=0}^{p-1} \mu (f^jf_k^{p-j} (C_{k,i})\Delta f^{j+1} f_k^{p-j-1} (C_{k,i}))\\
  & \le & \sum_{i=1}^{q_{m_k}}\sum_{j=0}^{p-1} \mu (f_k (C_{k,i_j})\Delta f (C_{k,i_j}))\\
  & \le & p\vartheta(q_{m_k}),
\end{eqnarray*}
où $i_j = i+p-1-j$.
\end{proof}

\section{Première preuve de la généricité du non mélange fort}

La technique des approximations périodiques s'avère être un outil extrêmement puissant d'obtention de propriétés ergodiques génériques des homéomorphismes. Commençons par en donner des applications simples, qui permettent de retrouver rapidement pour les homéorphismes les résultats historiques de P. Halmos \cite{HalmMix} et V. Rokhlin \cite{Rok} sur les automorphismes, avant d'aborder des résultats plus aboutis. Ainsi, comme première application nous montrons la généricité du contraire du mélange fort (la seconde preuve sera donnée comme application du théorème de transfert, au théorème \ref{mélange pas}). Rappelons la définition du mélange fort ergodique :

\begin{definition}\label{mélfortergo}
Un automorphisme $f\in \mathrm{Auto}(X,\mu)$ est dit \emph{fortement mélangeant} si pour tous mesurables $A$ et $B$, 
\[\mu\big(A\cap f^m(B)\big)\mathop{\longrightarrow}_{m\to+\infty}\mu(A)\mu(B).\]
\end{definition}

\begin{theoreme}[Katok, Stepin]\label{MélApprox}
Dans $\mathrm{Homeo}(X,\mu)$, le non mélange fort est générique.
\end{theoreme}

Ce théorème découle immédiatement du théorème \ref{approxmeo} et du résultat d'approximation suivant :

\begin{prop}\label{approxmél}
Si un automorphisme admet une approximation cyclique à la vitesse $\vartheta$, avec $\vartheta(q)=\frac \alpha q$ et $\alpha<1$, alors il n'est pas fortement mélangeant.
\end{prop}

La preuve de cette proposition est issue de celle du théorème 2.2 de \cite{KatokS1}.

\begin{proof}[Preuve de la proposition \ref{approxmél}] Soit $f$ un automorphisme admettant une approximation cyclique à la vitesse $\vartheta$, avec $\vartheta(q)=\frac \alpha q$ et $\alpha<1$. Soient aussi $\delta\in]0,1-\alpha[$, un entier $l$ tel que $\frac 2l <1-\alpha-\delta$, et une partition mesurable $F_1,\dots,F_l$ de $X$, constituée d'ensembles de mesure $1/l$. On commence par approcher ces ensembles par d'autres ensembles formés d'unions de cubes : pour $k$ assez grand, la subdivision dyadique $\D_{m_k}$, constituée de $q_{m_k}$ cubes $C_{k,i}$, est telle qu'il existe une partition de $X$ en $l$ sous-ensembles $F_1^k,\dots,F_l^k$ formés d'unions de cubes de $\D_{m_k}$ vérifiant $\mu(F_i\Delta F_i^k)\le \frac \delta l$ pour tout $i$. On suppose cette propriété vérifiée, et on prend une approximation cyclique $f_k$ de $f$ telle que $\sum_i\mu (f(C_{k,i})\Delta f_k(C_{k,i}))\le\vartheta(q_{m_k})$. On obtient, à l'aide du lemme~\ref{itéré} :
\begin{eqnarray*}
\sum_{j=1}^l\mu\big(f^{q_{m_k}}(F_j)\cap F_j\big) & \ge & \sum_{j=1}^l\Big(\mu\big(f_k^{q_{m_k}}(F_j^k)\cap F_j^k\big)\\
                                                  &     & - \mu\big(f_k^{q_{m_k}}(F_j^k)\Delta f^{q_{m_k}}(F_j^k)\big) - \mu\big(F_j^k\Delta F_j\big)\Big)\\
                                                  & \ge & 1 - q_{m_k} \vartheta(q_{m_k}) - \delta\\
                                                  & \ge & 1-\alpha-\delta\ge\frac 2l.
\end{eqnarray*}
Ainsi il existe $j_0\in \{1,\dots,l\}$ tel que
\[\underset{k\to+\infty}{\overline{\lim}}\ \mu\big(f^{q_{m_k}}(F_{j_0})\cap F_{j_0}\big) \ge\frac{2}{l^2},\]
ce qui contredit le fait que $f$ est fortement mélangeant (car sinon on aurait $\underset{k\to+\infty}{\overline{\lim}}\ \mu\big(f^{q_{m_k}}(F_{j_0})\cap F_{j_0}\big) = \mu(F_{j_0})^2 = \frac{1}{l^2}$).
\end{proof}

On peut en fait obtenir une propriété plus forte que le non mélange fort : la rigidité.

\begin{definition}\label{defrigid}
On dit qu'un automorphisme $f\in\mathrm{Auto}(X,\mu)$ est \emph{rigide} s'il existe une sous-suite des itérés de $f$ qui converge faiblement vers l'identité, ou de manière équivalente s'il existe une sous-suite des itérés de l'opérateur de Koopman $U_f$ (voir définition \ref{Koop}) qui tend fortement vers l'identité.
\end{definition}

\begin{theoreme}[Katok, Stepin]\label{rigidgéné}
Un élément générique de $\mathrm{Homeo}(X,\mu)$ est rigide.
\end{theoreme}

Ce théorème se déduit du théorème \ref{approxmeo} et de la proposition suivante :

\begin{prop}[Katok, Stepin]\label{rigid}
Si un automorphisme admet une approximation à vitesse $\vartheta$, avec $\vartheta(q) = o(1/q)$, alors il est rigide.
\end{prop}

\begin{proof}[Preuve de la proposition \ref{rigid}] Soit $f$ un automorphisme admettant une approximation à vitesse $\vartheta$, avec $\vartheta(q) = o(1/q)$. Soient aussi $\varepsilon>0$ et $Q>0$ tels que pour $q>Q$, $\vartheta(q)\le\varepsilon/q$. Alors il existe une infinité d'entiers $q_{m_k}$ et de permutations cycliques $f_k$ qui vérifient :
\[\sum_{i=1}^{q_{m_k}} \mu (f (C_{k,i})\Delta f_k (C_{k,i}))\le\frac{\varepsilon}{q_{m_k}}.\]
Avec le lemme \ref{itéré}, on en déduit:
\[\sum_{i=1}^{q_{m_k}} \mu (f^{q_{m_k}} (C_{k,i})\Delta f_k^{q_{m_k}} (C_{k,i})) = \sum_{i=1}^{q_{m_k}} \mu (f^{q_{m_k}} (C_{k,i})\Delta C_{k,i})\le\varepsilon,\]
ce qui signifie que $d_{\mathit{faible}}(f^{q_{m_k}},Id)\le\varepsilon$ dès lors que le diamètre des $C_{k,i}$ est inférieur à $\varepsilon$.
\end{proof}

Comme on l'a dit un automorphisme rigide n'est pas fortement mélangeant, on a même mieux :

\begin{prop}\label{analog}
Soit $f$ un automorphisme rigide. Alors pour tout sous-ensemble mesurable non dense $E$ de $X$, il existe un sous-ensemble mesurable $F$ de $X$, avec $\mu(F)>0$ et une suite strictement croissante $(m_k)_{k\in\N}$ d'entiers tels que pour tout $k$ on ait $\mu(f^{m_k}(E)\cap F)=0$.
\end{prop}

\begin{rem}
La propriété \og pour tous mesurables $E$ et $F$, $\mu(E\cap f^m(F))>0$ à partir d'un certain rang \fg~constitue un véritable analogue ergodique de la propriété de mélange fort topologique ; la proposition \ref{analog} montre que pour un homéomorphisme générique, celle-ci est fausse, alors qu'elle est vraie si on remplace l'hypothèse \og $E$ et $F$ mesurables\fg~par \og $E$ et $F$ ouverts\fg.
\end{rem}

\begin{proof}[Preuve de la proposition \ref{analog}] Soit $\varepsilon>0$ tel que l'$\varepsilon$-voisinage de $E$ soit de mesure strictement inférieure à $1-\varepsilon$ (un tel $\varepsilon$ existe car $E$ n'est pas dense). On choisit une suite $(m_k)_{k\in\N}$ telle que $\sum_k d_{\mathit{faible}}(f^{m_k},Id)<\varepsilon$, et considère
\[F = \left(\bigcup_k f^{m_k}(E)\right)^\complement.\]
Alors $F$ est mesurable et de mesure strictement positive et vérifie, pour tout entier $k$, $\mu(f^{m_k}(E)\cap F) = 0$.
\end{proof}

\section{Première preuve du théorème d'Oxtoby-Ulam}

Comme seconde application du théorème \ref{approxmeo}, on se propose de démontrer le théorème d'Oxtoby-Ulam : 

\begin{theoreme}[Oxtoby-Ulam]\label{OxtoUlApprox}
Dans $\mathrm{Homeo}(X,\mu)$, l'ergodicité est générique.
\end{theoreme}

Comme le théorème \ref{MélApprox}, il découle aussitôt du théorème \ref{approxmeo}, ainsi que de la proposition suivante :

\begin{prop}[Katok, Stepin]\label{approxergo}
Si un automorphisme admet une approximation cyclique à une vitesse $\vartheta$, avec $\vartheta(q)=o(1/q)$, alors il est ergodique.
\end{prop}

La preuve de cette proposition est une adaptation de celle du théorème 2.1 de l'article de A. Katok et A. Stepin \cite{KatokS1}.

\begin{proof}[Preuve de la proposition \ref{approxergo}] Soient $f$ un automorphisme admettant une approximation cyclique à une vitesse $\vartheta$ et $\delta>0$. Supposons que $f$ admette une partition en deux ensembles invariants $F_1$ et $F_2$ de mesures non triviales. On suppose que $\mu(F_1)\ge\mu(F_2)=v$, en particulier $v\in]0,\frac 12]$. Alors pour $m$ assez grand, la subdivision dyadique $\D_m$ (constituée de $q_m$ cubes) est telle qu'il existe une partition de $X$ en deux sous-ensembles $F_1^m$ et $F_2^m$ constitués d'unions de cubes $C_{m,i}$ de $\D_m$, vérifiant $\mu(F_1\Delta F_1^m) + \mu(F_2\Delta F_2^m) \le \delta v$. On prend $k$ assez grand pour que $m_k$ vérifie cette propriété et qu'il existe une approximation cyclique $f_k$ de $f$ telle que $\mu (f( C_{k,i})\Delta f_k(C_{k,i}))\le\vartheta(q_{m_k})$.

Soient $\sigma\in\Sn_q$ une permutation cyclique et $E_1$, $E_2$ deux sous-ensembles de $\{1,\dots,q\}$. Pour tout $x\in \{1,\dots,q\}$, on a
\[\sum_{i=0}^{q-1} \chi_{E_2}(\sigma^i(x)) = |E_2|.\]
En sommant cette égalité sur $x\in E_1$, on en déduit que\footnote{Remarquons que cela traduit l'ergodicité de $\sigma$ sur $\{1,\dots,q\}$ muni de la mesure de probabilité uniforme.} :
\[\frac 1 q \sum_{i=0}^{q-1} \frac{|\sigma^i(E_1)\cap E_2|}q = \frac{|E_1|}q\frac{|E_2|}q .\]
En particulier il existe $l\in \{0,\dots,q-1\}$ tel que $\frac{|\sigma^l(E_1)\cap E_2|}{|E_2|} \ge \frac{|E_1|}q$.

On peut appliquer cette remarque aux cubes constituant l'ensemble $F_1^{m_k}$ et à la permutation dyadique $f_k$ (que l'on voit alors comme une permutation de l'ensemble des cubes). On a au moins $q_{m_k}(1-(1+\delta)v)$ cubes constituant $F_1^{m_k}$ parmi les $q_{m_k}$ cubes de $\D_{m_k}$, ainsi il existe $l\in \{0,\dots,q_{m_k}-1\}$ tel que
\[\frac{\mu\big(f_k^l(F_1^{m_k})\cap F_2^{m_k}\big)}{\mu(F_2^{m_k})} \ge 1-(1+\delta)v.\]
Autrement dit 
\[\mu\big(f_k^l(F_1^{m_k})\cap F_2^{m_k}\big) \ge \big(1-(1+\delta)v\big)(1-\delta)v \ge v\left(1-v -\delta+\delta^2 v\right),\]
et puisque $v\le \frac 12$,
\[\mu\big(f_k^l(F_1^{m_k})\cap F_2^{m_k}\big) \ge v\left(\frac 1 2 -\delta\right).\]
Prenons ${m_k}$ assez grand pour que $\vartheta(q_{m_k})<\frac{\delta v}{q_{m_k}}$ ; à l'aide du lemme \ref{itéré},
\begin{eqnarray*}
\mu\big(f^l(F_1)\cap F_2\big) & \ge & \mu\big(f_k^l(F_1^{m_k})\cap F_2^{m_k}\big)\\
                              &     & - \mu\big(f_k^l(F_1^{m_k})\Delta f^l(F_1)\big) - \mu\big(F_2\Delta F_2^{m_k}\big)\\
                              & \ge & v\left(\frac 1 2 -\delta\right) - \frac{\delta lv}{q_{m_k}} -\delta v\\
                              & \ge & v\left(\frac 1 2 -3\delta\right).
\end{eqnarray*}
Puisque $\delta$ est arbitraire, on peut choisir $\delta<\frac{1}{6}$ si bien que $\mu\big(f^l(F_1)\cap F_2\big)>0$, ce qui contredit le fait que $F_1$ et $F_2$ sont deux ensembles invariants disjoints.
\end{proof}

\begin{rem}
On peut maintenant faire un point sur les propriétés génériques des orbites des points. Le théorème d'Oxtoby-Ulam implique que pour un élément générique de $\mathrm{Homeo}(X,\mu)$, les points dont l'orbite est dense forment un $G_\delta$ dense de mesure pleine de l'espace de base. Mieux encore, l'orbite de $\mu$-presque tout point est uniformément distribuée dans l'espace de base (par généricité de l'ergodicité et application du théorème de Birkhoff). En utilisant de plus les résultats de la section~\ref{2.2}, on sait que l'ensemble des points périodiques forme une partie dense, maigre et de mesure nulle de $X$.
\end{rem}

\section{Généricité du mélange faible}

Nous énonçons maintenant un résultat de généricité du mélange faible, qui constitue le dernier résultat obtenu pour les automorphismes par P. Halmos et V. Rokhlin dans les années 40, et adapté au cas des homéomorphismes à l'aide de l'approximation en topologie faible. Comme pour la preuve du mélange faible topologique (théorème \ref{mélange topo faible}), l'approximation par des permutations cycliques s'avère insuffisante : celle-ci sert plutôt à montrer que l'on a une dynamique \og quasi-périodique\fg et donc à interdire la présence de phénomènes de mélange (comme on l'a vu au théorème \ref{MélApprox}). Ici on lui ajoute l'approximation par des permutations bicycliques ayant deux cycles de longueurs $p$ et $p+1$ (voir \cite{KatokS3}). Rappelons la définition du mélange faible :

\begin{definition}
Un automorphisme $f$ est dit \emph{faiblement mélangeant} si pour tout couple d'ouverts $U,V$,
\[\frac{1}{N}\sum_{i=0}^{N-1}\big|\mu(f^i(U)\cap V)-\mu(U)\mu(V)\big|\underset{N\to+\infty}{\longrightarrow}0.\]
\end{definition}

On a alors le théorème de généricité :

\begin{theoreme}[Katok, Stepin]\label{mélfaiblgéné}
Dans $\mathrm{Homeo}(X,\mu)$, le mélange faible est générique.
\end{theoreme}

Dans leur article \cite{KatokS3}, A. Katok et A. Stepin montrent que tout automorphisme ergodique bien approché par des permutations ayant deux orbites de longueurs $p$ et $p+1$ est mélangeant. Le théorème \ref{mélfaiblgéné} résulte alors facilement d'un théorème de généricité d'approximation à vitesse fixée (par des permutations ayant deux orbites de longueurs $p$ et $p+1$) et du théorème d'Oxtoby-Ulam (théorème \ref{OxtoUlApprox}). Nous montrerons le théorème \ref{mélfaiblgéné} dans la partie \ref{partiefaible}, comme application du théorème de transfert.

\section{Généricité de l'entropie métrique nulle}

Montrons maintenant la généricité des homéomorphismes ayant une entropie métrique nulle. Cette preuve est issue de la démonstration esquissée par A. Katok et A. Stepin dans \cite{KatokS2}. On commence par quelques définitions concernant l'entropie métrique ; les preuves détaillées sont faites à la page 74 de \cite{12'}.
\newline

Soit $\alpha = \{A_1,\dots,A_q\}$ une partition finie de $X$ en sous-ensembles mesurables. L'\emph{entropie} de $\alpha$ est la quantité
\[H(\alpha) = -\sum_{i=1}^q \mu(A_i)\log\mu(A_i).\]
À l'aide de la concavité du logarithme, on obtient $0\le H(\alpha)\le \log(q)$, la deuxième inégalité étant une égalité si $\mu(A_i) = 1/q$ pour tout $i$. \'Etant donné un ensemble fini $\{\alpha_p\}_{1\le p \le k}$ de partitons finies de $X$, la \emph{partition jointe} des $\{\alpha_p\}$, notée $\bigvee_{1\le p \le k} \alpha_p$, est la partition finie de $X$ formée des $\bigcap_{1\le p\le k} A_{p,i_p}$, avec $A_{p,i_p}\in \alpha_p$ pour tout~$p$. L'entropie de la partition jointe vérifie l'inégalité suivante : pour $\alpha$ et~$\beta$ deux partitions mesurables, $H(\alpha \vee \beta) \le H(\alpha) + H(\beta)$. Ainsi pour tout $f\in \mathrm{Auto}(X,\mu)$, la suite $(a_m)_m$ définie par $a_m = H \left( \bigvee_{p=0}^{k-1} f^{-p}(\alpha) \right)$ est sous-additive. On en déduit que la suite $\left\{\frac{a_m}m\right\}_m$ admet une limite, qu'on appelle \emph{entropie} de la transformation $f$ relativement à la partition $\alpha$ :
\[H(f,\alpha) = \lim_{k \to + \infty} \frac{1}{k} H\left(\bigvee_{p=0}^{k-1} f^{-p}(\alpha)\right).\]
Ceci permet de définir l'\emph{entropie métrique} de $f$ comme la borne supérieure des entropies de $f$ relativement aux partitions mesurables finies de $X$ : 
\[H(f) = \sup_\alpha H(f,\alpha).\]

Maintenant que l'on a défini proprement l'entropie métrique, on peut énoncer le théorème principal de cette section, qui concerne la généricité de l'entropie métrique nulle. Elle a tout d'abord été obtenue par V. Rokhlin dans le cas des automorphismes en 1959 \cite{RokhEntr}, puis dans celui des homéomorphismes par A. Katok et A. Stepin \cite{KatokS2}, \cite{KatokS3}.

\begin{theoreme}[Rokhlin, Katok, Stepin]\label{entrmetnullgéné}
Les éléments d'entropie métrique nulle sont génériques parmi $\mathrm{Homeo}(X,\mu)$.
\end{theoreme}

Ce théorème découle tout de suite (on commence à en avoir l'habitude) du théorème \ref{approxmeo} et de la proposition suivante :

\begin{prop}\label{entremet}
Si $f\in \mathrm{Auto}(X,\mu)$ admet une approximation cyclique à la vitesse $\vartheta$, avec $\vartheta(q)=o(1/\log^2(q))$, alors l'entropie métrique de~$f$ est nulle.
\end{prop}

\begin{proof}[Preuve de la proposition \ref{entremet}] Puisque $f$ admet une approximation cyclique à la vitesse $\vartheta$, la définition \ref{defappcycl} fournit un entier $m_k$ et une permutation cyclique d'ordre $m_k$ notée $f_k$ tels que si on pose $\mu_k = \sum_{i=1}^{q_{m_k}} \mu \big(f( C_{k,i})\Delta f_k(C_{k,i})\big)$, on obtient $\mu_k\le \vartheta(q_{m_k})$. Commençons par majorer les entropies associées aux partitions $\D_{m_k}$. La définition même de l'entropie comme limite d'une suite sous-additive donne, pour tout entier $l$ :
\[0\le H(f,\D_{m_k})\le \frac{H\left(\D_{m_k}\vee f(\D_{m_k})\vee\dots\vee f^{l}(\D_{m_k})\right)}{l}.\]
La preuve va consister à montrer que l'entropie de la subdivision $\D_{m_k}\vee f(\D_{m_k})\vee\dots\vee f^l(\D_{m_k})$ est proche de celle de $\D_{m_k}\vee f_k(\D_{m_k})\vee\dots\vee f^l_k(\D_{m_k}) = \D_{m_k}$, qui vaut $\log(q_{m_k})$ pour tout $l$ (comme $f_k$ est une permutation des cubes de $\D_{m_k}$, la partition $\D_{m_k}\vee f_k(\D_{m_k})\vee\dots\vee f^l_k(\D_{m_k})$ est constituée de $q_{m_k}$ cubes $C_{k,i}$ de mesures identiques, et de $q_{m_k}^l-q_{m_k}$ ensembles vides). Un peu plus précisément, nous allons montrer que, puisque les $f_k^j(C_{k,i})$ sont proches des $f^j(C_{k,i})$, la partition $\D_{m_k}\vee f(\D_{m_k})\vee\dots\vee f^l(\D_{m_k})$, vue comme perturbation de la partition précédente, est elle constituée de $q_{m_k}$ ensembles de mesures \og presque \fg~$1/q_{m_k}$, et d'autres ensembles de mesures \og petites \fg. Mettons tout cela en forme.

Il y a une bijection entre les ensembles de la partition $\D_{m_k}\vee f_k(\D_{m_k})\vee\dots\vee f_k^l(\D_{m_k})$ et ceux de la partition $\D_{m_k}\vee f(\D_{m_k})\vee\dots\vee f^l(\D_{m_k})$, donnée par :
\[C_{k,i_1}\cap f_k(C_{k,i_2})\cap\dots\cap f_k^l(C_{k,i_l}) \longmapsto C_{k,i_1}\cap f(C_{k,i_2})\cap\dots\cap f^l(C_{k,i_l}).\]
Si les ensembles $C_{k,i_1},\, f_k(C_{k,i_2}),\,\dots,\, f_k^{k}(C_{k,i_k})$ sont tous égaux, on aura 
\[\mu\big(C_{k,i_1}\cap f_k(C_{k,i_2})\cap\dots\cap f_k^l(C_{k,i_l})\big) = \frac 1{q_{m_k}}\]
et sinon (si on a au moins deux ensembles non identiques, donc essentiellement disjoints puisque $f_k$ est une permutation) on aura 
\[\mu\big(C_{k,i_1}\cap f_k(C_{k,i_2})\cap\dots\cap f_k^l(C_{k,i_l})\big) = 0.\]
La partition $\D_{m_k}\vee f(\D_{m_k})\vee\dots\vee f^l(\D_{m_k})$ est ainsi formée de :
\begin{enumerate}
\item $q_{m_k}$ ensembles $C_{k,i_1}\cap f(C_{k,i_2})\cap\dots\cap f^l(C_{k,i_l})$, associés aux $C_{k,i_1}\cap f_k(C_{k,i_2})\cap\dots\cap f_k^l(C_{k,i_l})$ avec des ensembles tous égaux ; ils sont donc de mesures inférieures à $q_{m_k}$,
\item $q_{m_k}^l-q_{m_k}\le q_{m_k}^l$ ensembles $C_{k,i_1}\cap f(C_{k,i_2})\cap\dots\cap f^l(C_{k,i_l})$, associés aux $C_{k,i_1}\cap f_k(C_{k,i_2})\cap\dots\cap f_k^l(C_{k,i_l})$ avec des ensembles non tous égaux, de mesures petites que nous majorons maintenant.
\end{enumerate}

Soit $\{A_\alpha\}_\alpha$ la famille des ensembles du cas 2. Chaque $A_\alpha$ est l'intersection d'au moins un ensemble $C_{k,i}$ avec un ensemble $f^p(C_{k,j})$ où $p\le l$ ; en combinant la remarque précédente et le fait que $C_{k,i} \cap f^p_k(C_{k,j})  = \emptyset$ (on est dans le cas 2.), on obtient $\sum_\alpha \mu(A_\alpha)\le l \mu_k$. Puisque l'entropie d'une subdivision est maximale pour des ensembles ayant tous la même mesure, on en déduit que
\[-\sum_\alpha \mu(A_\alpha)\log(\mu(A_\alpha))\le -l\mu_k\log\left(\frac{l\mu_k}{q_{m_k}^l}\right).\]
Et puisque seuls les ensembles du cas 2. ont une contribution positive à la différence d'entropie entre $f$ et $f_k$, on obtient
\[H\left(\D_{m_k}\vee\dots\vee f^l(\D_{m_k})\right)- H\left(\D_{m_k}\vee\dots\vee f_k^l(\D_{m_k})\right)\le -l\mu_k\log\left(\frac{l\mu_k}{q_{m_k}^l}\right),\]
et donc
\[\frac{H\left(\D_{m_k}\vee f(\D_{m_k})\vee\dots\vee f^l(\D_{m_k})\right)}{l} \le \frac{\log(q_{m_k})}l + \mu_k\log\left(\frac{q_{m_k}^l}{l\mu_k}\right).\]
Cette majoration obtenue, il reste à optimiser $l$ pour que le membre de droite tende vers 0. Soit $\varepsilon\in]0,1[$, on choisit $q_{m_k}$ assez grand tel que $\vartheta(q_{m_k})<\varepsilon/\log^2(q_{m_k})$. Prenons $l = \lfloor\log(q_{m_k})/\sqrt\varepsilon\rfloor$, on a alors 
\begin{eqnarray*}
H(f,\D_{m_k}) & \le & \frac{\log(q_{m_k})}l + \frac\varepsilon{\log^2(q_{m_k})}\left(l\log(q_{m_k})-\log(l)+\log\left(\frac{\log^2(q_{m_k})}\varepsilon\right)\right)\\
          & \le & \frac{\sqrt\varepsilon}{1-\frac{\sqrt\varepsilon}{\log q_{m_k}}} + \sqrt\varepsilon+0+\frac{\log(\log^2(q_{m_k}))}{\log^2(q_{m_k})}-\frac{\varepsilon\log(\varepsilon)}{\log^2(q_{m_k})}\\
          & \le & \sqrt\varepsilon \left(2+ \frac{\sqrt\varepsilon}{\log q_{m_k}}\right)+ \frac{\log(\log^2(q_{m_k}))}{\log^2(q_{m_k})}+\frac{e^{-1}}{\log^2(q_{m_k})}.
\end{eqnarray*}
Ainsi, $\overline\lim_{k\to\infty} H(f,\D_{m_k})\le 2\sqrt\varepsilon$, d'où $\lim_{k\to\infty} H(f,\D_{m_k})=0$.

Or on sait que $H(f) = \lim_{k\to\infty} H(f,\D_{m_k})$, puisque l'ensemble des cubes des subdivisions $\D_{m_k}$ forme une base de l'ensemble des boréliens de $X$ (voir par exemple le théorème 3.7.9 de \cite{12'}). Finalement $H(f) = 0$.
\end{proof}

\section{Généricité des homéomorphismes de rang un}

Dans cette partie, nous nous concentrons sur les homéomorphismes de rang un. De manière informelle, un automorphisme $f$ est de rang un si toute partition mesurable finie de $X$ peut être approchée par une partition engendrée par des itérés par $f$ deux à deux disjoints d'une seule partie $A$ de $X$. La généricité des homéomorphismes de rang un sera bien entendu prouvée par la technique des approximations périodiques ; la preuve que nous présentons est en fait très proche de celle utilisée par A. Katok et A. Stepin en 1967 dans \cite{KatokS1} pour montrer qu'un automorphisme approché assez vite par des permutations cycliques est à spectre simple. Ce dernier résultat avait été obtenu (avec une vitesse d'approximation bien plus contraignante) par S. Yuzvinskii en 1967 dans \cite{Yuz}. De notre côté nous l'établirons comme corollaire de la généricité des homéomorphismes de rang un d'une part, et du fait qu'un automorphisme de rang un est à spectre simple d'autre part\footnote{Voir la remarque \ref{attention au rang}.}. Ce n'est pas la seule propriété notable des automorphismes de rang un : ils ont aussi une entropie métrique nulle\footnote{La preuve de cette implication est similaire à la preuve directe de la généricité de l'entropie métrique nulle, que l'on a déjà faite.}, et sont standards (un automorphisme $f$ est dit \emph{standard} s'il existe un ensemble $A$ tel que l'automorphisme induit par $f$ sur $A$ soit mesurablement conjugué à un odomètre), comme prouvé dans \cite{KatokMon}. Remarquons que la généricité du spectre simple, outre son intérêt propre, va nous permettre dans la partie suivante de donner une définition relativement simple du type spectral.

\begin{definition}\label{defrg}
On dira qu'un automorphisme $f$ est \emph{de rang un}\footnote{Voir la remarque \ref{attention au rang}.} si pour toute partition mesurable finie $P_1,\dots,P_s$ de $X$ et pour tout $\varepsilon>0$, on peut trouver un ensemble mesurable $A$ et un entier $\tau$ tels que :
\begin{enumerate}
\item les ensembles $A,f(A),\dots,f^{\tau-1}(A)$ sont deux à deux disjoints,
\item on a l'inégalité $\tau\,\mu(A\cap f^\tau(A))>1-\varepsilon$ (les ensembles $A$ et $f^\tau(A)$ coïncident presque),
\item si on note $\eta$ la partition finie de $X$
\[\left\{A,f(A),\dots,f^{\tau-1}(A), X\setminus\left(\bigcup_{i=1}^{\tau-1} f^i(A)\right)\right\},\]
il existe des parties $P_1',\dots,P_s'$ de $X$ telles que, pour tout $i\in \{1,\dots,s\}$, la partie $P_i'$ est une union d'éléments de la partition $\eta$ et satisfait $\mu(P_i\Delta P_i')\leq\varepsilon$.
\end{enumerate}
\end{definition}

\begin{rem}\label{attention au rang}
La définition classique du rang un n'inclut que les points 1. et 3. de la définition. Néamoins, nous aurons besoin de la propriété 2. lors des applications du théorème de généricité du rang un ; il nous a semblé plus simple de l'inclure directement dans la définition.
\end{rem}

\begin{rem}\label{remran}
Puisque les unions de cubes dyadiques de $X$ forment un sous-ensemble dense de l'ensemble des parties boréliennes de $X$, il suffit de vérifier la propriété 2. de la définition pour des partitions finies $P_1,\dots,P_m$ dont les éléments sont des unions de cubes dyadiques.
\end{rem}

\begin{theoreme}[Katok, Stepin]\label{rang1F}
Le rang un est générique parmi les éléments de $\mathrm{Homeo}(X,\mu)$.
\end{theoreme}

Ce résultat découle immédiatement du théorème \ref{approxmeo} et de la proposition suivante :

\begin{prop}\label{rang1}
Soit $f\in \mathrm{Auto}(X,\mu)$ admettant une approximation cyclique à vitesse $\vartheta$, avec $\vartheta(q) = o(1/q)$. Alors $f$ est de rang un, autrement dit pour toute partition mesurable finie $P_1,\dots,P_s$ de $X$ et pour tout $\varepsilon>0$, on peut trouver un ensemble mesurable $A$ et un entier $\tau$ vérifiant les trois points de la définition \ref{defrg}. De plus, on peut supposer que l'ensemble $A$ est ouvert (resp. fermé).
\end{prop}

Comme nous l'avons déjà dit, nous suivons la technique de preuve du théorème~3.1 de \cite{KatokS1}, même si le théorème que A. Katok et A. Stepin y démontrent est quelque peu différent. Un véritable énoncé de la généricité du rang un se trouve dans \cite{KatokMon}.

\begin{proof}[Preuve de la proposition \ref{rang1}] Soient $f$ un automorphisme vérifiant les hypothèses de la proposition et $\varepsilon>0$. Considérons un entier $\ell$ et une partition finie mesurable $\{P_1,\dots, P_s\}$ de $X$ ; d'après la remarque \ref{remran}, on peut supposer que ses éléments sont des unions de cubes dyadiques d'ordre $\ell$. Alors il existe un entier $m\ge\ell$ tel que la subdivision $\D_{m}$ soit de taille $q_m$, ainsi qu'une permutation $f_m$ de $\D_{m}$, tels que
\[\sum_{i=1}^{q_m}\mu\big(f(C_{m,i})\Delta f_m(C_{m,i})\big)\le\vartheta(q_m).\]
On choisit de plus l'entier $m$ assez grand, de telle sorte que $\vartheta(q_m)<\frac{\varepsilon}{3 q_m}$. Soit $C$ un cube (ouvert ou fermé) de $\D_{m}$, alors par cyclicité de la permutation $f_m$, la famille de cubes $\{f_m^i(C)\}_{0\le i\le q_m-1}$ parcourt exactement la subdivision $\D_{m}$. Posons
\[A = \bigcap_{i=0}^{q_m-1} f^{-i}(f_m^i(C)) = \bigcap_{i=0}^{q_m-1}\widetilde{A}_i.\]

Alors $f^i(A)\subset f_m^i(C)$ pour tout entier $i\in\{0,\dots,q_m-1\}$, si bien que les ensembles $A, f(A), \dots, f^{q_m-1}(A)$ sont deux à deux disjoints. L'ensemble $A$ vérifie donc le point 1. de la définition \ref{defrg}.

De plus, $A$ remplit une part importante de $C$, en effet,
\[C = \widetilde{A}_0\subset \left(\bigcap_{i=0}^{q_m-1} \widetilde{A}_i\right)\cup \bigcup_{i=1}^{q_m-1}\left(\widetilde{A}_{i-1}\setminus \widetilde{A}_i\right)\]
(pour $x\in \widetilde{A}_0$, considérer s'il existe le premier indice $i$ tel que $x\in \widetilde{A}_{i-1}$ et $x\notin \widetilde{A}_i$), si bien que
\begin{eqnarray*}
\mu(C)-\mu(A) & \le & \sum_{i=1}^{q_m-1}\mu (\widetilde{A}_{i-1}\setminus \widetilde{A}_i) = \frac 12\sum_{i=1}^{q_m-1}\mu (\widetilde{A}_{i-1}\Delta \widetilde{A}_i)\\
              & =   & \frac 12 \sum_{i=1}^{q_m-1}\mu\left(f(f_m^{i-1}(C))\Delta f_m(f_m^{i-1}(C))\right) \le \frac{\vartheta(q_m)}{2}.
\end{eqnarray*}
Puisque $f^i(A)\subset f_m^i(C)$ pour $i\in\{0,\dots,q_m-1\}$, et par préservation de la mesure, on a alors
\begin{equation}\label{eq0}
\mu\big(f_m^i(C) \Delta f^i(A)\big) = \mu(C)-\mu(A)\le \frac{\vartheta(q_m)}{2}\le\frac{\varepsilon}{6q_m},
\end{equation}
et par conséquent, 
\begin{eqnarray*}
\mu\left(C\Delta f^{q_m}(A)\right) & = & \mu\big(f_m^{q_m}(C) \Delta f^{q_m}(A)\big)\\
                             & \le & \mu\left(f_mf_m^{q_m-1}(C) \Delta f f_m^{q_m-1}(C)\right) + \mu\left(ff_m^{q_m-1}(C) \Delta ff^{q_m-1}(A)\right)\\
                             & \le & \frac{\varepsilon}{3q_m} + \frac{\varepsilon}{6q_m} \le \frac{3\varepsilon}{6q_m},
\end{eqnarray*}
d'où
\begin{eqnarray*}
\mu\big(A \Delta f^{q_m}(A)\big) & \le & \mu(A\Delta C) + \mu\big(C \Delta f^{q_m}(A)\big)\\
                                 & \le & \frac{4\varepsilon}{6q_m}.
\end{eqnarray*}
Ainsi,
\[\mu\big(A \cap f^{q_m}(A)\big) \ge \frac{\mu(A)+\mu(f^{q_m}(A))}{2}-\mu\big(A \Delta f^{q_m}(A)\big) \ge \frac{1}{q_m}-\frac{5\varepsilon}{6q_m}\ge \frac{1-\varepsilon}{q_m},\]
ce qui prouve que l'ensemble $A$ vérifie point 2. de la définition \ref{defrg}.

Enfin, pour tout entier $k\in\{1,\dots,\ell\}$, il existe des entiers $i_j$ tels que $P_k = \bigcup_j f_m^{i_j}(C)$ ; l'application de l'inégalité (\ref{eq0}) donne :
\begin{equation}\label{eq0'}
\mu\left(P_k \Delta \left(\bigcup_j f^{i_j}(A)\right)\right)\le \sum_j \mu(f_m^{i_j}(C) \Delta f^{i_j}(A)) \le \frac{\varepsilon}{6},
\end{equation}
ce qui montre que $A$ satisfait le point 3. de la définition \ref{defrg} (à l'aide de la remarque \ref{remran}). Pour $\varepsilon$ fixé, on vient de construire un ensemble mesurable $A$ vérifiant les assertions 1. et 2. et 3. de la définition \ref{defrg}. Ainsi l'automorphisme $f$ est de rang un et de plus, quitte à choisir l'ensemble $C$ ouvert (resp. fermé) dès le début, on peut supposer l'ensemble $A$ ouvert (resp. fermé).
\end{proof}

\section{Généricité des homéomorphismes à spectre simple}

Dans cette partie, nous utilisons le fait qu'un élément générique de $\mathrm{Homeo}(X,\mu)$ est de rang un, établi dans la section précédente, pour en déduire la généricité des homéomorphismes à spectre simple. Notons qu'historiquement, la preuve de la généricité du spectre simple est antérieure à celle du rang un ; l'idée d'automorphisme de rang un était néanmoins sous-jacente dans la preuve de l'article de A. Katok et A. Stepin \cite{KatokS1} ; nous avons découpée celle-ci en deux parties indépendantes : une première consacrée à la généricité des homéomorphismes de rang un, et une seconde où l'on prouve que le rang un implique la simplicité du spectre.

\begin{definition}\label{Koop}
L'\emph{opérateur de Koopman} associé à $f\in \mathrm{Auto}(X,\mu)$ est l'application
\begin{eqnarray*}
U_f : L^2(X,\mu) &  \longrightarrow & L^2(X,\mu)\\
\varphi & \longmapsto & \varphi\circ f.
\end{eqnarray*}
\end{definition}

\begin{definition}\label{defspecsimpl}
On dit que $f\in \mathrm{Auto}(X,\mu)$ est à \emph{spectre simple} si son opérateur de Koopman associé $U_f$ est cyclique, c'est-à-dire s'il existe $\chi\in L^2(X,\mu)$ tel que $\{U_f^k(\chi)\}_{k\in \N}$ forme une base (de Schauder) de $L^2(X,\mu)$.
\end{definition}

\begin{rem}
Pour $f$ un automorphisme à spectre simple, toutes les valeurs propres de $U_f$ sont simples. Réciproquement, si $U_f$ est diagonalisable (c'est-à-dire s'il existe une base hilbertienne de $L^2(X,\mu)$ constituée de vecteurs propres de $U_f$) avec toutes ses valeurs propres simples, alors $f$ est à spectre simple.
\end{rem}

\begin{theoreme}[Yuzvinskii, Katok, Stepin]\label{specsimplgéné}
Un homéomorphisme générique de $\mathrm{Homeo}(X,\mu)$ est à spectre simple.
\end{theoreme}

Ce théorème se déduit du théorème \ref{rang1F} et de la proposition suivante :

\begin{prop}\label{specsimpl}
Tout automorphisme de rang un est à spectre simple.
\end{prop}

Combinée à la proposition \ref{rang1}, cette proposition implique facilement le corollaire suivant :

\begin{coro}
Si $f\in \mathrm{Auto}(X,\mu)$ admet une approximation cyclique à la vitesse $\vartheta$, avec $\vartheta(q) = o(1/q)$, alors $f$ est à spectre simple.
\end{coro}

\begin{rem}
En travaillant directement sur la propriété de simplicité du spectre (sans passer par les automorphismes de rang un), comme dans \cite{KatokS3}, on peut obtenir une vitesse d'approximation moins contraignante, à savoir $\vartheta(q) = \alpha/q$, avec $\alpha<\frac 12$.
\end{rem}

Pour montrer la proposition \ref{specsimpl}, nous aurons besoin d'un lemme de théorie spectrale, dont on pourra trouver un énoncé plus général, ainsi qu'une démonstration, au théorème 1.21 de \cite{KatokT}.

\begin{lemme}\label{thspec}
Soit $U$ un opérateur normal et borné d'un espace de Hilbert $\Hi$. Si cet opérateur n'est pas cyclique, alors il existe deux vecteurs orthogonaux $\varphi,\psi\in\Hi$ de norme $1$ tels que pour tout sous espace $U$-cyclique $\mathcal K$, on ait
$d(\varphi,\mathcal K)+d(\psi,\mathcal K)\ge 1$
\end{lemme}

\begin{proof}[Preuve de la proposition \ref{specsimpl}] Soit $f$ un automorphisme de rang un et $\varphi, \psi\in L^2(X,\mu)$ qui vérifient $\|\varphi\|\le 1$, $\|\psi\|\le 1$ et $\delta>0$. Pour $m$ assez grand, il existe $\varphi_m, \psi_m\in L^2(X,\mu)$, constantes sur les cubes de la subdivision dyadique $\D_m$ telles que $\|\varphi_m\|\le 1$, $\|\psi_m\|\le 1$, $\|\varphi-\varphi_m\|\le \delta$ et $\|\psi-\psi_m\|\le \delta$. Puisque $\varphi$ et $\psi$ jouent des rôles identiques, on se concentre sur l'étude de $\varphi$ ; les conclusions seront identiques pour $\psi$. On note $C_1,\dots,C_{q_m}$ les cubes de $\D_m$. Choisissant $\varepsilon=\frac{1}{3 q_m}$, puisque $f$ est de rang un, il existe un ensemble mesurable $A$ et un entier $\tau$ tels que :
\begin{itemize}
\item les ensembles $A,f(A),\dots,f^{\tau-1}(A)$ sont deux à deux disjoints,
\item la partition finie de $X$ engendrée par les ensembles
\[A,f(A),\dots,f^{\tau-1}(A), X\setminus\bigcup_{k\leq \tau-1} f^k(A)\]
contient des éléments $P_1',\dots,P_{q_m}'$ tels que pour tout entier $i$ compris entre 1 et $q_m$, $\mu(C_i\Delta P_i')\le \varepsilon$.
\end{itemize}

Posons $\chi_{C_i}$ et $\chi_A$ les fonctions caractéristiques des ensembles $C_i$ et $A$, vues comme éléments de $L^2(X,\mu)$. Par hypothèse, il existe des nombres $\{b_i\}_{0\le i\le q_m-1}$ tels que
\[\varphi_m = \sum_{i=0}^{q_m-1} b_i\ \chi_{C_i}.\]
L'approximation naturelle de $\varphi_m$ par des fonctions engendrées par les itérés de $A$ par $U_T$ est alors :
\[\varphi'_m = \sum_{i=0}^{q_m-1} b_i\ \chi_{P_i'}.\]
Alors $\|\varphi-\varphi'_m\|\le \delta +\|\varphi_m-\varphi'_m\|$ et par inégalité triangulaire, 
\begin{eqnarray*}
\|\varphi_m-\varphi'_m\| & = & \sum_{i=0}^{q_m-1} |b_i| \left\|\chi_{C_i} - \chi_{P_i'}\right\| \\
                         & \le & \sum_{i=0}^{q_m-1} |b_i|\,\mu\left(C_i \Delta P_i'\right).
\end{eqnarray*}
En observant que de plus, par inégalité de Cauchy-Schwarz et par hypothèse sur la norme de $\varphi_m$, on a :
\[\sum_{i=0}^{q_m-1}|b_i|\le\sqrt{q_m}\sqrt{\sum_{i=0}^{q_m-1}|b_i|^2}\le q_m,\]
en déduit que 
\[\|\varphi_m-\varphi'_m\| \le q_m^{3/2}\varepsilon\le \frac 13.\]

Supposons que le spectre de $f$ ne soit pas simple. On applique alors le lemme \ref{thspec} en choisissant pour $\mathcal K$ le sous-espace engendré par $\chi_A$, qui contient donc tous les $\chi_{P_i'}$ ; ce lemme nous donne deux fonctions particulières $\varphi$ et $\psi$. Si on prend $m$ assez grand on aura $d(\varphi,\mathcal K)\le \frac 13+\delta$ et $d(\psi,\mathcal K)\le \frac 13+\delta$ d'où, en prenant $\delta<\frac 16$, 
\[d(\varphi,\mathcal K) + d(\psi,\mathcal K) < 1,\]
ce qui est en contradiction avec le lemme \ref{thspec}.
\end{proof}

\section{Généricité des homéomorphismes standards}

Nous donnons dans cette section une seconde conséquence du fait d'être de rang un : être standard. Un automorphisme $f$ est dit \emph{standard} s'il existe un ensemble mesurable $A$ tel que l'application induite $f_A$ soit métriquement conjuguée à un odomètre. Notons que si $f$ est standard, alors l'application $f_A$ est en particulier à spectre discret, tandis que le type spectral d'un homéomorphisme générique est sans atome (voir la section \ref{saltyp}). Ainsi, il ne faut pas confondre les comportements dynamiques de $f$ et $f_A$, qui sont en général bien différents. Les preuves de cette partie sont principalement tirées de \cite{KatokMon}.

\begin{definition}\label{Enzo}
Soit $\{r_m\}_{m\in\N}$ une suite d'entiers strictement positifs. On définit $f\in\mathrm{Auto}(I,\mathrm{Leb})$ par récurrence comme suit :
\begin{itemize}
\item On découpe $I$ en $r_0$ intervalles $I_{0}^1, \dots, I_{0}^{r_0}$ de longueurs identiques, sur les $r_0-1$ premiers intervalles, $f$ agit comme une translation envoyant chaque intervalle sur le suivant. Reste à définir $f$ sur le dernier intervalle $I_{0}^{r_0}$ ; puisqu'on veut que $f$ préserve la mesure, ce dernier intervalle sera envoyé sur le premier intervalle $I^1_0$.
\item On découpe l'intervalle restant $I_{0}^{r_0}$ en $r_1$ petits intervalles $I_{1}^1, \dots, I_{1}^{r_1}$ de longueurs identiques, on fait de même en découpant l'intervalle $I_{0}^1$ en $\tilde{I}_{1}^1, \dots, \tilde{I}_{1}^{r_1}$. Sur les $r_1-1$ premiers sous-intervalles $I_{1}^1, \dots, I_{1}^{r_1-1}$, $f$ agit comme une translation envoyant chaque petit intervalle $I_{1}^i$ sur  $\tilde{I}_{1}^{i+1}$. Reste à définir $f$ sur $I_1^{r_1}$ ; puisque $f$ préserve la mesure, ce dernier intervalle sera envoyé sur $\tilde{I}_1^1$.
\item On répète cette opération à l'infini.
\end{itemize}
Un \emph{odomètre} associé à la suite $\{r_m\}$ est un élément de $\mathrm{Auto}(X,\mu)$ métriquement conjugué à l'automorphisme $f$ construit ci-dessus. 

Un automorphisme $f\in\mathrm{Auto}(X,\mu)$ est dit \emph{standard} s'il existe un ensemble mesurable $A$ tel que l'application induite $f_A : A\to A$ soit un odomètre.
\end{definition}

Nous commençons cette partie par la preuve d'un \emph{addendum} à la proposition \ref{rang1}, qui se montre directement à partir de cette dernière.

\begin{prop}\label{rang2}
Soit $f\in \mathrm{Auto}(X,\mu)$ un automorphisme de rang un. Alors il existe une suite d'ensembles mesurables $(A_m)_{m\in\N}$ et une suite d'entiers $(\tau_m)_{m\in\N}$, vérifiant :
\begin{enumerate}
\item Pour tout $m$, les ensembles $A_m,f(A_m),\dots,f^{\tau_m-1}(A_m)$ sont deux à deux disjoints. On notera $B_m = X\setminus \bigcup_{i=0}^{\tau_m-1} f^i(A_m)$ si bien que les ensembles $A_m,f(A_m),\dots,f^{\tau_m-1}(A_m),B_m$ forment une partition finie $\eta_m$ de $X$.
\item $\tau_m\, \mu(f^{\tau_m}(A_m)\cap A_m)\underset{m\to+\infty}{\longrightarrow} 1$.
\item Pour toute partition finie mesurable $P_1,\dots,P_k$ de $X$, et pour tout $\varepsilon>0$, il existe $m\in\N$ et des ensembles $P_1',\dots,P_k'$ formés d'unions d'ensembles de $\eta_m$ tels que, pour tout entier $i$ compris entre 1 et $k$, $\mu(P_i\Delta P_i')\le \varepsilon$.
\item Les partitions $\eta_m$ sont de plus en plus fines : pour tout $m'\ge m$, chaque élément de $\eta_m$ est l'union d'élements de $\eta_{m'}$ ; de plus $B_1\supset B_2\supset\dots$.
\end{enumerate}
\end{prop}

\begin{proof}[Preuve de la proposition \ref{rang2}] Par définition, il existe de deux suites $(A^m)_m$ et $(\tau_m)_m$ qui vérifient les points 1., 2. et 3. de la proposition. Il reste donc à montrer que l'on peut modifier ces ensembles de manière à obtenir le point 4. On suppose que l'on a construit une suite d'ensembles mesurables $(A_m)_{m\in\N}$ vérifiant les trois premiers points de la proposition ; on veut modifier ces ensembles pour qu'ils vérifient aussi le dernier. Nous allons maintenant faire une suite de modifications sur la suite d'ensembles $(A^m)_m$ pour obtenir le point~4. Dans ce but, commençons par poser $A_m^m= A^m$. Quitte à extraire une sous-suite de la suite $(A_m^m)$, on pourra supposer que les partitions
\[\eta_m^m = \left\{A_m^m, f(A_m^m),\dots,f^{\tau_m-1}(A_m^m),B_m^m\right\}\]
sont telles que pour tout entier $m$, il existe une famille d'entiers $I_m\subset \{1,\dots,\tau_{m+1}-1\}$ pour laquelle
\begin{equation}\label{eq1}
\mu\left(A_m^m\Delta\left(\bigcup_{i\in I_m} f^i(A_{m+1}^{m+1})\right)\right)\le\frac{1}{\tau_m^3 2^{\tau_m}}
\end{equation}
et, pour tout $i\in I_m$,
\begin{equation}\label{eq2}
\mu\left(f^i(A_{m+1}^{m+1})\cap A_m^m\right)\ge\frac{3}{4\tau_{m+1}}.
\end{equation}
De manière moins formelle, quand $i$ parcourt $I_m$, les ensembles $f^i(A_{m+1}^{m+1})$ sont tous \og presque \fg~dans $A_m^m$ et en plus ils remplissent bien cet ensemble. L'existence d'une telle suite de partitions résulte de la propriété 3. de la proposition, qui a déjà été prouvée.

De l'inégalité (\ref{eq2}) et du fait que les $f^j(A_m^m)$ sont deux à deux disjoints pour $j\in\{1,\dots,\tau_m-1\}$, on déduit immédiatement que, pour $i\in I_m$, 
\[\mu\left(f^{i+j}(A_{m+1}^{m+1})\cap A_m^m\right) = \mu\left(f^i(A_{m+1}^{m+1})\cap f^{-j}(A_m^m)\right)\le\frac{1}{4\tau_{m+1}},\]
si bien qu'alors
\begin{equation}\label{eq3}
I_m\cap I_m+j = \emptyset.
\end{equation}
Pour $k$ allant de $m$ à $0$, on définit les ensembles $A_k^m$ par récurrence :
\begin{equation}\label{eq5}
A_k^m = \bigcup_{i\in I_k} f^i(A_{k+1}^m).
\end{equation}
Toujours par récurrence on remarque que, par (\ref{eq3}), et pour tout $j\in\{1,\dots, \tau_k-1\}$, on a
\begin{equation}\label{eq4}
f^j(A_k^m)\cap A_k^m = \emptyset,
\end{equation}
si bien que les ensembles
\[A_k^m,f(A_k^m),\dots,f^{\tau_k-1}(A_k^m), X\setminus\left(\bigcup_{i=1}^{\tau_k-1} f^i(A_k^m)\right)\]
forment une partition $\eta_k^m$ de $X$. De plus, on a :
\begin{eqnarray*}
\mu\left(A_k^m\Delta A_k^{m+1}\right) & = & \mu\left(\left(\bigcup_{i\in I_k} f^i(A_{k+1}^m)\right)\Delta \left(\bigcup_{i\in I_k} f^i(A_{k+1}^{m+1})\right)\right) \\
                                      & \le & \sum_{i\in I_k} \mu\left(f^i(A_{k+1}^m)\Delta f^i(A_{k+1}^{m+1})\right) = |I_k|\mu\left(A_{k+1}^m\Delta A_{k+1}^{m+1}\right),
\end{eqnarray*}
et par (\ref{eq3}), 
\[\mu\left(A_k^m\Delta A_k^{m+1}\right) \le \frac{\tau_{k+1}}{\tau_k}\mu\left(A_{k+1}^m\Delta A_{k+1}^{m+1}\right) ;\]
une récurrence simple donne :
\begin{equation}\label{eq6}
\mu\left(A_k^m\Delta A_k^{m+1}\right) \le \frac{\tau_m}{\tau_k} \mu\left(A_m^m\Delta A_m^{m+1}\right) \le \frac{\tau_m}{\tau_k} \frac{1}{\tau_m^3 2^{\tau_m}} \le \frac{1}{\tau_m^2 2^{\tau_m}}.
\end{equation}
On déduit de cette dernière inégalité que pour $k$ fixé, la suite $(A_k^m)_{m\ge k}$ est une suite de Cauchy de l'espace des ensembles mesurables de $X$ muni de la métrique $d(A,B) = \mu(A\Delta B)$, qui converge donc vers un ensemble mesurable que l'on note $A'_k$ (voir le lemme \ref{polonais}). Par passage à la limite dans l'équation (\ref{eq4}), on obtient $f^j(A'_k)\cap A'_k = \emptyset$ pour tout $j\in\{1,\dots, \tau_k-1\}$, on peut donc définir la partition
\[\eta'_k = \left\{A'_k,f(A'_k),\dots,f^{\tau_k-1}(A'_k),B'_k\right\},\]
avec $B'_k = X\setminus\left(\bigcup_{j=0}^{\tau_k-1} f^j(A'_k)\right)$. C'est cette partition qui va vérifier les conclusions de la proposition. Le point 1. de la proposition est évidemment vérifié. Le passage à la limite dans l'équation (\ref{eq5}) assure que pour tout entier $k$, 
\[A'_k = \bigcup_{i\in I_k} f^i(A'_{k+1}),\]
ce qui montre d'une part la première partie du point 4., et d'autre part que par l'équation (\ref{eq4}) on a
\[\bigcup_{j=0}^{\tau_k-1} f^j(A'_k) = \bigcup_{j=0}^{\tau_k-1}\bigcup_{i\in I_k} f^{i+j}(A'_{k+1}) \subset \bigcup_{j=0}^{\tau_{k+1}-1}f^{j}(A'_{k+1}),\]
ce qui fait que la suite $(B_k)_k$ est décroissante pour l'inclusion ; le point 4. est prouvé. D'un autre côté, à l'aide de (\ref{eq6}),
\begin{equation}\label{derdesder'}
\mu(A^k_k\Delta A'_k)\le\sum_{l=0}^\infty \mu\big(A_k^{k+l}\Delta \mu(A_k^{k+l+1})\big)\le \sum_{l=0}^\infty\frac{1}{\tau_{k+l}^2 2^{\tau_{k+l}}}\le \frac{2}{\tau_k^2}.
\end{equation}
Ceci implique que
\[\left|\tau_k\mu(f^{\tau_k}(A_k^k) \Delta A_k^k) - \tau_k\mu(f^{\tau_k}(A'_k) \Delta A'_k)\right|\le \frac{4}{\tau_k},\]
ce qui prouve le point 3. Enfin, comme dans l'équation (\ref{eq0'}) page \pageref{eq0'}, et à l'aide de (\ref{derdesder'}), on a 
\begin{eqnarray*}
\Bigg|\mu\left(P \Delta \left(\bigcup_l f^{i_l}(A_k^k)\right)\right) & - & \mu\left(P \Delta \left(\bigcup_l f^{i_l}(A'_k)\right)\right)\Bigg| \le\\
& &  \sum_l \mu(f^{i_l}(A'_k)\Delta f^{i_l}(A_k^k)) \le \tau_k \frac{2}{\tau_k^2} = \frac{2}{\tau_k}.
\end{eqnarray*}
Le point 2. est démontré ; ceci clôt la démonstration des quatre points de la proposition.
\end{proof}

\begin{prop}\label{6.1}
Un automorphisme $f$ vérifiant les conclusions de la proposition \ref{rang2} est standard. L'ensemble $A$ de la définition \ref{Enzo} peut être pris de mesure arbitrairement grande.
\end{prop}

\begin{proof}[Preuve de la proposition \ref{6.1}] Soit $k$ un entier, posons $A = \bigcup_{i=0}^{\tau_k-1}f^i(A_k)$, et notons $f_A$ l'application induite par $f$ sur $A$. Pour tout $l\ge k$, la partition $\eta_l$ induit une partition $\xi_l$ sur $A$, dont les éléments sont, par le point 4. de la proposition \ref{rang2}, les $f^i(A_l)$ inclus dans $A$, avec $i\in\{0,\dots,\tau_l-1\}$. L'application $f_A$ permute cycliquement les éléments de la partition $\xi_l$ : par préservation de la mesure, toute la masse qui sort de $A$ en sort par $f^{\tau_k-1}(A_k)$ et toute la masse qui entre dans $A$ y entre par $A_k$, donc $f_A(f^{\tau_k-1}(A_k)) = A_k$. Ces partitions $\xi_l$ forment une suite \emph{exhaustive} de partitions de $A$, autrement dit l'ensemble des fonctions indicatrices des unions d'éléments de ces partitions est dense dans $L^2(A,\mu)$.

Ainsi, le spectre de $U_{f_A}$ est discret et charge exactement l'ensemble $\bigcap_{l\ge k} \mathbf{U}_{\tau_l}$ (où $\mathbf{U}_\tau$ désigne l'ensemble des racines $\tau$-èmes de l'unité dans $\C$) ; le théorème de Von Neumann sur le spectre discret (voir le \S 1.3. de \cite{KatokMon} et le chapitre~\og\emph{discrete spectrum}\fg~de \cite{3}) assure alors qu'il n'existe qu'un seul automorphisme, à conjugaison par une bijection bi-mesurable près, qui possède le même type spectral que $U_{f_A}$. Par un argument similaire à celui du début de la preuve, on observe que le type spectral de l'odomètre\footnote{Plus fondamentalement, on remarque que les actions de $f_A$ et de l'odomètre ainsi défini sont assez similaires, le théorème de Von Neumann permet de concrétiser cette remarque.} associé à la suite $(\tau_l)_l$ charge lui aussi exactement l'ensemble $\bigcap_{l\ge k} \mathbf{U}_{\tau_l}$. Ainsi, l'automorphisme $f_A$ est mesurablement conjugué à l'odomètre associé à la suite $(\tau_l)_{l\ge 0}$.

De plus, la propriété 3. du théorème \ref{rang2} signifie que l'on peut prendre la mesure de $A$ arbitrairement grande.
\end{proof}

Et en combinant \ref{rang1F}, \ref{rang2} et \ref{6.1}, on obtient :

\begin{theoreme}\label{pffff}
Dans $\mathrm{Homeo}(X,\mu)$, les homéomorphismes standards sont génériques.
\end{theoreme}

\section{Type spectral et généricité}\label{saltyp}

On vient de voir que le spectre d'un élément générique de $\mathrm{Homeo}(X,\mu)$ est simple, on peut donc supposer que l'on est dans ce cas. Ceci permet de définir plus simplement que dans le cas général le \emph{type spectral}, une classe d'équivalence de mesures boréliennes sur le cercle associée à tout automorphisme. Celle-ci donne de nouvelles propriétés dynamiques intéressantes. On pourra consulter par exemple \cite{Nad} ou bien \cite{Martine}.

\begin{definition}
On dit que deux mesures sont \emph{équivalentes} si tout mesurable de mesure nulle pour l'une est aussi de mesure nulle pour l'autre.
\end{definition}

\begin{definition}
Soit $\Hi$ un espace de Hilbert, $\mathcal E$ l'ensemble des projecteurs orthogonaux dans $\Hi$ et $(X, \mathcal B)$ un espace topologique muni de la tribu borélienne. Une fonction $E : \mathcal B\to \mathcal E$ est appelé une \emph{mesure spectrale} si $E(X) = Id$ et $E(\bigcup_{i=1}^\infty A_i) = \sum_{i=1}^\infty E(A_i)$ pour toute collection $\{A_i\}_{i\in\N}$ d'éléments disjoints de~$\mathcal B$.
\end{definition}

\begin{theoreme}[Théorème spectral, voir \cite{Nad}]\label{précédent}
Soit $U$ un opérateur unitaire d'un espace de Hilbert $\Hi$. Alors il existe une mesure spectrale $E_U : \mathbf{S}^1\to \mathcal{E}$ telle que $U = \int_{\mathbf{S}^1}z^{-1}dE_U(z)$, si bien que pour tout $m\in \N$, $U^m = \int_{\mathbf{S}^1}z^{-m}dE_U(z)$.
\end{theoreme}

\begin{rem}
Le théorème \ref{précédent} est une généralisation de la réduction en dimension finie : la mesure spectrale est alors une somme de mesures spectrales de Dirac en les valeurs propres, en chaque valeur propre la mesure spectrale est égale à la projection sur le sous-espace propre associé.
\end{rem}

Si $U$ est à spectre simple, le théorème \ref{précédent} permet de définir une mesure $\mathfrak{m}_U$ sur $\mathbf{S}^1$. On procède comme suit : on choisit $\varphi_0\in \Hi$ tel que $\langle \varphi_0\rangle_U$ (le sous-espace cyclique engendré par $\varphi_0$) soit dense dans $\Hi$. Pour tout mesurable $A$, on pose alors $\mathfrak{m}_{U,\varphi_0}(A) = \langle E_U(A)(\varphi_0),\varphi_0\rangle$. Ainsi on a $\langle U^m(\varphi_0),\varphi_0\rangle = \int_{\mathbf S^1} z^{-m}d\mathfrak{m}(z)$ et puisque $U$ est $\varphi_0$-cyclique et unitaire, les $\langle U^m(\varphi_0),\varphi_0\rangle$ définissent entièrement $U$ : pour tous $\varphi, \varphi' \in\Hi$, s'écrivant $\varphi = \sum_{m\ge 0} x_m U^m(\varphi_0)$ et $\varphi' = \sum_{m\ge 0} y_m U^m(\varphi_0)$, on a :
\[\langle \varphi,\varphi'\rangle = \left\langle \sum_{i\ge 0} x_i U^i(\varphi_0),\sum_{j\ge 0} y_j U^j(\varphi_0)\right\rangle = \int_{\mathbf S^1}\sum_{i,j}x_i y_j z^{i-j}d\mathfrak{m}_{U,\varphi_0}(z).\]
Cela signifie que l'application 
\begin{eqnarray*}
\Psi : & \Hi & \longrightarrow L^2(\mathbf S^1,\mathfrak{m}_{U,\varphi_0})\\
       &\sum_{i\ge 0}x_iU^i(\varphi_0) & \longmapsto \sum_{i\ge 0}x_i z^i
\end{eqnarray*}
est une isométrie, qui conjugue $U$ avec la multiplication par $z$ dans $L^2(\mathbf S^1,\mathfrak{m}_{U,\varphi})$. Au passage cela prouve aussi que la classe d'équivalence de $\mathfrak{m}_{U,\varphi_0}$ est indépendante du choix du vecteur cyclique $\varphi_0$.

\begin{definition}
Le \emph{type spectral} de $U$ est la classe d'équivalence de la mesure $\mathfrak{m}_{U,\varphi_0}$.
\end{definition}

Par la suite, il sera pratique de choisir une mesure particulière dans la classe d'équivalence des mesures $\mathfrak{m}_{U,\varphi_0}$, avec $\varphi_0$ cyclique, qui ne dépende pas du choix du vecteur $\varphi_0$ ; on vérifie facilement que si $\{\psi_i\}_{i\in\N^*}$ une base hilbertienne de $\Hi$, la mesure borélienne $\mathfrak{m}_{U,\{\psi_i\}}$, définie sur le cercle $\mathbf S^1$ par 
\[\mathfrak{m}_{U,\{\psi_i\}}(A) = \sum_{i\in\N}\frac{1}{2^i}\langle E_U(A)(\varphi_i),\varphi_i\rangle\]
est équivalente à toute mesure $\mathfrak{m}_{U,\varphi_0}$ (autrement dit elle appartient au type spectral de $U$).
\bigskip

Bien sûr, ces considérations s'appliquent sans mal au cas où $U$ est l'opérateur de Koopman $U_f$ associé à un automorphisme $f\in\mathrm{Auto}(X,\mu)$ (qui est unitaire et, par hypothèse, à spectre simple). Comme il est pratique de travailler avec une mesure plutôt qu'une classe d'équivalence de mesures, on adopte la convention suivante :

\begin{conv}
On fixe une fois pour toutes une base hilbertienne $\{\psi_i\}_{i\in\N^*}$ de $\Hi$. Par abus de langage, nous appellerons \emph{type spectral} de $f\in\mathrm{Auto}(X,\mu)$, et noterons $\mathfrak{m}_f$, la mesure $\mathfrak{m}_{U_f,\{\psi_i\}}$.
\end{conv}

Notons que l'on peut aussi définir le type spectral dans le cas où le spectre de l'opérateur n'est pas simple, mais dans ce cas la définition est plus compliquée (voir \cite{Nad}).

\begin{prop}\label{speccon}
Génériquement, la mesure $\mathfrak{m}_{f}$ n'a pas d'autre atome que 1.
\end{prop}

\begin{proof}[Preuve de la proposition \ref{speccon}] Cela découle facilement de la caractérisation (5) du mélange faible\footnote{Il est par ailleurs facile de voir que le type spectral d'un automorphisme faiblement mélangeant n'a pas d'atome autre que 1.} dans la proposition \ref{équiéqui} et du fait qu'un atome de $\mathfrak{m}_{U_f}$ est une valeur propre de $U_f$. On conclut par généricité du mélange faible (théorème~\ref{mélfaiblgéné}).
\end{proof}

\begin{prop}\label{singLeb}
Génériquement, le type spectral est singulier par rapport à la mesure de Lebesgue.
\end{prop}

\begin{proof}[Preuve de la proposition \ref{singLeb}] Par la proposition \ref{rigid}, pour un automorphisme générique $f$, il existe une suite $\{k_m\}_m$ telle que $\int_{\mathbf S^1} z^{k_m} d\mathfrak{m}_{f}(z)\to 1$. Par théorème de Radon-Nikodym, la mesure $\mathfrak{m}_{u_f}$ se décompose en $\mathfrak{m}_{f} = \mathfrak{m}_1+\mathfrak{m}_2$, avec $\mathfrak{m}_1$ absolument continue et $\mathfrak{m}_2$ singulière par rapport à la mesure de Lebesgue. Le théorème de Riemann-Lebesgue assure alors que $\int_{\mathbf S^1} z^{k_m} d\mathfrak{m}_1(z)\to 0$ ; de plus $\left|\int_{\mathbf S^1} z^{k_m} d\mathfrak{m}_2(z)\right|\le \mathfrak{m}_2(\mathbf S_1)$. Par conséquent $\mathfrak{m}_2(\mathbf S_1) = 1$ et la mesure $\mathfrak{m}_{f}$ est singulière par rapport à la mesure de Lebesgue.
\end{proof}

\chapter[Forme faible du théorème de transfert]{Densité de $\mathrm{Homeo}(X,\mu)$ dans $\mathrm{Auto}(X,\mu)$ et forme faible du théorème de transfert}

Au cours de ce chapitre, nous amorçons la preuve du théorème de transfert de S. Alpern\footnote{Qui justifie la similitude des résultats de généricité des propriétés ergodiques entre les espaces $\mathrm{Homeo}(X,\mu)$ et $\mathrm{Auto}(X,\mu)$ observée au chapitre précédent : le théorème \ref{approxmeo} est vrai aussi bien pour les homéomorphismes que pour les automorphismes, voir \cite{KatokS1}, \cite{KatokS3}}. Nous commençons par énoncer un théorème de densité faible de l'ensemble des homéomorphismes dans l'ensemble des automorphismes, condition nécessaire pour espérer avoir un théorème de transfert, pour en déduire ensuite une forme faible du théorème de transfert, ce qui donne une seconde preuve du théorème d'Oxtoby-Ulam.

Comme auparavant, on fixe une application $\phi : I^n\to X$ donnée par le corollaire \ref{Brown-mesure} et on choisit une suite $(\D_m)_{m\in\N}$ de subdivisions dyadiques de $(X,\mu)$, ce qui permet de définir une notion de permutation dyadique sur ces subdivisions.

\section{Densité des homéomorphismes parmi les automorphismes}

Le premier pas vers le théorème de transfert est la densité des homéomorphismes parmi les automorphismes. Ce résultat a été obtenu indépendamment par J. Oxtoby en 1973 \cite{Oxtolusin} et par H. White en 1974 \cite{White}. Notons que le théorème qu'ils énoncèrent est un peu plus fort et correspond à un analogue du théorème de Lusin pour les homéomorphismes préservant la mesure (voir par exemple le théorème 6.2 de \cite{2}).

\begin{theoreme}[Oxtoby, White]\label{densite-homeo}
L'espace $\mathrm{Homeo}(X,\mu)$ est dense dans l'espace $\mathrm{Auto(X,\mu)}$ pour la topologie faible. Plus précisément, pour toute boule $B$ de la topologie forte centrée en un homéomorphisme, $B\cap\mathrm{Homeo}(X,\mu)$ est dense, au sens de la topologie faible, dans la boule $B$.
\end{theoreme}

Pour démontrer ce théorème nous aurons besoin d'un résultat similaire au théorème de Lax pour les automorphismes :

\begin{lemme}\label{laxauto}
Soit $f\in\mathrm{Auto(X,\mu)}$. Alors il existe un entier $m$ et une permutation dyadique $f_m$ d'ordre $m$ arbitrairement proche de $f$ pour la topologie faible. De plus, si $d_{\mathit{forte}}(f,\mathrm{Id})<\varepsilon$, alors on peut choisir $f_m$ tel que $d_{\mathit{forte}}(f_m,\mathrm{Id})<\varepsilon$.
\end{lemme}

\begin{proof}[Preuve du lemme \ref{laxauto}] Soit $\delta>0$. On considère un entier $m$ suffisament grand pour que le diamètre des cubes de la subdivision dyadique $\D_m$ de $X$ soit plus petit que $\delta$. Pour tout cube ouvert~$C_i$ de $\D_m$, par régularité intérieure de la mesure $\mu$ (voir la page 56 de \cite{Rud}), il existe un compact $K_i$ inclus dans $f^{-1}(C_i)$ vérifiant $\mu(C_i)(1-\delta)<\mu(K_i)<\mu(C_i)$. Les compacts $K_i$ sont deux à deux disjoints, on peut donc trouver, en minorant les distances entre ces compacts, une subdivision plus fine $\D_{m'}$ telle que chacun de ses cubes ne rencontre qu'au plus un compact $K_i$. Posons $O_i$ l'union de tous les cubes qui rencontrent $K_i$ (voir la figure \ref{Picasso}).
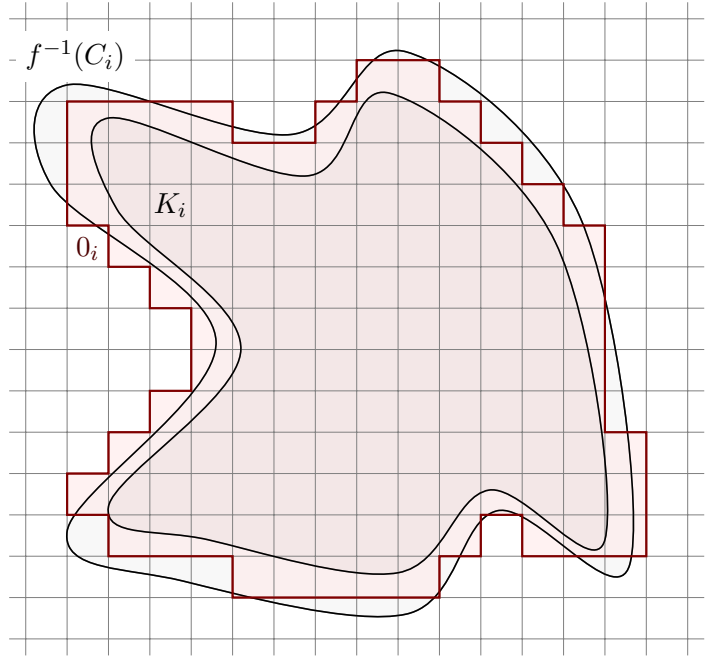
\begin{figure}
\begin{center}
\begin{tikzpicture}[scale=1.1]
\draw[fill=gray!10!white, semithick, opacity = 0.5] plot[tension=0.6, smooth cycle] coordinates{(-3.7,2.5) (-3.5,3.7) (-.8,3.1) (.6,4.1) (2.7,2.1) (3.3,-2.1) (1.7,-1.45) (0.6,-2.7) (-2.1,-2.3) (-3.5,-1.7) (-1.7,.6)};
\draw[fill=gray!30!white, semithick, opacity = 0.5] plot[tension=0.6, smooth cycle] coordinates{(-2.9,2.2) (-3,3.3) (-.6,2.6) (.4,3.6) (2.4,1.8) (3,-1.8) (1.6,-1.2) (.5,-2.2) (-1.8,-1.8) (-3,-1.4) (-1.4,.5)};
\draw[color=red!50!black, thick, fill=red!10!white, opacity = 0.5] (-3.5,3.5) -- (-1.5,3.5) -- (-1.5,3) -- (-.5,3) -- (-.5,3.5) -- (0,3.5) -- (0,4) -- (1,4) -- (1,3.5) -- (1.5,3.5) -- (1.5,3) -- (2,3) -- (2,2.5) -- (2.5,2.5) -- (2.5,2) -- (3,2) -- (3,-.5) -- (3.5,-.5) -- (3.5,-2) -- (2,-2) -- (2,-1.5) -- (1.5,-1.5) -- (1.5,-2) -- (1,-2) -- (1,-2.5) -- (-1.5,-2.5) -- (-1.5,-2) -- (-3,-2) -- (-3,-1.5) -- (-3.5,-1.5) -- (-3.5,-1) -- (-3,-1) -- (-3,-.5) -- (-2.5,-.5) -- (-2.5,0) -- (-2,0) -- (-2,1) -- (-2.5,1) -- (-2.5,1.5) -- (-3,1.5) -- (-3,2) -- (-3.5,2) -- cycle;
\draw[step=.5cm,gray,very thin] (-4.2,-3.2) grid (4.2,4.7);
\draw[semithick] plot[tension=0.6, smooth cycle] coordinates{(-3.7,2.5) (-3.5,3.7) (-.8,3.1) (.6,4.1) (2.7,2.1) (3.3,-2.1) (1.7,-1.45) (0.6,-2.7) (-2.1,-2.3) (-3.5,-1.7) (-1.7,.6)};
\draw[semithick] plot[tension=0.6, smooth cycle] coordinates{(-2.9,2.2) (-3,3.3) (-.6,2.6) (.4,3.6) (2.4,1.8) (3,-1.8) (1.6,-1.2) (.5,-2.2) (-1.8,-1.8) (-3,-1.4) (-1.4,.5)};
\draw[color=red!50!black, thick] (-3.5,3.5) -- (-1.5,3.5) -- (-1.5,3) -- (-.5,3) -- (-.5,3.5) -- (0,3.5) -- (0,4) -- (1,4) -- (1,3.5) -- (1.5,3.5) -- (1.5,3) -- (2,3) -- (2,2.5) -- (2.5,2.5) -- (2.5,2) -- (3,2) -- (3,-.5) -- (3.5,-.5) -- (3.5,-2) -- (2,-2) -- (2,-1.5) -- (1.5,-1.5) -- (1.5,-2) -- (1,-2) -- (1,-2.5) -- (-1.5,-2.5) -- (-1.5,-2) -- (-3,-2) -- (-3,-1.5) -- (-3.5,-1.5) -- (-3.5,-1) -- (-3,-1) -- (-3,-.5) -- (-2.5,-.5) -- (-2.5,0) -- (-2,0) -- (-2,1) -- (-2.5,1) -- (-2.5,1.5) -- (-3,1.5) -- (-3,2) -- (-3.5,2) -- cycle;
\node[fill=white] at (-3.4,4.05) {$f^{-1}(C_i)$};
\node at (-2.25,2.25) {$K_i$};
\node[color=red!30!black] at (-3.25,1.72) {$0_i$};
\end{tikzpicture}
\caption{Construction de l'ensemble $O_i$}
\label{Picasso}
\end{center}
\end{figure}
Quitte à prendre une subdivision encore plus fine on peut supposer que la mesure de $O_i$ est plus petite que celle de $C_i$ (en effet, si on pose $O_{i,k}$ l'ensemble $O_i$ associé à une subdivision d'ordre $k$, les $(O_{i,k})_k$ sont emboîtés et la mesure de $O_{i,k}$ tend vers celle de $K_i$). On peut donc trouver une permutation dyadique $f_m$ qui envoie tout cube de $O_i$ sur un cube de $C_i$ (pour la subdivision $\D_{m'}$). Alors pour presque tout point $x$ dans un $K_i$, $f_m(x)$ et $f(x)$ appartiennent au même cube $C_i$, ils sont donc à une distance inférieure à $\delta$. De plus l'ensemble des points qui vérifient cette propriété est de mesure plus grande que $1-\delta$, on a donc $d_{\mathit{faible}}(f_m,f)<\delta$.
\bigskip

Reste à établir l'inégalité concernant la distance avec l'identité. On suppose que l'on a l'inégalité du lemme : $d_{\mathit{forte}}(f,\mathrm{Id})<\varepsilon$. On se donne un $\delta$ \og de sécurité \fg~tel que $d_{\mathit{forte}}(f,\mathrm{Id})<\varepsilon-3\delta$. On remarque immédiatement que, ayant choisi une permutation dyadique $f_m$ qui est $\delta$-proche de $f$, pour presque tout $x$ dans l'union des $K_i$ (et même l'union des $O_i$), on a $\mathrm{dist}(x,f_m(x))<\varepsilon-2\delta$. L'idée est alors de modifier $f_m$ en un $f_m'$ de telle sorte que $f_m'$ soit égal à l'identité en dehors des $O_i$ et assez proche de $f_m$ sur les $O_i$.

On pose $f_m' = \tau\circ f_m$ avec $\tau$ défini comme suit : posant $\F'$ l'union des $O_i$ et $\F$ la $f_m$-image de $\F'$, on prend $\tau_{|\F^\complement} = {f_m^{-1}}_{|\F^\complement}$ (dont l'image est $\F'^\complement$). Reste à définir $\tau$ sur $\F$. Pour cela, on numérote les cubes de $\F$ et de $\F'$, et définit tout simplement un ordre sur les cubes de la subdivision la plus fine (celle à laquelle appartiennent les cubes des $O_i$), contenant $MN$ cubes, compatible avec celui de la subdivision plus grossière (relative aux $C_i$), contenant $M$ cubes. On commence par définir l'ordre sur la subdivision la plus grossière en numérotant ses cubes de manière à ce que deux cubes successifs soient adjacents (en prenant garde au fait qu'il faut aussi que le dernier cube et le premier soient adjacents). On numérote alors de la même manière les petits cubes à l'intérieur de chaque grand cube ; tout petit cube dans le $i$-ème grand cube aura un numéro compris entre $(i-1)N+1$ et $iN$. Ainsi, si deux petits cubes ont des numéros espacés d'au plus $N$, ils seront distants d'au plus $2\delta$. On définit alors tout simplement $\tau$ sur $\F'$ comme étant l'unique bijection croissante, pour l'ordre qu'on a défini, allant de $\F$ à $\F'$. La distance (faible) de $\tau$ à l'identité est alors majorée à l'aide du nombre de cubes de $\F'^\complement$.

Or on contrôle le nombre de cubes de $\F$. En effet, on peut choisir les $K_i$ de manière à ce qu'ils recouvrent une proportion d'au moins $1-\frac{1}{M}$ du cube. Par conséquent, la proportion de cubes dans $\F'^\complement$ est d'au plus $\frac{1}{M}$ ; puisqu'il y a $MN$ cubes dans la petite subdivision, le nombre de cubes dans $\F^\complement$ est au plus $N$.

Par la remarque faite au dessus, on en déduit que la distance faible de $\tau$ à l'identité est majorée par $2\delta$ ; par conséquent on a d'une part $d_{\mathit{faible}}(f_m',f)<3\delta$ et d'autre part ${f_m'}_{|\F'^\complement} = Id$ et $d_{\mathit{forte}}(f_m',\mathrm{Id})<\varepsilon$ sur $\F'$, si bien que $d_{\mathit{forte}}(f_m',\mathrm{Id})<\varepsilon$.
\end{proof}

Pour montrer le théorème il suffit donc de pouvoir approcher les permutations cycliques par des homéomorphismes préservant la mesure, avec une inégalité concernant la distance à l'identité. Ce problème est résolu par le lemme suivant :

\begin{lemme}\label{homeo-permutation}
Soit $f_m$ une permutation dyadique d'ordre $m$ de $X$. Alors il existe $f\in$ $\mathrm{Homeo}(X,\mu)$ arbitrairement proche de $f_m$ pour la topologie faible. De plus, si $d_{\mathit{forte}}(f_m,\mathrm{Id})<\varepsilon$, alors on peut choisir $f$ tel que $d_{\mathit{forte}}(f,\mathrm{Id})<\varepsilon$.
\end{lemme}

\begin{proof}[Preuve du lemme \ref{homeo-permutation}] Soit $\delta>0$. Posons $x_1,\dots,x_k$ les centres des cubes dyadiques $C_i$ permutés par $f_m$ et $y_1,\dots,y_k$ leurs $f_m$-images. On se donne un $\varepsilon'$ \og de sécurité \fg~tel que $d_{\mathit{forte}}(f_m,\mathrm{Id})<\varepsilon'<\varepsilon$. La proposition sur l'extension des applications finies (proposition \ref{extension}) nous permet d'obtenir $g\in\mathrm{Homeo}(X,\mu)$, préservant l'orientation, tel que pour tout $i$, $g(x_i)=y_i$ ; on peut de plus supposer que $d_{\mathit{forte}}(g,\mathrm{Id})<\varepsilon'$.

On utilise alors le lemme \ref{4.2}, appliqué à $\delta$ et $\varepsilon -\varepsilon'$, qui nous donne un homéomorphisme $f$ qui coïncide avec $f_m$ sur un ensemble de mesure au moins $1-\delta$, d'où $d_{\mathit{forte}}(f,f_m)<\delta$, et qui vérifie $d_{\mathit{faible}}(f,\mathrm{Id})<\varepsilon-\varepsilon'+\varepsilon'$ ; le lemme est prouvé.
\end{proof}

\section{Forme faible du théorème de transfert}

Du théorème précédent on peut déduire facilement une forme faible du théorème de transfert (voir le corollaire 10 de \cite{X}) :

\begin{theoreme}[de transfert faible]\label{transfertfaibl}
Soit $P$ un $G_\delta$ de l'ensemble $\mathrm{Auto}(X,\mu)$ pour la topologie faible. Si l'adhérence de $P$ pour la topologie uniforme contient $\mathrm{Homeo}(X,\mu)$, alors $P\cap \mathrm{Homeo}(X,\mu)$ est un $G_\delta$ dense de $\mathrm{Homeo}(X,\mu)$ pour la topologie uniforme.
\end{theoreme}

\begin{proof}[Preuve du théorème \ref{transfertfaibl}] Par hypothèse, l'ensemble $P$ s'écrit $P = \bigcap P_k$ où $P_k$ est un ouvert de $\mathrm{Auto}(X,\mu)$ pour la topologie faible.

Montrons que $P_k\cap \mathrm{Homeo}(X,\mu)$ est ouvert dans $\mathrm{Homeo}(X,\mu)$ pour la topologie uniforme. Soit $(f_m)_{m\ge 0}$ une suite d'éléments de $\mathrm{Homeo}(X,\mu)\setminus P_k$ convergeant pour la topologie uniforme vers une limite notée $f$. D'une part cette limite est encore dans $\mathrm{Homeo}(X,\mu)$ ; d'autre part c'est aussi la limite de la suite $(f_m)$, pour la topologie faible cette fois. Puisque $P_k$ est ouvert pour la topologie faible, $f$ n'appartient pas à $P_k$. On en déduit que $P_k\cap \mathrm{Homeo}(X,\mu)$ est ouvert dans $\mathrm{Homeo}(X,\mu)$ pour la topologie uniforme.

Reste à montrer que $P_k\cap \mathrm{Homeo}(X,\mu)$ est dense dans $\mathrm{Homeo}(X,\mu)$ pour la topologie uniforme. Soit $f\in\mathrm{Homeo}(X,\mu)$ et $\varepsilon>0$. Puisque $P_k$ est dense pour la topologie uniforme, $B_{\varepsilon/2}(f)\cap P_k$ contient un homéomorphisme $g$ (où $B_\varepsilon(f)$ désigne la boule de centre $f$ et de rayon $\varepsilon$ pour la topologie forte). Mais $P_k$ est ouvert dans $\mathrm{Auto}(X,\mu)$ pour la topologie uniforme, donc il existe $\varepsilon'<\frac{\varepsilon}{2}$ tel que $B_{\varepsilon'}(g)\subset P_k$. On utilise alors le théorème de densité des homéomorphismes parmi les automorphismes (théorème \ref{densite-homeo}) : $B_{\varepsilon'}(g)\,\cap\mathrm{Homeo}(X,\mu)$ est dense dans $B_{\varepsilon'}(g)$ pour la topologie faible, en particulier il existe $g'$ dans cette intersection. Alors $g'\in B_\varepsilon(f)\cap P_k \cap\mathrm{Homeo}(X,\mu)$. En conclusion $P_k\cap \mathrm{Homeo}(X,\mu)$ est dense dans $\mathrm{Homeo}(X,\mu)$ pour la topologie uniforme.
\end{proof}

\section{Nouvelle preuve du théorème d'Oxtoby-Ulam}

Le théorème de transfert faible permet d'obtenir une preuve relativement simple du théorème d'Oxtoby-Ulam, directement adaptée de l'article de P. Halmos (théorèmes~5 et 6 de \cite{Halm44}) :

\begin{theoreme}[Oxtoby-Ulam]\label{ergo}
Dans $\mathrm{Homeo}(X,\mu)$, l'ergodicité est générique.
\end{theoreme}

\begin{proof}[Preuve du théorème \ref{ergo}] Montrons tout d'abord que l'ensemble des automorphismes ergodiques forme un $G_\delta$ de $\mathrm{Auto}(X,\mu)$. On sait, par le théorème ergodique de Von Neumann, qu'un automorphisme $f$ est ergodique si et seulement si pour tous mesurables $A$ et $B$,
\[\frac 1m\sum_{k=0}^{m-1} \mu(A\cap f^k(B))\underset{m\to+\infty}{\longrightarrow}\mu(A)\mu(B).\]
Prenons $\{A_i\}_i$ une famille dénombrable dense\footnote{Rappelons que l'on a muni l'ensemble des boréliens de $X$ de la topologie induite par la distance $d(A,B) = \mu(A\Delta B)$.} dans l'ensemble des mesurables de $X$ (par exemple formée des unions finies de cubes dyadiques) et posons 
\[E_{i,j,k,m} = \left\{f\in\mathrm{Auto}(X,\mu)\mid \left|\frac 1m\sum_{l=0}^{m-1}\mu(A_i\cap f^l(A_j))-\mu(A_i)\mu(A_j)\right| <\frac 1k\right\}\]
ainsi que 
\[F = \bigcap_{i,j,k}\bigcup_m E_{i,j,k,m}.\]
Les ensembles $E_{i,j,m,n}$ sont clairement ouverts pour la topologie faible, donc $F$ est un $G_\delta$. Par ce que l'on vient de dire, l'ensemble des automorphismes ergodiques est inclus dans $F$, reste à prouver l'autre inclusion : soit $f\in \mathrm{Auto}(X,\mu)$ non ergodique, montrons que $f\notin F$. Par hypothèse il existe un ensemble mesurable $B$ de mesure non triviale et stable par $f$. Posons $\delta = \mu(B)\mu(B^\complement)/8$ et choisissons $i$, $j$ et $m$ tels que $\mu(A_j\Delta B)<\delta$, $\mu(A_i\Delta B^\complement)<\delta$ et $1/m<4\delta$. Alors pour tout entier $l$
\begin{eqnarray*}
\left|\mu(A_i\cap f^l(A_j)) - \mu(B^\complement\cap f^l(B))\right| & \le & \mu((A_i\cap f^l(A_j)) \Delta (B^\complement\cap f^l(B)))\\
                                                                   & \le & \mu(A_i\Delta B^\complement)+\mu(A_j\Delta B)<2\delta
\end{eqnarray*}
et par conséquent pour tout entier $m$
\[\left|\frac 1m\sum_{l=0}^{m-1}\mu(A_i\cap f^l(A_j))-\frac 1m\sum_{l=0}^{m-1}\mu(B^\complement\cap f^l(B))\right| <2\delta.\]
De plus, on a :
\begin{eqnarray*}
\left|\mu(A_i)\mu(A_j)-\mu(B^\complement)\mu(B)\right| & \le & \mu(A_i)\left|\mu(A_j)-\mu(B)\right| + \mu(B)\left|\mu(A_i)-\mu(B^\complement)\right|\\
  & < & 2\delta.
\end{eqnarray*}
Remarquons enfin que l'on a, par invariance de $B$ par $f$ :
\[\left|\frac 1m\sum_{l=0}^{m-1}\mu(B^\complement\cap f^l(B)) - \mu(B)\mu(B^\complement)\right| = \mu(B)\mu(B^\complement) = 8\delta.\]
En appliquant l'inégalité triangulaire inversée, puis en combinant les équations précédentes on obtient alors : 
\begin{eqnarray*}
\lefteqn{\left|\frac 1m\sum_{l=0}^{m-1}\mu(A_i\cap f^l(A_j)) - \mu(A_i)\mu(A_j) \right|}\\
 & \ge & -\left|\frac 1m\sum_{l=0}^{m-1}\mu(A_i\cap f^l(A_j))-\frac 1m\sum_{l=0}^{m-1}\mu(B^\complement\cap f^l(B))\right|\\
  & & + \left|\frac 1m\sum_{l=0}^{m-1}\mu(B^\complement\cap f^l(B)) - \mu(B)\mu(B^\complement)\right|\\
  & & - \left|\mu(A_i)\mu(A_j)-\mu(B^\complement)\mu(B)\right|\\
  &>& 4\delta>\frac 1 m,
\end{eqnarray*}
si bien que $f\notin F$.

Par le théorème \ref{transfertfaibl} (forme faible du théorème de transfert), il suffit maintenant de montrer que pour tout $h\in \mathrm{Homeo}(X,\mu)$ et tout $\varepsilon >0$, il existe $f\in \mathrm{Auto}(X,\mu)$ ergodique telle que $d_{\mathit{forte}}(f,h)<\varepsilon$.

On applique encore une fois le théorème de Lax (théorème \ref{Lax}) à $h$ ; on obtient une subdivision dyadique $\D_m$ dont les cubes sont de diamètre $diam$ et une permutation dyadique cyclique $h_m$ telle que $d_{\mathit{forte}}(h,h_m)+ diam < \varepsilon$. Par commodité, numérotons les cubes $(C_i)_{1\le i \le N}$ de $\D_m$ de telle manière que $h_m(C_i) = C_{i+1}$. Prenons $e : C_1\to C_1$ un automorphisme ergodique\footnote{Par exemple un automorphisme linéaire d'Anosov.}, que l'on étend par l'identité en dehors de $C_1$ en un automorphisme\footnote{Attention, $e$ n'est pas un homéomorphisme.} $e : I^n\to I^n$. On pose alors $f = e\circ h_m$ l'application recherchée, qui vérifie
\[d_{\mathit{forte}}(h,f)\le d_{\mathit{forte}}(h,h_m)+d_{\mathit{forte}}(h_m,f)< \varepsilon-diam+diam=\varepsilon.\]
L'automorphisme $f$ est ergodique : la structure de permutation cyclique de $h_m$ permet de transmettre l'ergodicité de $e$ sur $C_1$ aux autres cubes de $X$. En effet, supposons qu'il existe un sous-ensemble mesurable $S$ de $X$, invariant par $f$ et de mesure vérifiant $0<\mu(S)<1$. Pour tout $i$ on pose $S_i = S\cap C_i$. Par la structure de $f$ --- qui comme $h_m$ envoie chaque $C_i$ sur $C_{i+1}$ --- on a $f^N(S_i) = S_i$ ; et l'ensemble $S$ est équiréparti sur les $C_i$, ce qui implique que l'on a $0<\mu(S_i)<1/N$. En particulier l'ensemble $S_1$ vérifie $f^N(S_1) = S_1$ et $0<\mu(S_1)<\mu(C_1)$, mais puisque, par construction de $f$, on a $f^N(S_1) = e(S_1)$, on en déduit que $e(S_1)  =S_1$, ce qui contredit l'ergodicité de $e$. Ainsi, $f$ est ergodique.
\end{proof}

\chapter[Transfert des propriétés ergodiques génériques]{Transfert des propriétés ergodiques génériques de $\mathrm{Auto}(X,\mu)$ vers $\mathrm{Homeo}(X,\mu)$}\label{partie 3}

Le but de ce chapitre est d'établir cette fois-ci la forme forte du théorème de transfert : toute propriété ergodique générique parmi les automorphismes l'est aussi parmi les homéomorphismes. Pour cela nous utiliserons le résultat de densité des homéomorphismes parmi les automorphismes établi au chapitre précédent, mais aussi la densité de le classe de conjugaison de tout automorphisme apériodique parmi les homéomorphismes (en un sens qui sera défini). C'est ce résultat, qui était d'ailleurs connu depuis les années 40 sous une forme plus faible\footnote{Toute classe de conjugaison d'un automorphisme apériodique est dense parmi les \emph{automorphismes}, établi par P. Halmos en 1944 \cite{HalmMix}.}, qui constituait la dernière étape nécessaire à l'obtention du théorème de transfert. Elle a été franchie par S.~Alpern en 1978 (\cite{AlpernGene}, \cite{AlpernTopo}), ce qui a amené ce dernier à énoncer le théorème de transfert.

Notons que les autres arguments de la preuve du théorème de transfert étaient déjà connus de P. Halmos en 1956 ; on trouve dans son livre fondateur \cite{3} des formes faibles du théorème \ref{conjug-aper} alors appelé \emph{conjugacy lemma} (page 77), de lemmes de Rokhlin (lemmes \ref{rokhlin} et \ref{rokhlin-periodique}) et du lemme \ref{laxauto}, appelé \emph{weak approximation theorem}, ainsi que des preuves des théorèmes \ref{mélange pastop} et \ref{mélange faiblemes} pour l'ensemble $\mathrm{Auto}(X,\mu)$, appelés \emph{first} et \emph{second category theorems}.  L'outil fondamental utilisé dans les preuves de ce chapitre est le lemme de Rokhlin (lemme \ref{rokhlin}).

Comme auparavant, on fixe une application $\phi : I^n\to X$ donnée par le corollaire~\ref{Brown-mesure} et on choisit une suite $(\D_m)_{m\in\N}$ de subdivisions dyadiques de $(X,\mu)$, ce qui permet de définir une notion de permutation dyadique sur ces subdivisions.

\section[Densité des classes de conjugaison des apériodiques]{Densité des classes de conjugaison des automorphismes apériodiques}

On commence par établir la densité des classes de conjugaison des automorphismes conservatifs apériodiques.

\begin{definition}
Un automorphisme $f\in\mathrm{Auto}(X,\mu)$ est dit \emph{apériodique} si l'ensemble de ses points périodiques est de mesure nulle.
\end{definition}

\begin{theoreme}[Alpern, \cite{2}]\label{conjug-aper}
Soit $f\in\mathrm{Auto}(X,\mu)$. Si $f$ est apériodique, alors l'adhérence pour la topologie forte de la classe de conjugaison de $f$ dans $\mathrm{Auto}(X,\mu)$ contient $\mathrm{Homeo}(X,\mu)$.
\end{theoreme}

On obtient facilement le lemme d'Halmos comme corollaire de ce théorème et de la densité de $\mathrm{Homeo}(X,\mu)$ dans $\mathrm{Auto}(X,\mu)$ (théorème \ref{densite-homeo}) :

\begin{coro}[lemme d'Halmos]\label{LemHalm}
Dans $\mathrm{Auto}(X,\mu)$, la classe de conjugaison de tout apériodique est dense (pour la topologie faible).
\end{coro}

\begin{rem}\label{PasHalmHomeo}
Un résultat récent de M. Entov, L. Polterovich et P. Py \cite{8} indique que pour toute surface orientable différente de la sphère, il existe un quasi-morphisme homogène\footnote{Un \emph{quasi-morphisme} d'un groupe $G$ est une application $\varphi : G\to \R$ telle qu'il existe une constante $C>0$ vérifiant $|\varphi(xy)-\varphi(x)-\varphi(y)|<C$ pour tout $x, y\in G$. Ce quasi-morphisme est dit \emph{homogène} si de plus $\varphi(x^n) = n\varphi(x)$ pour tout $x\in G$ et $n\in \Z$.} non trivial et continu (pour la topologie forte) sur $\mathrm{Homeo}(X,\mu)$. Celui-ci est automatiquement invariant par conjugaison\footnote{Ceci est vrai pour tous les quasi-morphismes homogènes.}. Pour un tel quasi-morphisme $\varphi$, il existe $f\in\mathrm{Homeo}(X,\mu)$ tel que $\varphi(f)\neq 0$. Supposons qu'il existe $\tilde{f}$ dont la classe de conjugaison est dense. Puisque $\varphi$ est constant sur la classe de conjugaison de $\tilde{f}$, et comme $\varphi(\mathrm{Id})=0$, par continuité, l'identité et $h$ ne peuvent pas appartenir tous les deux à l'adhérence de la classe de conjugaison de $\tilde{f}$. Ainsi, il n'y a pas d'équivalent du lemme d'Halmos pour $\mathrm{Homeo}(X,\mu)$.
\end{rem}

La preuve du théorème \ref{conjug-aper} requiert le lemme suivant : 

\begin{lemme}\label{partition}
Soit $f\in\mathrm{Auto}(X,\mu)$ un automorphisme apériodique et $p$, $q$ deux entiers naturels premiers entre eux. Alors il existe deux ensembles $t_1$ et $t_2$ de même mesure tels que les ensembles 
\[t_1,f(t_1),\dots,f^{p-1}(t_1),\quad t_2,f(t_2),\dots,f^{q-1}(t_2)\]
forment une partition de $X$.
\end{lemme}

Autrement dit, le lemme affirme l'existence d'une tour pleine à deux colonnes de hauteurs respectives $p$ et $q$ :

\begin{definition}
Une \emph{tour} est la donnée d'un ensemble mesurable $t$ de mesure positive, appelé la \emph{base} de la tour, et d'une partition mesurable de cet ensemble $t$ telle que la fonction temps de retour
\[\tau_t(x) = \min \{\tau>0\mid f^\tau(x)\in t\}\]
soit constante sur tout élément de la partition. \'Etant donné $a\in t$ et posant $\tau=\tau_t(a)$, la \emph{colonne} au dessus de $a$ est la suite $(a,f(a),\dots,f^{\tau-1}(a))$ ; $a$ est appelé la \emph{base} de la colonne et $\tau$ sa \emph{hauteur}. La tour est dite \emph{pleine} si ses itérés par $f$ recouvrent $X$ tout entier.
\end{definition}

\begin{rem}
La condition \og$p$ et $q$~premiers entre eux~\fg~est nécessaire : considérons une rotation irrationnelle $f$ du cercle et supposons qu'il existe une partition telle que donnée dans le lemme avec $p=kp'$ et $q=kq'$. Posons 
\[t = \big(t_1\cup f^k(t_1)\cup\cdots\cup f^{k(p-1)}(t_1)\big) \cup \big(t_2\cup f^k(t_2)\cup\cdots\cup f^{k(q-1)}(t_2)\big)\]
Alors $t$ satisfait $f^k(t) = t$ ce qui contredit l'ergodicité de $f^k$.
\end{rem}

Par la suite on notera $m = pq$. Pour montrer le lemme précédent nous aurons besoin du lemme de Rokhlin :

\begin{lemme}[Rokhlin]\label{rokhlin}$ $
\begin{enumerate}
\item
Soient $f\in \mathrm{Auto}(X,\mu)$ un automorphisme apériodique et $m\in\N^*$. Alors il existe un ensemble mesurable $F\subset X$ tel que $\bigcup_{i\in\Z}f^i(F) \stackrel{\mathrm{p.p.}}{=} X$ et tel que les $m$ premiers itérés de $F$ : $F,f(F),\dots,f^{m-1}(F)$ soient deux à deux disjoints.
\item
Soient $f\in\mathrm{Auto}(X,\mu)$ un automorphisme dont les points de périodes plus petites que $m$ forment un ensemble de mesure nulle et $\varepsilon>0$. Alors il existe un ensemble mesurable $F\subset X$ tel que les $m$ premiers itérés de $F$ par $f$ soient deux à deux disjoints et dont l'union forme un ensemble de mesure plus grande que $1-\varepsilon$.
\end{enumerate}
\end{lemme}

La preuve du premier point du lemme de Rokhlin dans le cas où $f$ est ergodique est beaucoup plus rapide et simple. À titre pédagogique, nous commençons donc par ce cas.

\begin{proof}[Preuve du premier point du lemme dans le cas ergodique] Choisissons un ensemble $F'$ de mesure strictement positive, mais aussi strictement plus petite que $1/m$ ; alors l'ensemble 
\[F = F'\setminus \big(f^{-1}(F')\cup\cdots\cup f^{-m}(F')\big)\]
est, par construction, disjoint de ses $m-1$ premiers itérés et de mesure non nulle. En effet, si $F$ était de mesure nulle, on aurait
\[F'\stackrel{\mathrm{p.p.}}{\subset}f^{-1}(F')\cup\cdots\cup f^{-m}(F'),\]
d'où
\[F'\cup f^{-1}(F')\cup\cdots\cup f^{-(m-1)}(F')\stackrel{\mathrm{p.p.}}{\subset}f^{-1}(F')\cup\cdots\cup f^{-m}(F').\]
Puisque $\mu$ est invariante par $f$, elle donne la même masse au deux membres de cette inclusion ; on peut donc écrire une égalité presque partout et non une simple inclusion. Mais alors, par ergodicité, on en déduit que 
\[\mu\big(F'\cup f^{-1}(F')\cup\cdots\cup f^{-(m-1)}(F')\big) \in\{0,1\},\]
ce qui est impossible vu que $\mu(F')\in]0,\frac 1m[$ par hypothèse. Finalement, par ergodicité, on a $\bigcup_{m\in\Z}f^m(F) \stackrel{\mathrm{p.p.}}{=} X$. Ceci prouve la première assertion du lemme dans le cas où $f$ est ergodique.
\end{proof}

\begin{proof}[Preuve du lemme de Rokhlin] Passons au cas général. On prouve en premier lieu la seconde assertion du lemme. Soient $\varepsilon>0$, et $G_N$ l'ensemble des points de $X$ dont les $m$ premiers itérés par $f$ sont dans des cubes différents de la subdivision dyadique $\D_N$. On ne tiendra pas compte des points dont les itérés sont sur une frontière d'un cube, car ils forment un ensemble de mesure nulle. Par hypothèse sur les points périodiques de $f$, on a $\mu (\bigcup_{N\in \N}G_N$)$\,=1$ ; on peut donc trouver $N_0$ tel que $G_{N_0}$ soit de mesure au moins égale à $1-\varepsilon/m$. On pose pour plus de simplicité $G= G_{N_0}$. Alors pour tout cube $C$ de la subdivision dyadique $\D_{N_0}$, l'ensemble $G\cap C$ est disjoint de ses $m$ premiers itérés.

Notant $C_1,\dots ,C_d$ les cubes de $\D_{N_0}$, on définit par récurrence une suite d'ensembles $F_1,\dots,F_d$ comme suit :
\[F_0:=\emptyset\]
\[F_i:=F_{i-1}\cup\big( (C_i\cap G)\setminus(f^{-(m-1)}(F_{i-1})\cup\cdots\cup f^{m-1}(F_{i-1}))\big)\]
pour $i$ allant de 0 \`a $d$.
Montrons par récurrence que $F_i$ est disjoint de ses $m$ premiers itérés. L'initialisation est triviale. Calculons $F_i\cap f^j(F_i)$ pour $1\leq j\leq m-1$ : si $x\in F_i\cap f^j(F_i)$, alors $x$ est dans l'un des ensembles suivants :
\[\begin{array}{c}
F_{i-1}\,\cap\, f^j(F_{i-1}),\\
F_{i-1}\,\cap\, f^j\big((C_i\cap G)\setminus(f^{-(m-1)}(F_{i-1})\cup\cdots\cup f^{m-1}(F_{i-1}))\big),\\
\big((C_i\cap G)\setminus(f^{-(m-1)}(F_{i-1})\cup\cdots\cup f^{m-1}(F_{i-1}))\big)\,\cap\, f^j(F_{i-1}),\\
(C_i\cap G)\,\cap\, f^j(C_i\cap G).
\end{array}\]
Le premier ensemble est vide par hypothèse de récurrence. Le deuxième et le troisième le sont car alors $x\in F_{i-1}\cap F_{i-1}^\complement$ ou $x\in f^j(F_{i-1})\cap f^j(F_{i-1}^\complement)$. Enfin, le dernier l'est aussi car on a vu plus haut que pour tout cube $C$ de $\D_{N_0}$, $G\cap C$ est disjoint de ses $m$ premiers itérés. Posons $F = F_d$, alors $F$ est disjoint de ses $m$ premiers itérés.

Montrons maintenant que l'ensemble $\bigcup_{j=0}^{m-1} f^j(F)$ recouvre l'ensemble $\bigcap_{j=0}^{m-1} f^j(G)$, qui est de mesure plus grande que $1-\varepsilon$. Pour cela remarquons tout d'abord que la suite $(F_i)$ est croissante ; ceci implique que $F$ est non vide : si $i$ est le plus petit entier tel que $C_i\cap G$ soit non vide, alors $F_i$ est non vide. Soit donc $x\in G\cap\cdots\cap f^{m-1}(G)$. Alors presque sûrement, il existe un seul $i$ tel que $x\in C_i$. Si $x\notin F_i$, on sait qu'il existe un $j$ dans $[-(m-1),m-1]\cap \N$ tel que $x\in f^j(F_{i-1})$, car sinon $x$ serait dans
\[(C_i\cap G)\setminus(f^{-(m-1)}(F_{i-1})\cup\cdots\cup f^{m-1}(F_{i-1})),\]
donc dans $F_i$. Par conséquent,
\[x\in \bigcup_{j=-(m-1)}^{m-1}f^j(F_{i-1}) \cup F_i \ \subset \bigcup_{j=-(m-1)}^{m-1}f^j(F).\]
On peut alors appliquer un raisonnement similaire aux points $f^{-k}(x)$ pour $k$ allant de 0 à $m-1$ ; on obtient : 
\[x\in\bigcap_{k=0}^{m-1} \bigcup_{j=-(m-1)}^{m-1}f^{j+k}(F) = \bigcup_{j=0}^{m-1}f^j(F).\]
Ainsi $\bigcup_{j=0}^{m-1} f^j(F)$ recouvre l'ensemble $\bigcap_{j=0}^{m-1} f^j(G)$, qui est de mesure au moins $1-\varepsilon$. Ceci prouve la seconde assertion du lemme. La première est alors obtenue en itérant (éventuellement à l'aide d'une induction transfinie) la propriété que l'on vient d'obtenir, en l'appliquant à chaque fois au complémentaire de l'ensemble invariant $\bigcup_{i\in\Z}f^i(F)$.
\end{proof}

\begin{proof}[Preuve du lemme \ref{partition}] On se donne un ensemble $F$ donné par le premier point du lemme de Rokhlin, avec $m\ge pq$. On remarque tout d'abord que par invariance, les ensembles $\bigcup_{j\in\Z}f^j(F)$ et $\bigcup_{j\in\N}f^j(F)$ sont de même mesure et par conséquent coïncident à un ensemble de mesure nulle près. Considérons la tour de Kakutani au dessus de $F$, c'est-à-dire une partition de $\bigcup_{j\in\N}f^j(F)$ selon le temps de retour dans $F$ par $f$. Si on note $F_k$ l'ensemble des points de $F$ dont le temps de retour par $f$ dans $F$ est égal à $k$, alors $\bigcup_{k\in\N} F_k = \bigcup_{j\in\N}f^j(F)$. Toujours par le lemme de Rokhlin, toute colonne est de hauteur supérieure à $m$. D'autre part, tout entier $k\ge m$ se décompose selon l'identité de Bézout : $k = \alpha p+\beta q$, avec $\alpha$ et $\beta$ positifs \footnote{Voir la note \ref{note-page} page \pageref{note-page}.}. Partant, on partitionne chaque colonne en deux colonnes de hauteurs $p$ et $q$ : pour tout $k$, on divise la colonne au desus de $F_k$ en deux plus petites,
\[F_k,f(F_k),\dots, f^{\alpha p-1}(F_k)\qquad\mathrm{et}\qquad f^{\alpha p}(F_k),f^{\alpha p+1}(F_k),\dots,f^k(F_k),\]
sur la base $F_k\cup f^{\alpha p}(F_k)$. La première colonne peut être réduite à une colonne de hauteur $p$ de base
\[F_k\cup f^p(F_k)\cup\cdots\cup f^{(\alpha-1)p}(F_k);\]
le même procédé permet de réduire la seconde à une colonne de hauteur $q$. Posant $t_1$ (respectivement $t_2$) l'union de toutes les bases de colonnes de hauteur $p$ (respectivement $q$), on obtient les ensembles demandés par le lemme, à ceci près qu'ils n'ont pas forcément la même mesure.

Modifions donc la preuve précédente pour qu'ils aient la même mesure. Posons $\alpha$ le plus petit des deux quotients $\frac{p}{p+q}$ et $\frac{q}{p+q}$ et soit $N$ un entier assez grand (qu'on minorera par la suite). On reprend notre partition de $F$ par le temps de retour. On prend cette fois-ci une tour donnée par le lemme de Rokhlin, de hauteur $h$ plus grande que $Nm$ et on écrit $h=Am+B=(A-1)m+(B+m)$ avec $0\le B<m$. On considère la colonne de hauteur $B+m$ située en haut de la tour et on la subdivise en deux colonnes de hauteurs $p$ et $q$ comme \emph{supra}. Puisque $B+m<2m$, le quotient du volume de cette colonne sur celui de la colonne de départ est $\frac{B+m}{h}<\frac{2}{N}<\alpha$ pour $N>\frac{2}{\alpha}$. Reste à s'occuper de la colonne restante, qui est de hauteur $(A-1)m$ ; on la réduit à une colonne de hauteur $m$ (par le même procédé qu'au dessus). On se retrouve alors avec une tour composée de trois colonnes de hauteurs respectives $p$, $q$ et $m$, dont les deux premières ont des bases de mesures $\mu_p$ et $\mu_q$ plus petites que $\frac{1}{p+q}$. Nous allons nous servir de la colonne restante pour boucher les trous : on la divise tout d'abord en deux colonnes dont les bases sont de mesures respectives $\frac{1}{p+q}-\mu_p$ et $\frac{1}{p+q}-\mu_q$. La première colonne se réduit à une colonne de hauteur $p$ et la seconde à une colonne de hauteur $q$ (comme on l'a déjà fait), pour finir on réunit ensemble les colonnes de mêmes hauteurs, ce qui clôt la démonstration du lemme.
\end{proof}

\begin{proof}[Preuve du théorème \ref{conjug-aper}] Soient $h\in\mathrm{Homeo}(X,\mu)$, $f\in\mathrm{Auto(X,\mu)}$ apériodique et $\varepsilon>0$. Par le corollaire \ref{Laxbis}, on peut trouver un entier $m$ et une permutation dyadique d'ordre $m$, notée $h_m$, tels que :
\begin{itemize}
\item le diamètre des cubes de la subdivision dyadique $\D_m$ soit  plus petit que $\varepsilon$,
\item la permutation dyadique $h_m$ possède exactement deux cycles de longueurs $p$ et $q$ premières entre elles,
\item  $d_{\mathit{forte}}(h_m,h)<\varepsilon$.
\end{itemize}
Soient $C_1$ et $C_2$ deux cubes adjacents, chacun dans l'un de ces cycles. 

On applique maintenant le lemme \ref{partition}, qui nous permet d'obtenir une tour associée à $f$ composée d'une colonne de hauteur $p$ et de base $t_1$ et d'une colonne de hauteur $q$ et de base $t_2$, avec $t_1$ et $t_2$ de mesures $\frac{1}{p+q}$, ce qui se trouve être l'aire des cubes de la subdivision. Soit $\Phi\in\mathrm{Auto(X,\mu)}$ (donnée par le lemme \ref{ensembles}) qui envoie $f^k(t_1)$ sur $h_m^k(C_1)$ et $f^l(t_2)$ sur $h_m^l(C_2)$ pour tout $0\le k\le p$ et $0\le l\le q$. Posons $f' = \Phi f\Phi^{-1}$. Alors tous les cubes de la subdivision ont la même image par $f'$ et $h_m$, à l'exception de $h_m^{p-1}(C_1)$ et $h_m^{q-1}(C_2)$ qui sont envoyés par $h_m$ sur respectivement $C_1$ et $C_2$ et par $f'$ sur $C_1\cup C_2$. Puisque $C_1$ et $C_2$ sont adjacents, on obtient la majoration $d_{\mathit{forte}}(f',h_m)<2\varepsilon$ et finalement $d_{\mathit{forte}}(f',h)<4\varepsilon$.
\end{proof}

\section{Généricité de l'apériodicité}

Dans la preuve du théorème de transfert nous aurons besoin de la généricité de l'apériodicité. Celle-ci se déduit facilement du théorème d'Oxtoby-Ulam. Nous en donnons toutefois une preuve directe, basée sur le lemme de Rokhlin et son \emph{alter ego} pour les applications presque partout périodiques.

\begin{prop}\label{aper}
L'ensemble des automorphismes apériodiques est un $G_\delta$ dense de l'ensemble $\mathrm{Auto(X,\mu)}$ (pour la topologie faible).
\end{prop}

\begin{lemme}[du type Rokhlin]\label{rokhlin-periodique}
Soit $E$ un ensemble mesurable de mesure 1 et $f$ un automorphisme de $E$ presque partout périodique de période $p$. Alors il existe un sous ensemble mesurable $F$ de $E$ de mesure $1/p$ et dont les $p$ premiers itérés par $f$ sont deux à deux disjoints.
\end{lemme} 

\begin{proof}[Preuve du lemme \ref{rokhlin-periodique}] Cette preuve est tirée du livre de P.~Halmos (\cite[page 70]{3}). Si $p=1$, la propriété est triviale. Sinon, il existe un ensemble mesurable $F_1$ tel que $\mu(F_1\cap f(F_1))=0$. En effet, il existe un entier $k$ tel que l'ensemble $D_k = \{x\in E \mid \mathrm{dist}(x,f(x))\ge 1/k\}$ soit de mesure strictement positive. Il existe alors une boule de diamètre $1/2k$ qui intersecte $D_k$ sur un ensemble de mesure non nulle, cette intersection est l'ensemble que l'on recherche. Si $p=2$, on s'arrête là, sinon il existe un sous ensemble $F_2$ de $F_1$ tel que $\mu(F_2\cap f^2(F_2))=0$ (sinon $f$ serait périodique de période 2 pour presque tout point de $F_1$). On continue ainsi jusqu'à la période de $f$ : on obtient un ensemble $F= F_p$ disjoint de ses $p$ premiers itérés.

Reste à montrer que l'on peut modifier la construction de manière a avoir $\mu(F)  = 1/p$. Pour cela on applique le procédé précédent au complémentaire de $F\cup f(F)\cup\cdots\cup f^{p-1}(F)$ et obtient un autre ensemble disjoint de ses $p$ premiers itérés. On itère ce procédé par induction (éventuellement transfinie) et finalement l'union $F\cup f(F)\cup\cdots\cup f^{p-1}(F)$ recouvre tout l'ensemble, ce qui est équivalent, par invariance de la mesure, au fait que $\mu(F)=1/p$.
\end{proof}

\begin{proof}[Preuve de la proposition \ref{aper}] Soit $A_k$ l'ensemble des automorphismes $f$ tels que l'ensemble des points fixes de $f^k$ soit de mesure strictement plus petite que $\frac{1}{k}$. L'ensemble des automorphismes apériodiques est alors égal à $\bigcap_{k\in\N}A_k$ (on vérifie cette égalité en considérant les $A_{k!}$).

Chaque $A_k$ est ouvert. En effet, prenons $f\in A_k$. Alors il existe $\varepsilon>0$ tel que l'ensemble des points fixes de $f^k$ soit de mesure plus petite que $\frac{1}{k}-2\varepsilon$. Le lemme de Rokhlin (lemme \ref{rokhlin}), appliqué à $m=2$ et à $g=f^k$ sur le complémentaire des points fixes de $f^k$ nous donne un ensemble $F$ tel que $F\cap f^k(F)=\emptyset$ et que $F\cup f^k(F)$ soit de mesure plus grande que $1-\frac{1}{k}+\varepsilon$ ; la mesure de $F$ est alors égale à $\frac{1}{2}(1-\frac{1}{k}+\varepsilon)$. Si on prend un autre automorphisme $g$ assez proche de $f$ pour la topologie faible, l'ensemble $ F' = F\setminus g^k(F)$ est de mesure plus grande que $\frac{1}{2}(1-\frac{1}{k})$(on utilise la seconde définition de la topologie faible), ce qui implique que $E' = F'\cup g^k(F')$ est de mesure plus grande que $1-\frac{1}{k}$. Aucun des points de $E'$ n'est $k$-périodique pour $g$, ce qui prouve que $g\in A_k$.

Nous venons de montrer que l'ensemble des automorphismes apériodiques forme un $G_\delta$ de $\mathrm{Auto}(X,\mu)$. Montrons maintenant qu'il est dense dans $\mathrm{Auto(X,\mu)}$. On peut facilement déduire ce fait du lemme d'Halmos (lemme \ref{LemHalm}), mais nous en donnons maintenant une preuve directe. Soient $f\in\mathrm{Auto(X,\mu)}$ et $\varepsilon>0$. Posons $F_1$ l'ensemble des points fixes de $f$. On applique alors le lemme \ref{point-fixe} qui nous donne un automorphisme $g_1$ de $F_1$ sans point périodique et de distance à l'identité inférieure à $\varepsilon/2$ ; on pose alors $f_1 = g_1 \circ f$, qui est sans point périodique. En appliquant le lemme \ref{rokhlin-periodique} à $f_1$, on obtient un ensemble $F_2$ disjoint de $f_1(F_2)$ tel que $F_2\cup f_1(F_2)$ soit formé exactement des points de $f_1$ de période 2. Le lemme \ref{point-fixe} nous donne alors un automorphisme $g_2$ de $F_2$ sans point périodique et de distance à l'identité plus petite que $\varepsilon/2^2$ ; on pose $f_2 = g_2 \circ f_1$. La répétition à l'infini de ce principe (la convergence est assurée par le fait qu'on modifie les automorphismes sur des ensembles dont les mesures tendant vers 0 et que la somme de ces modifications est inférieure à $\varepsilon$) nous donne un automorphisme apériodique $\varepsilon$-proche de $f$.
\end{proof}

\begin{rem}
La densité de l'ensemble des automorphismes apériodiques permet de court-circuiter la preuve du lemme d'Halmos (corollaire \ref{LemHalm}). En effet, il suffit maintenant de démontrer que de telles classes de conjugaison sont denses dans l'ensemble des automorphismes apériodiques (pour la topologie faible). Prenons donc un automorphisme apériodique $f$, un réel $\varepsilon>0$ et un second automorphisme apériodique $g$ à approcher. Prenant $k$ tel que $1/k\le \varepsilon/2$, le second point du lemme de Rokhlin (lemme \ref{rokhlin}) nous fournit deux ensembles $F$ et $G$ dont les $k$ premiers itérés par respectivement $f$ et $g$ sont deux à deux disjoints, et tels que les complémentaires des unions de ces itérés aient une mesure plus petite que $\varepsilon/2$. Quitte à en réduire un des deux, on peut supposer qu'ils ont la même mesure. Alors par le lemme \ref{ensembles}, il existe un automorphisme $\varphi$ envoyant $f^i(F)$ sur $g^i(G)$ pour $i\in\{0,\dots,k\}$. Cette application $\varphi$ est la conjugaison recherchée pour notre approximation.
\end{rem}

\section{Théorème de transfert des propriétés ergodiques génériques}

Une fois tous les résultats préparatoires démontrés, l'obtention du théorème de transfert des propriétés ergodiques, démontré par S. Alpern en 1978 \cite{AlpernGene}, est aisée.

\begin{definition}\label{propergo}
Nous appellerons \emph{propriété ergodique} toute propriété $(P)$ portant sur les éléments de $\mathrm{Auto}(X,\mu)$, telle que l'ensemble des éléments de $\mathrm{Auto}(X,\mu)$ satisfaisant $(P)$ est stable par conjugaison.
\end{definition}

\begin{theoreme}[Transfert $\mathrm{Auto(X,\mu)}$-$\mathrm{Homeo(X,\mu)}$, Alpern]\label{transfert}
Une propriété dynamique ergodique qui est générique dans $\mathrm{Auto(X,\mu)}$ relativement à la topologie faible, est aussi générique dans $\mathrm{Homeo}(X,\mu)$ relativement à la topologie de la convergence uniforme. 
\end{theoreme}

\begin{proof}[Preuve du théorème \ref{transfert}] Par le théorème \ref{transfertfaibl}, il suffit de montrer que l'adhérence de $P$ pour la topologie uniforme contient $\mathrm{Homeo}(X,\mu)$.

Puisque l'ensemble des automorphismes apériodiques est un $G_\delta$ dense (proposition \ref{aper}), par le théorème de Baire, $P$ contient au moins un automorphisme apériodique $f$ ; comme $P$ est stable par conjugaison sous $\mathrm{Auto(X,\mu)}$ et comme l'adhérence de la classe de conjugaison de $f$ dans $\mathrm{Auto(X,\mu)}$ contient $\mathrm{Homeo}(X,\mu)$ (théorème \ref{conjug-aper}), l'adhérence pour la topologie forte de $P$ dans $\mathrm{Auto}(X,\mu)$ contient $\mathrm{Homeo}(X,\mu)$.
\end{proof} 

\begin{coro}\label{20}
Toute propriété ergodique sur $\mathrm{Auto}(X,\mu)$ vérifiée par au moins un automorphisme apériodique et vraie sur un $G_\delta$ de $\mathrm{Auto}(X,\mu)$ pour la topologie faible, est générique dans $\mathrm{Homeo}(X,\mu)$ pour la topologie forte.
\end{coro}

\begin{proof}[Preuve du corollaire \ref{20}] Puisque la classe de conjugaison de tout automorphisme apé\-rio\-dique est dense dans $\mathrm{Homeo}(X,\mu)$, la propriété est vérifiée sur un ensemble dense de $\mathrm{Homeo}(X,\mu)$ ; par densité de $\mathrm{Homeo}(X,\mu)$ dans $\mathrm{Auto}(X,\mu)$ pour la topologie faible, on en déduit que cette propriété est vraie sur un ensemble faiblement dense de $\mathrm{Auto}(X,\mu)$ et donc, par hypothèse, sur un $G_\delta$ dense. Le théorème précédent affirme alors que cette propriété est vraie génériquement dans $\mathrm{Homeo}(X,\mu)$ pour la topologie forte.
\end{proof}

\section{Nouvelle preuve de la généricité du non mélange fort}

Désormais, grâce au théorème de transfert, chaque propriété dynamique générique dans l'ensemble des automorphismes le sera aussi dans l'ensemble des homéomorphismes. Ce résultat est intéressant d'un point de vue pratique\footnote{En plus de son grand intérêt philosophique.} puisque, comme on l'a dit plus haut, l'ensemble des automorphismes est plus simple à manipuler que celui des homéomorphismes ; c'est pourquoi on s'intéresse maintenant aux propriétés génériques dans $\mathrm{Auto}(X,\mu)$. On commence, dans cette partie, par donner une nouvelle preuve de la généricité du non-mélange fort ; dans la partie suivante \ref{partiefaible}, nous établissons une nouvelle preuve de la généricité du mélange faible. Ces deux résultats\footnote{Que l'on a déjà obtenues au chapitre \ref{chapcycl}.} sont des applications \og historiques \fg~du théorème de transfert : ces généricités ont été établies dans $\mathrm{Auto}(X,\mu)$ par respectivement P. Halmos \cite{HalmMix} et V. Rokhlin \cite{Rok} dans les années 40\footnote{Il est quelque peu anachronique de parler d'application historique lorsqu'on applique le théorème de transfert, prouvé dans les années 70, à des résultats datant des années 40.}\footnote{Ils avaient intitulé leurs articles \og Généralement les automorphismes préservant la mesure sont mélangeants \fg~\cite{Halm44} et \og Généralement les automorphismes préservant la mesure ne sont pas mélangeants \fg~\cite{Rok}.}. Ces preuves ont été reprises par P. Halmos lui-même dans sont livre fondateur \cite{3} au chapitre \og \emph{Category} \fg ; celles-ci se transferent directement à l'espace $\mathrm{Homeo}(X,\mu)$.

La preuve de généricité du non-mélange fort que nous présentons est issue de \cite{1}.

\begin{theoreme}[Rokhlin]\label{mélange pastop}
Génériquement, dans $\mathrm{Auto}(X,\mu)$, les éléments ne sont pas fortement mélangeants.
\end{theoreme}

\begin{proof}[Preuve du théorème \ref{mélange pastop}] On a vu que l'ensemble des permutations dyadiques est dense dans $\mathrm{Auto}(X,\mu)$ (théorème \ref{laxauto}). Grâce au lemme \ref{Pioure}, on peut même avoir la densité des permutations dyadiques cycliques d'ordre arbitrairement grand : pour tout entier $\ell$, l'ensemble $\bigcup_{k\ge \ell}P_k$ est dense, où $P_k$ est l'ensemble des permutations dyadiques cycliques d'ordre $k$.

Plaçons nous dans le cas du segment unité $I$ muni de la mesure de Lebesgue $\mathrm{Leb}$, le cas général s'en déduisant facilement par le théorème \ref{leb-stand}. Soient $A=[0,1/2]$ et $M_k$ l'ensemble des automorphismes $f$ tels que $\mu(f^k(A)\cap A)\in [1/8,3 /8]$. Cet ensemble est fermé pour la topologie faible. En effet, l'application 
\begin{eqnarray*}
\mathrm{Auto}(X,\mu) & \longrightarrow & [0,1]\\
f & \longmapsto & \mu(f^k(A)\cap A)
\end{eqnarray*}
est clairement continue relativement à la topologie faible.

Or $P_k$ est disjoint de $M_k$ pour tout $k$, donc $\bigcup_{k\ge l}P_k$ est disjoint de $\bigcap_{k\ge l}M_k$. Le premier ensemble étant dense, on en déduit que le second est d'intérieur vide et donc que l'ensemble
\[\bigcup_{l\ge 0}\bigcap_{k\ge l}M_k,\]
qui contient tous les automorphismes fortement mélangeants, est un $F_\sigma$ d'intérieur vide.
\end{proof}

Et à l'aide du théorème de transfert (théorème \ref{transfert}), on en déduit immédiatement le corollaire :

\begin{coro}\label{mélange pas}
Génériquement, dans $\mathrm{Homeo}(X,\mu)$, les éléments ne sont pas fortement mélangeants.
\end{coro}

\section{Généricité du mélange faible}\label{partiefaible}

On donne maintenant une nouvelle preuve de la généricité du mélange faible ergodique. Elle est due à P. Halmos \cite{HalmMix} et découle immédiatement du lemme d'Halmos (corollaire \ref{LemHalm}) et d'une caractérisation du mélange faible. Commençons par donner plusieurs caractérisations classiques du mélange faible (voir par exemple le paragraphe 2.6 de \cite{Petersen}) :

\begin{prop}\label{équiéqui}
Soit $f\in\mathrm{Auto}(X,\mu)$. Les conditions suivantes sont équivalentes :
\begin{enumerate}
\item $f$ est faiblement mélangeant,
\item pour tout couple de mesurables $U,V$, il existe une suite $(u_i)_{i\in\N}\in \N^\N$ de densité 1 telle que
\[\mu(f^{u_i}(U)\,\cap\, V)-\mu(U)\mu(V) \underset{i\to+\infty}{\longrightarrow}0,\]
\item Pour tous $\varphi,\psi\in L^2(X,\mu)$,
\[\frac{1}{N}\sum_{i=0}^{N-1}|\int \varphi\circ f^i\cdot\psi-\int \varphi\int \psi| \underset{N\to+\infty}{\longrightarrow}0,\]
\item $f\times f$ est ergodique,
\item les seuls vecteurs propres de l'opérateur de Koopman $U_f$ sont les fonctions constantes.
\end{enumerate}
\end{prop}

Nous aurons également besoin d'une autre caractérisation :

\begin{lemme}\label{mélange-équi}
Donnons-nous $(\varphi_i)_{i\in\N}$ une famille dénombrable dense de l'espace $L^2(X,\mu)$. Alors $f\in\mathrm{Auto}(X,\mu)$ est faiblement mélangeant si et seulement si
\[\underset{m\to+\infty}{\underline{\lim}}\left|\int \big(\varphi_i\circ f^m\big)\varphi_j - \int \varphi_i \int \varphi_j\right| = 0\]
pour tout couple $(i,j)$, c'est-à-dire :
\[\forall (i,j),\,\forall k, M,\,\exists m\ge M \,:\,\left|\int \big(\varphi_i\circ f^m\big)\varphi_j - \int \varphi_i \int \varphi_j\right|<\frac{1}{k}.\]
\end{lemme}

\begin{proof}[Preuve du lemme \ref{mélange-équi}] La propriété 2. de la proposition sur les caractérisations équivalentes du mélange faible (proposition \ref{équiéqui}) implique la propriété du lemme. Réciproquement, si on suppose $f$ non faiblement mélangeant, il va exister une fonction $\varphi$ de $L^2(X,\mu)$ non constante et un complexe $c$ tel que $\varphi\circ f = c\varphi$. Puisque cet opérateur est unitaire, $|c|=1$. Prenons $\varepsilon>0$ et $i$ tel que $\|\varphi-\varphi_i\|_{L^2}<\varepsilon$. On a alors :
\begin{eqnarray*}
|\langle \varphi_i\circ f^m,\varphi_i\rangle - \langle \varphi\circ f^m,\varphi\rangle| & \le & |\langle \varphi_i\circ f^m,\varphi_i-\varphi\rangle| + |\langle \varphi_i\circ f^m - \varphi\circ f^m,\varphi\rangle| \\
  & \le & \|\varphi_i\|_{L^2} \|\varphi_i-\varphi\|_{L^2} + \|\varphi_i-\varphi\|_{L^2} \|\varphi\|_{L^2} \\
  & \le & \varepsilon (2\|\varphi\|_{L^2} + \varepsilon).
\end{eqnarray*}
D'autre part, on obtient, en utilisant successivement l'inégalité de Cauchy-Schwarz :
\begin{eqnarray*}
\left|\Big|\int \varphi\Big|^2 - \Big|\int \varphi_i\Big|^2\right| & \le & (2+\varepsilon)\|\varphi\|_{L^1}\left| \int \varphi -\int \varphi_i\right|\\
                                                                   & \le & (2+\varepsilon)\|\varphi\|_{L^2} \|\varphi-\varphi_i\|_{L^2}\\
                                                                   & \le & \varepsilon(2+\varepsilon)\|\varphi\|_{L^2},
\end{eqnarray*}
si bien que 
\[\left|\Big|\int \varphi_i\circ f^m\varphi_i - \int \varphi_i \int \varphi_i\Big| - \Big|\int \varphi\circ f^n\varphi - \int \varphi \int \varphi\Big|\right|\le K\varepsilon,\]
avec $K$ ne dépendant que de la norme de $\varphi$. Cela nous ramène au cas de $\varphi$ : 
\begin{eqnarray*}
\left|\int \varphi\circ f^m\varphi - \int \varphi \int \varphi\right| & =   & \left|c^n\int |\varphi|^2 -  \Big|\int \varphi\Big|^2\right|\\
                                                                    & \ge & \left|\int |\varphi|^2 -  \Big|\int \varphi\Big|^2\right|.
\end{eqnarray*}
Mais par l'inégalité de Cauchy-Schwarz on a :
\[\left|\int \varphi\right|^2 < \int |\varphi|^2,\]
si bien que la quantité
\[\left|\int \varphi\circ f^m\varphi - \int \varphi \int \varphi\right|\]
est strictement positive, et cela uniformément en $m$, ce qui dit que $\varphi$ ne vérifie pas la caractérisation du mélange faible voulue.
\end{proof}

La généricité se déduit alors facilement :

\begin{theoreme}[Halmos]\label{mélange faiblemes}
Génériquement dans $\mathrm{Auto}(X,\mu)$, les éléments sont faiblement mélangeants.
\end{theoreme}

\begin{proof}[Preuve de la proposition \ref{mélange faible}] On sait qu'il existe au moins un automorphisme faiblement mélangeant (donc apériodique) : en effet, le théorème~\ref{leb-stand} nous autorise à nous placer dans le tore $\T^2$ muni de la mesure de Lebesgue et dans cet espace un automorphisme d'Anosov est faiblement mélangeant. On peut donc appliquer le corollaire \ref{20}, pourvu que l'on sache que la propriété est vraie sur un $G_\delta$ de $\mathrm{Auto}(X,\mu)$ ; mais du lemme \ref{mélange-équi} on déduit facilement que l'ensemble des éléments faiblement mélangeants est un $G_\delta$ pour la topologie faible.
\end{proof}

Et le théorème de transfert implique que :

\begin{coro}\label{mélange faible}
Génériquement dans $\mathrm{Homeo}(X,\mu)$, les éléments sont faiblement mélangeants.
\end{coro}

\section{Généricité du $\alpha$-mélange faible}\label{alphamél}

Bon nombre de propriétés génériques dans $\mathrm{Auto}(X,\mu)$ ont été trouvées à la fin des années 60 via l'étude de la vitesse d'approximation par des permutations cycliques \cite{KatokS1}, \cite{KatokS2}, \cite{Yuz}. Un peu plus tard, cette technique a été adaptée à l'étude des homéomorphismes conservatifs par A. Katok et A. Stepin \cite{KatokS3}, comme on l'a vu au chapitre \ref{chapcycl}. Ainsi, le théorème \ref{approxmeo} est aussi valable pour l'ensemble des automorphismes\footnote{La preuve est même plus facile puisque la généricité a lieu en topologie faible.} : pour une vitesse $\vartheta$ fixée, l'ensemble des éléments de $\mathrm{Auto}(X,\mu)$ admettant une approximation cyclique à la vitesse $\vartheta$ est un $G_\delta$ dense en topologie faible. Tous les théorèmes de généricité vus au chapitre \ref{chapcycl} sont donc aussi vrais pour l'espace $\mathrm{Auto}(X,\mu)$. Notons que l'étude des approximations par des permutations dans $\mathrm{Auto}(X,\mu)$ est bien plus développée que celle dans $\mathrm{Homeo}(X,\mu)$ (voir entre autres \cite{Katok}, \cite{KatokS1}, \cite{KatokS2}, \cite{KatokT}, \cite{Yuz}, etc.) ; par exemple le fait de pouvoir se ramener au cas du segment unité muni de la mesure de Lebesgue permet parfois de simplifier les preuves \cite{Yuz}. D'autre part, de même que pour les homéomorphismes, le théorème de généricité des approximations à vitesse fixée admet dans le cas des automorphismes une variante avec des approximations bicycliques, mais ce n'est pas la seule ; un grand nombre de types d'approximations peuvent être utilisées dans le théorème de généricité des approximations à vitesse fixée pour les automorphismes. Une étude théorique complète pourra être trouvée dans le premier chapitre du livre de A. Katok \cite{Katok}. Par exemple, à l'aide d'approximations par des permutations ayant un nombre arbitrairement grand d'orbites, A. Stepin a montré en 1987 \cite{Stepin} que le $\alpha$-mélange faible est générique (voir aussi \cite{Stepin2} et la partie 3.3 de \cite{Katok}). Pour d'autres propriétés, on pourra aussi consulter \cite{delJ}.

\begin{definition}
Soit $\alpha\in[0,1]$. Un automorphisme conservatif $f$ est dit $\alpha$-faiblement mélangeant s'il existe une suite strictement croissante d'entiers $\{m_k\}_{k\in\N}$ telle que pour tout couple d'ensembles mesurables $(A,B)$, on ait
\[\lim_{k\to\infty}\mu(f^{m_k}(A)\cap B) = \alpha\mu(A)\mu(B) + (1-\alpha)\mu(A\cap B).\]
\end{definition}

La proposition suivante (voir \cite{Katok}) permet d'en déduire une autre propriété d'un automorphisme générique.

\begin{prop}\label{convoétr}
Soit $f$ un automorphisme. S'il existe $\alpha\in ]0,1[$ tel que $f$ soit à spectre simple et $\alpha$-faiblement mélangeant, alors les convolutions successives du type spectral
\[\mathfrak{m}_{U_f}^{(m)} = \underbrace{\mathfrak{m}_{U_f}*\dots*\mathfrak{m}_{U_f}}_{m\ \text{fois}}\]
sont deux à deux étrangères, c'est à dire que pour $m\neq m'$, les mesures $\mathfrak{m}_{U_f}^{(m)}$ et $\mathfrak{m}_{U_f}^{(m')}$ sont étrangères.
\end{prop}

Ainsi, génériquement, les convolutions successives du type spectral sont deux à deux étrangères.

Ces généricités se transmettent bien sûr à l'ensemble des homéomorphismes\footnote{Il est possible qu'il existe des démonstrations directes de tels théorèmes de généricité, qui ne passent pas par l'ensemble $\mathrm{Auto}(X,\mu)$ et le théorème de transfert, comme on l'a fait par exemple pour le théorème \ref{approxmeo}.} par le théorème de transfert \ref{transfert} :

\begin{coro}\label{alphamélgéné}
Pour tout $\alpha\in[0,1]$, un homéomorphisme générique est $\alpha$-mélangeant. Ainsi, pour un homéomorphisme générique, les convolutions successives du type spectral sont deux à deux étrangères.
\end{coro}


\section[Généricité et type spectral]{Génériquement, le type spectral est étranger à une mesure donnée}

Une autre méthode d'obtention de propriétés génériques pour les automorphismes conservatifs est l'étude du type spectral (voir la partie \ref{saltyp}). On a vu au chapitre~\ref{chapcycl} que celui-ci donne des indications sur la dynamique d'un automorphisme. Mais l'application qui à un automorphisme associe son type spectral est continue. Un $G_\delta$ de l'ensemble des mesures boréliennes sur le cercle devient alors automatiquement un $G_\delta$ de l'ensemble des automorphismes. Par conséquent, l'étude de la dynamique générique des automorphismes se ramène parfois à celle de certains sous-ensembles de l'ensemble des mesures boréliennes du cercle : si on arrive à montrer que l'ensemble des types spectraux associés aux automorphismes vérifiant une propriété $P$ est un $G_\delta$, il suffit alors de montrer que $P$ est vraie sur un ensemble dense pour obtenir sa généricité. Ce transfert des $G_\delta$ est dû à la propriété suivante \cite{Nad} :

\begin{prop}
L'application $U\mapsto\mathfrak{m}_U$, qui va de l'ensemble des opérateurs unitaires de $\Hi$ muni de la topologie faible vers l'ensemble des mesures de probabilité boréliennes de $\mathbf S^1$ muni de la topologie faible, est continue.
\end{prop}

Nous proposons, pour illustrer de cette méthode, de montrer que le type spectral est génériquement étranger à une mesure donnée :

\begin{theoreme}\label{specortho}
\'Etant donnée une mesure de probabilité borélienne $\nu$ sur $\mathbf S^1$, le type spectral est génériquement étranger à $\nu$.
\end{theoreme}

\begin{proof}[Preuve du théorème \ref{specortho}] Posons $\Pb$ l'ensemble des mesures de probabilité boréliennes sur $\mathbf S^1$, et $\B$ l'ensemble des boréliens de $\mathbf S^1$. Soit $\nu\in\Pb$, l'ensemble~$\Pb_{\nu^\bot}$ des mesures $\mathfrak{m}\in\Pb$ étrangères à $\nu$ est égal à :
\begin{eqnarray*}
\Pb_{\nu^\bot} & = & \left\{\mathfrak{m}\in\Pb\mid \forall m\in\N,\exists U\in \B : \mathfrak{m}(U)<\frac 1 m\text{ et } \nu(U)>1-\frac 1 m\right\}\\
               & = & \bigcap_{m\in\N}\bigcup_{U\in \B} \left\{\mathfrak{m}\in\Pb\mid \mathfrak{m}(U)<\frac 1 m\text{ et } \nu(U)>1-\frac 1 m\right\},
\end{eqnarray*}
ce qui l'exprime comme un $G_\delta$ de $\Pb$. Par conséquent l'ensemble des automorphismes dont le type spectral est étranger à $\nu$ est un $G_\delta$.

Il reste à trouver un automorphisme apériodique dont le type spectral est étranger à la mesure $\nu$. Nous utilisons pour cela un très joli argument que nous a suggéré S.~Crovisier. Le théorème \ref{leb-stand} nous autorise à nous placer dans l'espace $(I,\mathrm{Leb})$.

L'idée est de fabriquer une infinité non dénombrable d'automorphismes apériodiques dont les mesures spectrales sont deux à deux étrangères. L'une au moins de ces mesures spectrales sera étrangère à $\nu$ et l'automorphisme correspondant sera l'apériodique recherché.

Considérons la relation d'équivalence $\sim$ sur $\R$, définie par $\alpha\sim\beta$ s'il existe $(k,l)\in\Z^2$ tels que $k\alpha\equiv l\beta$ mod $2\pi$. Chaque classe d'équivalence est dénombrable, par conséquent le nombre de ces classes d'équivalence est indénombrable. En utilisant l'axiome du choix, on exhibe une famille $\{\alpha_i\}_{i\in J}$ d'irrationnels, avec $J$ infini et indénombrable, telle que les $\alpha_i$ soient des représentants de classes d'équivalences deux à deux disjointes. Ainsi pour tous couples $(i,j)\in J^2$, $(k,l)\in\Z^2$, on a
\begin{equation}\label{qlibr}
k\alpha_i\not\equiv l\alpha_j\ [2\pi],
\end{equation}
On considère alors la famille de rotations $\{r_i\}_{i\in\N}$ de $I$ de nombres de rotation respectifs $\alpha_i$. Montrons que les mesures spectrales associées à ces rotations sont deux à deux étrangères.

Pour tout irrationnel $\alpha$, l'opérateur de Koopman $U_\alpha$ associé à la rotation $r_\alpha$ est diagonal dans la base hilbertienne de Fourier $\{e_k = t\mapsto e^{2i\pi kt}\}_{k\in\Z}$. La valeur propre associée au vecteur propre $t\mapsto e^{2i\pi kt}$ est alors $e^{2i\pi k\alpha}$. Posons $z_0 = e^{2i\pi \alpha}$ ; puisque les $z_0^k$ sont tous distincts ($\alpha\in \R\setminus \Q$), $U_\alpha$ est à spectre simple pour le vecteur
\[\varphi_0 = 1+ \sum_{k\in\Z^*}\frac 1{(k+1)^2} e_k.\]
Le type spectral associé à $r_\alpha$ et à $\varphi_0$ est alors une somme de masses de Dirac aux points $\{e^{2i\pi k\alpha}\}_{k\in\Z}$.

Par l'équation \ref{qlibr}, les types spectraux $\mathfrak{m}_i$ associés aux rotations $r_i$ sont deux à deux orthogonaux, car ils chargent des unions dénombrables de points deux à deux distincts.

Soit alors $\{A_i\}_{i\in J}$ les supports des mesures $\{\mathfrak{m}_i\}_{i\in J}$. La mesure $\nu$ est étrangère à au moins une mesure $\mathfrak{m}_{i_0}$, car sinon elle chargerait un ensemble infini non dénombrable de boréliens deux à deux disjoints. La rotation $r_{i_0}$ est alors un automorphisme apériodique dont le type spectral est étranger à $\nu$.
\end{proof}

La méthode de la démonstration du théorème \ref{rigid} suit un principe général d'obtention de propriétés spectrales génériques : on commence par prouver que la propriété sur les mesures associée est vraie sur exactement un $G_\delta$, donc aussi la propriété sur les automorphismes ; puis on montre qu'il existe un apériodique vérifiant cette propriété et conclut par densité de sa classe de conjugaison. Ceci donne par exemple d'autres preuves du fait qu'un automorphisme est génériquement rigide ou faiblement mélangeant. D'autres résultats, ainsi que des preuves plus détaillées, pourront être trouvés dans le livre de M. nadkarni\cite[théorème 8.34]{Nad}.
\bigskip

On peut désormais faire un point sur les propriétés du type spectral d'un homéomorphisme générique :
\begin{itemize}
\item il est sans atome (proposition \ref{speccon}),
\item il est singulier par rapport à la mesure de Lebesgue (proposition \ref{singLeb}),
\item il a ses convolutions successives deux à deux étrangères (proposition \ref{convoétr}),
\item il est étranger à toute mesure donnée \emph{a priori} (théorème \ref{specortho}),
\item il a son support égal au cercle $\mathbf S^1$ tout entier, car un homéomorphisme générique est apériodique (proposition \ref{aper}) et tout automorphisme apériodique a son type spectral de support total (voir \cite[théorème 3.5]{KatokT}).
\end{itemize}

\chapter{Loi du 0-1}

Comme nous l'avons déjà fait à la fin du chapitre précédent, nous pouvons nous consacrer à l'étude de l'espace des automorphismes conservatifs, les résultats de généricité se transmettant ensuite aux homéomorphismes par le théorème de transfert. Une propriété intéressante de l'espace $\mathrm{Auto}(X,\mu)$ est que les propriétés ergodiques y vérifient une loi du 0-1 : étant donnée une propriété ergodique, soit cette propriété est générique pour la topologie faible, soit son contraire l'est. Ce très joli résultat --- dont la preuve, une fois établies quelques propriétés sur les ensembles Baire-mesurables, est étonnamment facile --- a été obtenu en 1996 par E. Glasner et J. King \cite{4}. Une conséquence de ce théorème est le fait que toutes les classes de conjugaison dans $\mathrm{Auto}(X,\mu)$ --- et donc dans $\mathrm{Homeo}(X,\mu)$ --- sont maigres, comme l'ont montré E.~Glasner et B. Weiss en 2008 \cite{5}.

\section[Ensembles Baire-mesurables]{Préliminaires topologiques sur les ensembles Baire-mesurables}

Commençons par un rapide survol des propriétés des ensembles Baire-mesurables. Pour des précisions, nous renvoyons au chapitre 4 de \cite{7}.

\begin{definition}
On dit qu'une partie d'un espace topologique est \emph{Baire-mesurable} si elle peut s'écrire comme la différence symétrique d'une partie ouverte et d'une partie maigre.
\end{definition}

On a une définition équivalente des parties Baire-mesurables avec les fermés :

\begin{prop}\label{complementaire}
Une partie est Baire-mesurable si et seulement si elle peut s'écrire comme la différence symétrique d'une partie fermé et d'une partie maigre.
\end{prop}

\begin{proof}[Preuve de la proposition \ref{complementaire}] Soit $B = O\Delta M$ un ensemble Baire-me\-sura\-ble, avec $O$ ouvert et $M$ maigre. Posons $P = \overline{O}\setminus O$. Cet ensemble est un fermé d'intérieur vide, par conséquent $M' = P \Delta M$ est un ensemble maigre. Posant $F=\overline{O}$, on a $B = O\Delta M = (\overline{O} \Delta P)\Delta M = F\Delta M'$, ce qui l'exprime comme la différence symétrique d'une partie fermé avec une partie maigre.

Réciproquement, considérons un ensemble $B$ s'écrivant $B= F\Delta M'$, avec $F$ fermé et $M'$ maigre. Posons $O = \mathring{F}$. Alors $P = F\setminus O$ est un fermé d'intérieur vide, on en déduit que $M = P\Delta M'$ est maigre, si bien que $B = F\Delta M' = (O\Delta P)\Delta M' = O\Delta M$ s'écrit comme la différence symétrique d'un ensemble ouvert et d'un ensemble maigre et est donc Baire-mesurable.
\end{proof}

\begin{definition}
Un ouvert est dit \emph{régulier} s'il est égal à l'intérieur de son adhérence.
\end{definition}

Un lemme topologique classique permet d'y voir plus clair en ce qui concerne les ouverts réguliers :

\begin{lemme}\label{petit-topologique}
Pour toute partie $A$ d'un espace topologique, on a $\mathring{\overline{A}} = \mathring{\overline{\mathring{\overline{A}}}} = \overline{\overline{A}^\complement}^\complement$.
\end{lemme}

%

\indent Ce lemme nous fournit une vaste classe d'ouverts réguliers : les intérieurs d'ensembles fermés. La proposition suivante établit une manière canonique d'exprimer les ensembles Baire-mesurables, à l'aide des ouverts réguliers.

\begin{propdef}\label{régulier}
Toute partie $B$ Baire-mesurable d'un espace topologique de Baire s'écrit de manière unique sous la forme $B = U\,\Delta\,M$, avec $U$ un ouvert régulier et $M$ un ensemble maigre. Cette écriture est appelée l'\emph{image régulière} de la partie Baire-mesurable $B$.
\end{propdef}

\begin{proof}[Preuve de la proposition et définition \ref{régulier}] Soit $B$ une partie Baire-mesu\-ra\-ble. Montrons tout d'abord l'existence d'une image régulière de $B$. \'Ecrivons $B = O\Delta M$, avec $O$ ouvert et $M$ maigre ; on veut écrire $B$ comme la différence symétrique d'un ouvert régulier et d'un ensemble maigre. Notons $U=\mathring{\overline{O}}$ le candidat naturel\footnote{Par le lemme \ref{petit-topologique} !} pour l'ouvert régulier et $N = U\setminus O$ (on notera que $O$ est inclus dans $U$). L'ensemble $N$ est d'intérieur vide (car $\mathring N = \mathring U\setminus \overline O = \emptyset$), mais \emph{a priori} il n'est pas fermé. Vérifions que $N$ est néamoins assez petit pour que $M' = M\Delta \overline N$ soit maigre. On a
\[\mathring{\overline N} \subset\mathring{\overline U}\setminus O = U\setminus O = N,\]
par conséquent
\[\mathring{\overline N}\subset \mathring N = \emptyset,\]
donc $\overline N$ est d'intérieur vide. De plus, d'une part
\[U\setminus \overline{N}\subset U\setminus N \subset O,\]
et d'autre part
\[O = U\cap O = U\setminus(\overline{U}\setminus O) \subset U\setminus (\overline{U\setminus O}) = U\setminus \overline{N}.\]
On en déduit que $O = U\setminus \overline{N}$. Finalement, $M'$ est maigre et $B = U\Delta M'$ s'écrit comme la différence symétrique d'un ouvert régulier et d'un ensemble maigre.

Passons à l'unicité d'une telle écriture. Supposons que l'on ait $U\Delta M = V\Delta N$, avec $U$ un ouvert régulier, $V$ un ouvert et $M$ et $N$ des ensembles maigres. Alors $V\setminus \overline{U}\subset V\setminus U\subset V\Delta U$. Mais puisque $U\Delta M = V\Delta N$, on a, de manière générale, $V\Delta U = N\Delta M$. En effet, soit $x\in V\Delta U$. Par symétrie des rôles joués par $U$ et $V$, on peut supposer que $x\in V\setminus U$. On a alors deux cas :
\begin{itemize}
\item soit $x\in N$, donc $x\notin V\Delta N = U\Delta M$, mais $x\notin U$, donc $x\notin M$,
\item soit $x\notin N$, donc $x\in V\Delta N = U\Delta M$, mais $x\notin U$, donc $x\in M$.
\end{itemize} 
Dans les deux cas $x\in M\Delta N$, et on en déduit que $V\Delta U\subset M\Delta N$. L'autre inclusion se montre de manière semblable.\\
Nous avons donc montré que $V\setminus \overline{U} \subset M \Delta N$ ; par conséquent le terme de gauche est un ensemble ouvert et maigre. Par théorème de Baire, il est vide ; on en déduit que $V\subset\overline{U}$, et puisque $U$ est un ouvert régulier, $V\subset \mathring{\overline{U}} = U$. Par symétrie des rôles, cela implique que $V=U$, et par l'égalité $V\Delta U = N\Delta M$ vue plus haut, on en déduit l'unicité.
\end{proof}

Enfin, montrons que les parties Baire-mesurables forment un ensemble assez gros et maniable : en particulier il contient l'ensemble des boréliens.

\begin{prop}\label{tribu}
L'ensemble des parties Baire-mesurables d'un espace topologique est une tribu ; plus précisément c'est la tribu engendrée par les ensembles ouverts et les ensembles maigres. En particulier, tout ensemble borélien est Baire-mesurable.
\end{prop}

\begin{proof}[Preuve de la proposition \ref{tribu}] L'ensemble des parties Baire-mesurables est trivialement contenu dans la tribu engendrée par les ouverts et les sous-ensembles maigres ; on voit tout aussi facilement qu'il contient tous les ouverts et tous les ensembles maigres. Il ne reste plus qu'à montrer que c'est une tribu.

Pour montrer la stabilité par passage au complémentaire, donnons-nous un ensemble Baire-mesurable $B$, que l'on écrit $B = O\Delta M$, avec $O$ ouvert et $M$ maigre. Alors $B^\complement = O^\complement\Delta M$ ; par conséquent $B$ s'écrit comme la différence symétrique d'un fermé et d'un ensemble maigre, et en appliquant la proposition \ref{complementaire}, on en déduit qu'il est Baire-mesurable.

Passons à la stabilité par union dénombrable. Prenons $(B_i)_{i\in\N}$ une suite de sous ensembles Baire-mesurables. On peut donc écrire $B_i = O_i\Delta M_i$, avec $O_i$ ouvert et $M_i$ maigre. Posons $B=\bigcup B_i$, $O=\bigcup O_i$ et $M=\bigcup M_i$. Alors $O$ est ouvert et $M$ est maigre. D'autre part, pour tout $i$, $O_i\setminus M_i\subset B_i\subset O_i\cup M_i$. L'inclusion $\bigcup O_i \setminus \bigcup M_i \subset \bigcup (O_i\setminus M_i)$ permet d'en déduire que $O\setminus M \subset B \subset O\cup M$. Par conséquent, l'ensemble $B\Delta O$ est maigre car inclus dans $M$, et donc $B = O\Delta (B\Delta O)$ peut s'écrire comme la différence symétrique d'un ensemble ouvert et d'un ensemble maigre.
\end{proof}

\section{Loi du 0-1 dans $\mathrm{Auto}(X,\mu)$}

Forts de ces considérations sur les ensembles Baire-Mesurables, nous sommes prêts à énoncer et démontrer la loi du 0-1 de E. Glasner et J. King :

\begin{theoreme}[Glasner, King, \cite{4}]\label{01}
Soit $A$ un espace topologique de Baire\footnote{Au sens usuel du terme, c'est-à-dire tel que toute intersection dénombrable d'ouverts denses de $A$ est dense.} et $\Gamma$ un sous-groupe du groupe des homéomorphismes de $A$. S'il existe un élément $f$ de $A$ dont l'orbite par $\Gamma$ est dense dans $A$, alors toute partie $B$ Baire-mesurable et $\Gamma$-invariante de $A$ est soit maigre soit grasse.
\end{theoreme}

\begin{proof}[Preuve du théorème \ref{01}] Puisque $B$ est Baire-mesurable, par la proposition et définition \ref{régulier}, on peut l'écrire selon son image régulière : 
\[B = U\,\Delta\,M,\]
où $U$ est un ouvert régulier et $M$ un ensemble maigre. Soit $\varphi\in\Gamma$. Par invariance on a :
\[B = \varphi(B) = \varphi(U)\,\Delta\,\varphi(M).\]
Puisque $\varphi$ est un homéomorphisme, $\varphi(U)$ est un ouvert régulier et $\varphi(M)$ est maigre. L'unicité de l'image régulière (proposition et définition \ref{régulier}) implique alors que $U$ et $M$ sont tous deux $\Gamma$-invariants.

Si $U$ est vide, alors $B=M$ et par conséquent $B$ est maigre. Sinon $U$ intersecte non trivialement l'orbite dense de $f$ ; par invariance $U$ est dense. Puisque $U$ est un ouvert régulier, c'est l'espace $A$ tout entier et donc $B=M^\complement$, et de cela on déduit que $B$ est résiduel.
\end{proof}

\begin{definition}
Une propriété ergodique sur $\mathrm{Auto}(X,\mu)$ (voir la définition \ref{propergo}) sera dite \emph{raisonnable} si l'ensemble des éléments la vérifiant est Baire-mesurable (en particulier s'il est borélien par la proposition \ref{tribu}).
\end{definition}

\begin{coro}[Loi du 0-1, Glasner, King, \cite{4}]\label{loi-auto}
Soit $(P)$ une propriété ergodique raisonnable portant sur les éléments de $\mathrm{Auto}(X,\mu)$. Alors soit $(P)$ est générique, soit son contraire l'est.
\end{coro}

\begin{proof}[Preuve du corollaire \ref{loi-auto}] On applique la loi du 0-1 à l'espace $A = \mathrm{Auto}(X,\mu)$ muni de la topologie faible (c'est un espace de Baire, voir l'appendice) et au groupe $\Gamma$ des automorphismes intérieurs de $A$. Tout élément apériodique de $\mathrm{Auto}(X,\mu)$ peut jouer le rôle de l'élément dont l'orbite est dense : c'est le lemme d'Halmos (corollaire~\ref{LemHalm}). Les hypothèses du théorème \ref{01} sont bien satisfaites.
\end{proof}

\begin{rem}\label{pludid}
Par la remarque \ref{PasHalmHomeo}, cette preuve ne s'applique pas directement à l'espace $\mathrm{Homeo}(X,\mu)$, au moins pour certaines variétés $X$.

En fait, si $X$ est une surface fermée différente de la sphère $\mathbf S^2$, on peut trouver une propriété dynamique\footnote{Attention, on entend ici par propriété dynamique une propriété invariante par conjugaison dans $\mathrm{Homeo}(X,\mu)$ ; ce n'est pas pas une propriété \emph{ergodique}.} raisonnable portant sur les éléments de $\mathrm{Homeo}(X,\mu)$, qui ne vérifie pas de loi du 0-1 : elle n'est pas générique, et son contraire ne l'est pas non plus. Pour définir une telle propriété, il suffit de considérer le quasi-morphisme $\varphi$ évoqué dans la remarque \ref{PasHalmHomeo}. On prend un homéomorphisme $f$ et un réel $\epsilon$ tel que $0<\varepsilon<|\varphi(f)|$ ; et considère la propriété \og appartenir \`a l'ensemble $\varphi^{-1}([\epsilon,+\infty])$~\fg. L'ensemble des homéomorphismes vérifiant cette propriété est un fermé contenant un ouvert non-vide (car il contient $f$) différent de $\mathrm{Homeo}(X,\mu)$ (car l'identité est dans son complémentaire). Cet ensemble est de plus invariant par conjugaison.
\end{rem}

\begin{rem}
Observons que le théorème \ref{01} permet par contre de montrer une loi du 0-1 pour les propriétés dynamiques dans l'espace des homéomorphismes dissipatifs (c'est-à-dire sans hypothèse de préservation de mesure) de la sphère : l'ensemble des homéomorphismes de $\mathbf{S}^n$ satisfaisant une propriété dynamique\footnote{Par dynamique on entend ici une propriété invariante par conjugaison dans l'ensemble des homéomorphismes dissipatifs.} est soit gras soit maigre. En effet, d'après le théorème \ref{01}, il suffit de construire un homéomorphisme dont la classe de conjugaison est dense, on peut par exemple procéder comme suit.

Par séparabilité, il existe une famille dénombrable dense  $\{f_m\}_{m\in\N}$ d'homéomorphismes de $\mathbf{S}^n$ ; pour tout $m$, l'homéomorphisme $f_m$ fixe au moins un point de $\mathbf{S}^n$. Quitte à perturber un peu $f_m$, on peut supposer que $f_m$ est égal à l'identité sur une petite boule fermée $B_m$ (on choisit la boule $B_m$ assez petite pour que le fait de modifier $f_m$ sur $B_m$ ne change rien à la densité de la famille $\{f_m\}$). On consid\`ere alors une famille $\{D_m\}_{m\in\N}$ de boules fermées d'intérieurs non vides et deux à deux disjointes dans $\mathbf{S}^n$. Pour tout entier $m$, il existe un homéomorphisme $h_m$ de $\mathbf{S}^n$ (on peut prendre un homéomorphisme de M\"obius, mais peu importe) qui envoie $D_m$ sur $\mathbf{S}^n\setminus\mathring{B_m}$. On définit l'homéomorphisme $g$ égal à l'identité sur $\mathbf{S}^n\setminus\bigcup_m D_m$ et à $h_m^{-1}f_mh_m$ sur chacun des $D_m$.

Alors pour tout $m$, l'homéomorphisme $h_mgh_m^{-1}$ coïncide avec $f_m$ sur $\mathbf{S}^n\setminus\mathring{B_m}$. On a choisi les familles $\{f_m\}$ et $\{B_m\}$ de manière à ce que cela suffise à garantir la densité de la famille $\{h_mgh_m^{-1}\}$ dans l'ensemble des homéomorphisme de $\mathbf{S}^n$. En particulier, la classe de conjugaison de $g$ est dense dans l'ensemble des homéomorphisme de $\mathbf{S}^n$.
\end{rem}

Par le théorème de transfert, le corollaire \ref{loi-auto} possède un analogue pour les propriétés \emph{ergodiques} des homéomorphismes conservatifs :

\begin{coro}
Soit $(P)$ une propriété ergodique\footnote{Insistons sur le fait que le corollaire n'est valable que pour des propriétés \emph{ergodiques} : celles-ci doivent être invariantes par conjugaison dans $\mathrm{Auto}(X,\mu)$ et pas seulement dans $\mathrm{Homeo}(X,\mu)$ (sinon cela contredirait la remarque \ref{pludid}).} raisonnable portant sur les éléments de $\mathrm{Homeo}(X,\mu)$. Alors dans $\mathrm{Homeo}(X,\mu)$, soit $(P)$ est générique, soit son contraire l'est.
\end{coro}

On peut déduire de ce corollaire une propriété assez instructive sur les propriétés ergodiques dans les ensembles $\mathrm{Homeo}(X,\mu)$ :

\begin{coro}\label{indp}
La généricité ou non d'une propriété ergodique raisonnable sur $\mathrm{Homeo}(X,\mu)$ est indépendante de la variété $X$ et de la bonne mesure $\mu$ choisis.
\end{coro}

\begin{proof}[Preuve du corollaire \ref{indp}] Soit $P$ un $G_\delta$ dense de $\mathrm{Homeo}(X,\mu)$ et $\widetilde P$ l'ensemble correspondant dans $\mathrm{Auto}(X,\mu)$, constitué des conjugués par $\mathrm{Auto}(X,\mu)$ des éléments de $P$. Par la loi du 0-1, $\widetilde P$ est soit gras, soit maigre, mais puisque $P$ est lui-même gras, le théorème de transfert implique que $\widetilde P$ est gras. Du théorème \ref{leb-stand} on déduit que la propriété correspondant à $\widetilde P$ est indépendante de l'espace $X$ et de la mesure $\mu$, et par théorème de transfert la propriété associée à $P$ est indépendante de $X$ et de $\mu$.
\end{proof}

\section{Dans $\mathrm{Homeo}(X,\mu)$, chaque classe de conjugaison est maigre}

Comme application de la loi du 0-1, nous allons montrer que toute classe de conjugaison dans $\mathrm{Homeo}(X,\mu)$ est maigre ; ce résultat est dû à E. Glasner et B. Weiss \cite{5}.

\begin{theoreme}[Glasner, Weiss]\label{classe-maigre}
Dans $\mathrm{Auto}(X,\mu)$, toute classe de conjugaison est maigre.
\end{theoreme}

On peut déduire facilement ce théorème du théorème \ref{specortho} comme suit : on prend un automorphisme $f$, on peut commencer par supposer qu'il est à spectre simple (sinon sa classe de conjugaison, contenue dans l'ensemble des automorphismes qui ne sont pas à spectre simple, serait maigre). Tout élément dans la classe de conjugaison de $f$ aura un type spectral équivalent à $\mathfrak{m}_f$, celui-ci ne sera donc pas étranger à $\mathfrak{m}_f$ ; on conclut par le théorème \ref{specortho}.

Nous présentons néanmoins une preuve directe du théorème \ref{classe-maigre}, dont la démonstration est basée sur celle présentée par E. Glasner et B. Weiss (théorème 3.4 de \cite{5}). On y fait appel à un résultat de théorie des ensembles (théorème \ref{inverse}), similaire au théorème de Jankov et Von Neumann (théorème 29.9 de \cite{11}).

\begin{theoreme}\label{inverse}
Soient $X$ et $Y$ deux espaces polonais\footnote{Voir la définition \ref{polo}.} et $\varphi : X\to Y$ une application continue. On note $E = \varphi(X)$. Alors il existe une fonction Baire-mesurable (c'est-à-dire mesurable pour la tribu formée des ensembles Baire-me\-su\-rables) $\psi : E\to X$ telle que $\varphi\circ\psi = \mathrm{Id}_E$.
\end{theoreme}

Une preuve de ce théorème est faite dans \cite{5}. Il va nous permettre d'\og inverser \fg~l'application de conjugaison.

\begin{proof}[Preuve du théorème \ref{classe-maigre}] Le théorème \ref{leb-stand} nous permet de supposer que $X$ est l'ensemble $I$ muni de sa tribu borélienne $\mathcal B$ et que $\mu$ est la mesure de Lebesgue $\mathrm{Leb}$. Raisonnons par l'absurde : supposons qu'il existe un élément $f_0\in\mathrm{Auto}(X,\mu)$ dont la classe de conjugaison sous $\mathrm{Auto}(I,\mathrm{Leb})$, notée $C(f_0)$, est grasse. Le théorème précédent (théorème \ref{inverse}) appliqué à la fonction 
\begin{eqnarray*}
\varphi : \mathrm{Auto}(I,\mathrm{Leb})& \longrightarrow     &\mathrm{Auto}(I,\mathrm{Leb})\\
R                   & \longmapsto & R^{-1} f_0 R
\end{eqnarray*}
nous fournit une application Baire-mesurable
\[\psi : C(f_0)\longrightarrow \mathrm{Auto}(I,\mathrm{Leb})\]
qui vérifie $\varphi\circ \psi = \mathrm{Id}_{C(f_0)}$. Ainsi, pour tout $f\in C(f_0)$ et $N\in\Z$, on a :
\begin{equation}\label{1}
\psi(f)f^N = f_0^N\psi(f).
\end{equation}

Considérons maintenant l'espace polonais $\mathcal{B}_{1/2}$, sous-ensem\-ble de $\mathcal{B}$, constitué des éléments de $\mathcal B$ de mesure $1/2$. Le fait que cet espace est polonais vient tout simplement du fait que c'est la sphère de centre l'ensemble vide et de rayon $1/2$ dans $\mathcal{B}$. Posons $D_0 = [0,1/2]$ et prenons une famille dénombrable $\{D_i\}_{i\ge1}$ de $\mathcal{B}_{1/2}$ qui est $1/20$-dense dans $\mathcal{B}_{1/2}$. Pour tout $i\ge 1$, notons
\[\mathcal{E}_i = \big\{f\in C(f_0)\mid\mathrm{Leb}(\psi(f)D_0\Delta D_i)<1/20\big\}.\]
Chaque $\mathcal E_i$ est une partie baire-mesurable de $\mathrm{Auto}(I,\mathrm{Leb})$ ; en effet les applications
\[\begin{array}{rclcrcl}
\mathrm{Leb} : \mathcal{B} & \longrightarrow     & [0,1]  & \quad \mathrm{et}\quad & \mathrm{Auto}(I,\mathrm{Leb}) & \longrightarrow & \mathcal{B}\\ 
      D               & \longmapsto     & \mathrm{Leb}(D) &                        & f                        & \longmapsto     & f(D_0)\Delta D_i
\end{array}\]
sont continues ; on conclut par le fait que $\psi$ est Baire-mesurable. La famille $\{\mathcal E_i\}_{i\ge 1}$ constitue un recouvrement dénombrable de l'ensemble $C(f_0)$ ; cela découle facilement du fait que $\psi(f)$ préserve la mesure. Nous pouvons donc écrire $\mathcal{E}_i = U_i\Delta M_i$ pour tout entier $i$, avec $U_i$ ouvert et $M_i$ maigre. Alors, par le théorème de Baire, et comme $C(f_0)$ est gras, il existe au moins un $i$ tel que $U_i$ soit non vide. Posons $U_0 = U_i\cap \mathcal{E}_i = U_i\setminus M_i$. Alors, pour $f_1, f_2\in U_0$, nous avons :
\begin{equation}\label{2}
\mathrm{Leb}\left(\psi(f_1)D_0 \Delta \psi(f_2)D_0\right)<1/10.
\end{equation}

Cette inégalité signifie que pour tout $f_1,f_2\in U_0$, les actions de $f_1$ et de $f_2$ sur $D_0$ sont assez proches. On va maintenant montrer que l'on peut trouver deux automorphismes $f_1$ et $f_2$ dont les actions sur $D_0$ sont assez différentes, cela fournira la contradiction désirée. Plus précisément, on imposera à $f_1$ d'être ergodique, ce qui entrainera que l'intersection $f_1^N (D_0)\cap D_0$ ne sera pas trop grosse pour une infinité d'entiers $N$ ; et on choisira $f_2$ telle que l'intersection $f_2^N (D_0)\cap D_0$ soit assez grosse là aussi pour une infinité d'entiers $N$.

Construisons un tel couple d'automorphismes $f_1,f_2$. Puisque l'ensemble des automorphismes ergodiques est générique, on peut choisir une application ergodique $f_1\in U_0$. Par théorème ergodique de Von Neumann, 
\[\frac{1}{N}\sum_{j=0}^{N-1}\mathrm{Leb}(f_1^j D_0 \cap D_0)\underset{N\to+\infty}{\longrightarrow} 1/4,\]
si bien qu'il existe une suite croissante $N_i\to +\infty$ telle que :
\begin{equation}\label{3}
\mathrm{Leb}\left(f_1^{N_i} D_0 \cap D_0\right)<1/3.
\end{equation}

Montrons maintenant que la limite supérieure des ensembles
\[A_i = \big\{f\in \mathrm{Auto}(I,\mathrm{Leb})\mid\mathrm{Leb}(f^{N_i}D_0\Delta D_0)<1/10\big\}\]
est un $G_\delta$ dense. Chaque $A_i$ est clairement ouvert (par la remarque sur la continuité au dessus). La densité de tous les ensembles de la forme $\bigcup_{j\ge m}A_j$, avec $m$ parcourant $\N$, est une conséquence du lemme de Rokhlin : soient un automorphisme apériodique $g\in\mathrm{Auto}(I,\mathrm{Leb})$, $\varepsilon>0$ et $m\in\N$. Prenant $N_i\ge m$, on obtient une tour de Rokhlin $(B,g(B),\cdots,g^{N_i-1}(B))$, avec
\[\mathrm{Leb}\left(\bigcup_{j=0}^{N_i-1} g^j(B)\right)>1-\frac{\varepsilon}{2}.\]
On définit alors l'automorphisme $f$ par 
\[f=\left\{\begin{array}{ll}
g & \quad \text{sur}\quad \cup_{j=0}^{N_i-2}g^j(B)\\
g^{-N_i+1} & \quad \text{sur}\quad g^{N_i-1}(B)\\
\mathrm{Id} & \quad \text{ailleurs.}
\end{array}\right.\]
Quitte à augmenter $i$, on suppose $1/N_i<\varepsilon/2$ et par conséquent $f$ est une $\varepsilon$-approximation $N_i$-périodique de $g$, donc $f\in\cup_{j\ge m}A_j$. On conclut par densité des automorphismes apériodiques. Ainsi, $A= \cap_{m\ge 1}\cup_{j\ge m}A_j$ la limite supérieure des $A_j$ est un $G_\delta$ dense ; on peut donc prendre $f_2\in A\cap U_0$.

Puisque $f_2$ est dans $A$, il existe un entier $i$ tel que $\mathrm{Leb}\big(f_2^{N_i}D_0\Delta D_0\big)<1/10$. Or $\psi(f_2)$ préserve la mesure, cela implique que $\mathrm{Leb}\big(\psi(f_2)f_2^{N_i}D_0\Delta \psi(f_2)D_0\big)<1/10$, et par (\ref{1}),
\begin{equation}\label{4}
\mathrm{Leb}\left(f_0^{N_i}\psi(f_2)D_0\Delta \psi(f_2)D_0\right)<1/10.
\end{equation}

Or de manière générale, pour deux ensembles mesurables $M$ et $N$, on a :
\[2\,\mathrm{Leb}(M\cap N) = \mathrm{Leb}(M) + \mathrm{Leb}(N) - \mathrm{Leb}(M\Delta N),\]
si bien que :
\begin{eqnarray*}
\mathrm{Leb}\left(f_0^{N_i} \psi(f_1)D_0 \cap \psi(f_1)D_0\right) & = & \frac 12 \bigg(\mathrm{Leb}\left(f_0^{N_i} \psi(f_1)D_0\right) + \mathrm{Leb}\left(\psi(f_1)D_0\right)\\
                                                             &   & - \mathrm{Leb}\left(f_0^{N_i} \psi(f_1)D_0 \Delta \psi(f_1)D_0\right)\bigg).
\end{eqnarray*}
Les deux premiers termes du membre de droite valent $1/2$. Pour majorer le troisième terme, on utilise l'inclusion
\[E\Delta F\subset (E\Delta G)\cup (G\Delta H) \cup (F\Delta H),\]
valable quels que soient les ensembles $E$, $F$ et $G$. On en déduit :
\begin{eqnarray*}
\mathrm{Leb}\left(f_0^{N_i} \psi(f_1)D_0 \cap \psi(f_1)D_0\right) & \ge & \frac 12 - \frac 12\bigg(\mathrm{Leb}\left(f_0^{N_i} \psi(f_1)D_0\Delta f_0^{N_i} \psi(f_2)D_0\right)\\
                                                             &   & + \mathrm{Leb}\left(f_0^{N_i}\psi(f_2)D_0\Delta \psi(f_2)D_0\right)\\
                                                             &   & + \mathrm{Leb}\left(\psi(f_1)D_0 \Delta \psi(f_2)D_0\right)\bigg).
\end{eqnarray*}
Et en utilisant les équations \ref{2} et \ref{4}, on obtient :
\begin{eqnarray}\label{5}
\mathrm{Leb}\left(f_0^{N_i} \psi(f_1)D_0 \cap \psi(f_1)D_0\right) & \ge & \frac 12 -\frac 12\left(\frac{1}{10} + \frac{1}{10} + \frac{1}{10}\right)\nonumber\\
                                                             & \ge & \frac{7}{20}>\frac 13.
\end{eqnarray}

D'autre part, par (\ref{3}), $\mathrm{Leb}\big(f_1^{N_i} D_0 \cap D_0\big)<1/3$, si bien que, toujours par préservation de la mesure, $\mathrm{Leb}\big(\psi(f_1)f_1^{N_i} D_0 \cap \psi(f_1)D_0\big)<1/3$ et de nouveau, par (\ref{1}),
\begin{equation}\label{6}
\mathrm{Leb}\left(f_0^{N_i} \psi(f_1)D_0 \cap \psi(f_1)D_0\right)<1/3.
\end{equation}
On obtient une contradiction en comparant les équations \ref{5} et \ref{6}. Donc l'ensemble $C(f_0)$ n'est pas gras et par conséquent, par la loi du 0-1, il est maigre.
\end{proof}

En appliquant le théorème de transfert \ref{transfert} et en remarquant que la trace sur $\mathrm{Homeo}(X,\mu)$ d'une classe de conjugaison d'un homéomorphisme sous $\mathrm{Auto}(X,\mu)$ contient la classe de conjugaison de cet homéomorphisme sous $\mathrm{Homeo}(X,\mu)$, on obtient le corollaire suivant :

\begin{coro}\label{homeomaigr}
Dans $\mathrm{Homeo}(X,\mu)$, toutes les classes de conjugaisons sont maigres.
\end{coro}

\appendix
\chapter{Propriétés des espaces $\mathrm{Homeo}(X,\mu)$ et $\mathrm{Auto}(X,\mu)$}\label{AAA}

Le but de cette annexe est d'établir les propriétés importantes des espaces topologiques $\mathrm{Homeo}(X,\mu)$ et $\mathrm{Auto}(X,\mu)$. Nous montrerons en particulier que ces espaces sont de Baire\label{debutbaire}\footnote{Le théorème de Baire établit que dans un espace complet, une intersection d'ouverts denses est elle-même dense. Un espace est dit \emph{de Baire} si la conclusion de ce théorème reste vraie.}.

\section{L'espace $\mathrm{Homeo}(X,\mu)$ et la topologie uniforme}

En premier lieu, remarquons que l'espace $\mathrm{Homeo}(X,\mu)$ est un sous-en\-sem\-ble fermé de l'ensemble des homéomorphismes. Pour notre étude, il est fondamental que cet espace soit polonais :

\begin{definition}\label{polo}
Un espace topologique est dit \emph{polonais} s'il séparable et complètement métrisable, c'est-à-dire métrisable par une métrique complète.
\end{definition}

Une propriété importante\footnote{En tous cas pour nous.} des espaces polonais est qu'on peut y appliquer le théorème de Baire et cela bien sûr pour toutes les topologies compatibles avec la métrique initiale de l'espace (pour des précisions sur les espaces polonais, voir \cite{ana}). Toutefois, il faut prendre garde au fait que toutes les métriques ne rendent pas l'espace complet ; dans notre cas la convergence des suites de Cauchy n'est pas assurée dans le cadre de la métrique usuelle $d_{\mathit{forte}}$ de $\mathrm{Homeo}(X,\mu)$. Une métrique qui le rend complet est donnée par \label{delta}
\[\delta(f,g) = d_{\mathit{forte}}(f,g)+d_{\mathit{forte}}(f^{-1},g^{-1}).\]

\begin{lemme}\label{gpe-topo}
L'espace $\mathrm{Homeo}(X,\mu)$ muni de la métrique $d_{\mathit{forte}}$ est un groupe topologique.
\end{lemme}

\begin{proof}[Preuve du lemme \ref{gpe-topo}] Montrons la continuité de la composition. Soient $f$, $g$, $f'$ et $g'$ dans $\mathrm{Homeo}(X,\mu)$, tels que $d_{\mathit{forte}}(f,f')<\eta$ et $d_{\mathit{forte}}(g,g')<\eta$. Alors
\[d_{\mathit{forte}}(f\circ g,f'\circ g')\le d_{\mathit{forte}}(f\circ g, f\circ g')+d_{\mathit{forte}}(f\circ g',f'\circ g').\]
Le second terme est inférieur à $\eta$. Pour le premier il suffit de remarquer que $f$ est uniformément continue si bien qu'il va exister $\eta$ tel que la distance soit plus petite que $\varepsilon$. On a donc la continuité de la composition.

Passons à celle du passage à l'inverse. Donnons-nous deux éléments de $\mathrm{Homeo}(X,\mu)$ $f$ et $g$. Alors
\begin{eqnarray*}
d_{\mathit{forte}}(f^{-1},g^{-1}) & = & d_{\mathit{forte}}(f^{-1}\circ g,\mathrm{Id})\\
                                 & = & d_{\mathit{forte}}(f^{-1}\circ g,f^{-1}\circ f).
\end{eqnarray*}
Et on conclut par uniforme continuité de $f^{-1}$.
\end{proof}

\begin{lemme}\label{complet}
L'espace $\mathrm{Homeo}(X,\mu)$ muni de la métrique $\delta(f,g)$ est complet.
\end{lemme}

\begin{proof}[Preuve du lemme \ref{complet}] Pour cela, il suffit de montrer qu'il est fermé dans l'ensemble des applications continues de $X$ dans $X$ (sans préservation du volume), car la métrique utilisée est plus fine que la métrique uniforme.
Prenons une suite de Cauchy $(f_m)_{m\in\N}$ d'éléments de $\mathrm{Homeo}(X,\mu)$ con\-ver\-geant vers $f$ continue. De même la suite $f_m^{-1}$ converge vers une fonction $g$ elle aussi continue. Nous voulons alors montrer que $f$ est inversible, on établit pour cela que $f\circ g = g\circ f=\mathrm{Id}$, mais ces égalités découlent directement du passage à la limite des égalités $f_m\circ f_m^{-1} = f_m^{-1}\circ f_m = \mathrm{Id}$, à l'aide du lemme \ref{gpe-topo}.
\end{proof}

\begin{lemme}\label{séparable}
L'espace $\mathrm{Homeo}(X,\mu)$ muni de la métrique $d_{\mathit{forte}}$ est séparable.
\end{lemme}

\begin{proof}[Preuve du lemme \ref{séparable}] Puisque $X$ est une variété différentielle, il va exister une triangulation de $X$ (voir \cite{munk}), et même des triangulations arbitrairement fines. On en déduit la densité des applications affines par morceaux de $X$ dans $X$ parmi les applications continues de $X$ dans $X$. Dans l'ensemble des applications affines par morceaux, il existe une famille dénombrable dense. Mais cela implique que l'ensemble des applications continues de $X$ dans $X$ est à base dénombrable d'ouverts. Par inclusion de $\mathrm{Homeo}(X,\mu)$ dans cet espace (ils sont munis de la même topologie), on en déduit que $\mathrm{Homeo}(X,\mu)$ est à base dénombrable d'ouverts.
\end{proof}

Du lemme \ref{gpe-topo}, on déduit le corollaire suivant :

\begin{coro}\label{derdesder}
Les deux distances définies sur l'espace $\mathrm{Homeo}(X,\mu)$ engendrent la même topologie.
\end{coro}

Et en combinant les lemmes \ref{complet} et \ref{séparable} et le corollaire \ref{derdesder}, on obtient :

\begin{coro}\label{homeopolo}
$\mathrm{Homeo}(X,\mu)$ est polonais. En particulier on peut y appliquer le théorème de Baire.
\end{coro}

\section{L'espace $\mathrm{Auto}(X,\mu)$ et la topologie faible}

Nous donnons ici les propriétés importantes pour notre étude de l'espace $\mathrm{Auto}(X,\mu)$. Comme pour l'espace $\mathrm{Homeo}(X,\mu)$, il est fondamental que l'espace $\mathrm{Auto}(X,\mu)$ muni de la topologie faible soit polonais. Les propriétés à retenir dans cette partie sont les définitions équivalentes de la topologie faible (lemme \ref{topeq}), le fait que $\mathrm{Auto}(X,\mu)$ est polonais (lemme \ref{autopolo}), quelques constructions d'automorphismes (lemmes \ref{ensembles} et \ref{point-fixe}) et le théorème d'isomorphisme des espaces $\mathrm{Auto}(X,\mu)$.

Rappelons que la topologie faible est définie par la distance
\[d_{\mathit{faible}}(f,g) = \inf\big\{\alpha \mid \mu\{x\mid \mathrm{dist}( f(x),g(x))>\alpha\}<\alpha\big\}.\]
Il sera utile par la suite d'avoir une formulation équivalente de la topologie faible, donnée par le lemme suivant :

\begin{lemme}\label{topeq}
Sur $\mathrm{Auto}(X,\mu)$, la topologie faible est équivalente à celle définie comme suit : $f_k$ converge vers $f$ si et seulement si pour tout ensemble mesurable $A$, $\mu(f_k(A) \,\Delta\, f(A))$ tend vers 0.
\end{lemme}

\begin{proof}[Preuve du lemme \ref{topeq}] Supposons que $d_{\mathit{faible}}(f_k,f)$ tende vers 0 et donnons nous un cube $C$ (cette notion est licite au vu du corollaire \ref{Brown-mesure}), qui va dans un premier temps jouer le rôle du mesurable $A$, ainsi qu'un $\varepsilon>0$. On pose $C_\varepsilon$ l'ensemble des points situés à une distance $\mathrm{dist}$ inférieure à $\varepsilon$ de $C$.
Le théorème de Lusin (pour un énoncé dans un cadre plus général que celui \og classique \fg~de $\R$, voir par exemple \cite[lemme 2.4.2]{Ito}) assure qu'il existe un homéomorphisme conservatif $g$ qui coïncide avec $f$ sur un ensemble de mesure au moins $1-\varepsilon/2$. Par uniforme continuité de $g^{-1}$, il existe $\eta\in]0,\varepsilon[$ tel que si $\mathrm{dist}(x,g(C))<\eta$, alors $g^{-1}(x)\in C_\varepsilon$, c'est-à-dire que $x\in g(C_\varepsilon)$. Mais à partir d'un certain rang, on aura $\mathrm{dist}(f_k(x),f(x))<\eta$ sur un ensemble de mesure au moins $1-\varepsilon/2$. Sans perdre en généralité on se place désormais dans l'intersection de cet ensemble avec l'ensemble où $f$ et $g$ coïncident, intersection qui est de mesure au moins $1-\varepsilon$. Alors, tout point de $f_k(C)$ est dans $f(C_\varepsilon)$. On a donc deux ensembles, à savoir $f_k(C)$ et $f(C)$, dans un ensemble $f(C_\varepsilon)$ à peine plus gros au sens de la mesure ; de cela on déduit que :
\begin{eqnarray*}
\mu\big(f_k(C)\Delta f(C)\big) & < & \mu\big(f(C_\varepsilon)\setminus f_k(C)\big) + \mu\big(f(C_\varepsilon)\setminus f(C)\big)\\
															 & = & \mu\big(f(C_\varepsilon)\big) - \mu\big(f_k(C)\big) + \mu\big(f(C_\varepsilon)\big) - \mu\big(f(C)\big)\\
															 & = & 2\mu(C_\varepsilon)-2\mu(C),
\end{eqnarray*}
avec $\mu(C_\varepsilon)$ qui tend vers $\mu(C)$ lorsque $\varepsilon$ tend vers 0. Cela a lieu, on l'a dit plus haut, sur un ensemble de mesure au moins $1-\varepsilon$. On conclut par le fait que les cubes engendrent la tribu des boréliens et que l'ensemble des mesurables tels que $\mu(f_k(A) \,\Delta\, f(A))$ tende vers 0 est une tribu.

Au contraire, supposons que pour tout borélien $A$, $\mu(f_k(A) \,\Delta\, f(A))$ tend vers 0. Donnons-nous $\varepsilon>0$, toujours par théorème de Lusin, il existe un homéomorphisme conservatif $g$ arbitrairement proche de $f$ ; par uniforme continuité de $g$, il existe $\eta>0$ tel que l'image par $g$ de tout cube de diamètre inférieur à $\eta$ soit de diamètre plus petit que $\varepsilon$. Prenons une subdivision dyadique dont les cubes vérifient cette propriété. Pour chaque cube $C$ il va exister un temps à partir duquel $f_k(C)$ est dans $g(C)$ pour une proportion d'au moins $1-\varepsilon$ du cube. Ainsi, sur un ensemble de mesure au moins $1-\varepsilon$, la distance faible entre $f$ et $f_k$ sera d'au plus $\varepsilon$. On a montré l'équivalence des deux définitions.
\end{proof}

\begin{lemme}\label{autotopo}
$\mathrm{Auto}(X,\mu)$ est un groupe topologique.
\end{lemme}

\begin{proof}[Preuve du lemme \ref{autotopo}] Il s'agit essentiellement de remarquer que, par le théorème de Lusin, tout élément de $\mathrm{Auto}(X,\mu)$ coïncide avec une application uniformément continue sur un ensemble de mesure arbitrairement grande. Le reste de la preuve est identique à celle du lemme \ref{gpe-topo}.
\end{proof}

Une vérification importante à faire est que le théorème de Baire est vrai sur $\mathrm{Auto}(X,\mu)$ :

\begin{lemme}\label{autopolo}
$\mathrm{Auto}(X,\mu)$ muni de la topologie faible est un espace polonais.
\end{lemme}

\begin{proof}[Preuve du lemme \ref{autopolo}] On remarque tout d'abord que $\mathrm{Auto}(X,\mu)$ est un sous-espace fermé de celui des applications mesurables de $X$ dans $X$ muni de la topologie faible. En appliquant le théorème de Lusin, l'ensemble des applications continues est dense dans cet espace ; et on a déjà vu (dans la preuve du lemme \ref{séparable}) que l'ensemble des applications continues est séparable pour la topologie forte, \emph{a fortiori} pour la topologie faible. On a donc une base dénombrable d'ouverts sur l'ensemble des applications mesurables de $X$ dans $X$, dont la trace sur $\mathrm{Auto}(X,\mu)$ nous en donne une base dénombrable d'ouverts.

La métrique rendant $\mathrm{Auto}(X,\mu)$ complet est donnée par
\[\delta_{\mathit{faible}}(f,g) = d_{\mathit{faible}}(f,g)+d_{\mathit{faible}}(f^{-1},g^{-1}).\]
La démonstration du fait que cette métrique rend bien $\mathrm{Auto}(X,\mu)$ complet est similaire à celle du lemme \ref{complet} (il s'agit essentiellement de vérifier que les deux fonctions limites obtenues sont l'inverse l'une de l'autre). Et cette métrique engendre bien la topologie faible car $\mathrm{Auto}(X,\mu)$ est un groupe topologique.
\end{proof}

D'autres propriétés de la topologie faible, ainsi que des commentaires, se trouvent dans le chapitre \og \emph{weak topology} \fg~du livre de P. Halmos \cite{3}.
\bigskip

Donnons maintenant deux lemmes d'existence d'automorphismes conservatifs, issus principalement de la théorie des ensembles :

\begin{lemme}\label{ensembles}
\'Etant donnés $E$ et $F$ deux sous-ensembles mesurables de $X$ de même mesure, il existe une application $\Phi\in\mathrm{Auto}(X,\mu)$ telle que $\Phi(E) = F$.
\end{lemme}

\begin{proof}[Preuve du lemme \ref{ensembles}] La preuve de ce lemme se fait par récurrence transfinie. Nous allons construire $\Phi$ seulement sur $E$, la construction sur le complémentaire se faisant de la même façon. Donnons-nous $T$ un automorphisme ergodique de $X$ préservant la mesure (donné par exemple par l'image d'un automorphisme d'Anosov par l'application $\phi$ du corollaire \ref{Brown-mesure}). Alors il va exister $k\in\N$ tel que $\mu(f^k E\cap F)>0$. On pose alors $\Phi=f^k$ sur $E\cap f^{-k}F$. On recommence en remplaçant $E$ par à $E \cap {f^{-k}F}^\complement$. Une récurrence transfinie permet ainsi de définir $f$ sur $E$ tout entier : en effet, s'il y avait un moment où le procédé s'arrêtait, il resterait deux ensembles $E'\subset E$ et $F'\subset F$ sur lesquels $\Phi$ ne serait pas définie, mais ce que l'on a fait juste avant montre que l'on peut définir $\Phi$ sur des sous-ensembles de $E'$ et $F'$ de mesures non nulles.
\end{proof}

\begin{lemme}\label{point-fixe}
Soit $E$ un sous ensemble mesurable de $X$ et $\varepsilon>0$. Alors il existe $\Phi\in\mathrm{Auto}(X,\mu)$, sans point périodique dans $E$, et égal à l'identité sur $E^\complement$, tel que $d_{\mathit{forte}}(\Phi,\mathrm{Id})<\varepsilon$. 
\end{lemme}

\begin{proof}[Preuve du lemme \ref{point-fixe}] On se donne une subdivision dyadique dont les cubes sont de diamètre plus petit que $\varepsilon$. Plaçons nous dans un cube $C$ de cette subdivision et prenons $\alpha$ un nombre irrationnel quelconque. On choisit une direction privilégiée, disons l'axe des abscisses. On définit alors $\Phi$ sur chaque segment horizontal comme étant une rotation d'angle $\alpha$ selon la mesure de $E\cap [\textrm{segment horizontal}]$ : soit $S=\{(t,x_2,\dots,x_n)\}_{0\le t<1}$ un tel segment. Chaque point de $E$ est associé bijectivement (à un ensemble de mesure nulle près) à la valeur de sa fonction de répartition $f:S\cap E\to[0,\ell]$, avec $\ell$ la mesure de $S\cap E$. On pose alors\[
\Phi(t,x_2,\dots,x_n) = (f^{-1}(f(t)+\alpha \ell),x_2,\dots,x_n).\]
Sur $S\cap E^\complement$, on pose $\Phi = \mathrm{id}$. On a alors défini $\Phi$ sur tous les segments horizontaux de tous les $C\cap E$ ; $\Phi$ est bien mesurable et, par théorème de Fubini, préserve la mesure. Par conséquent $\Phi$ vérifie bien les conclusions du lemme.
\end{proof}

\section{L'espace des boréliens muni de la topologie faible}

Enfin, terminons par un petit lemme qui munit l'ensemble des boréliens de $X$ d'une métrique complète :

\begin{lemme}\label{polonais}
L'ensemble $\mathcal{B}$ des boréliens de $X$ muni de la topologie induite par la métrique $d(A,B) = \mu(A\Delta B)$ est un espace polonais, complet pour la distance $d$.
\end{lemme}

\begin{proof}[Preuve du lemme \ref{polonais}] Cet espace s'injecte isométriquement dans $L^1(X,\mu)$ via 
\begin{eqnarray*}
i : \mathcal B & \to & L^1(X,\mu)\\
A & \mapsto & \chi_A.
\end{eqnarray*}
Il suffit alors de vérifier que $i(\mathcal B)$ est un sous-espace fermé de $L^1(X,\mu)$ ; le fait que $\mathcal B$ est polonais découle alors de la complétude de $L^1(X,\mu)$ (théorème de Riesz-Fischer) et du fait qu'il est séparable (car $X$ est compact).
\end{proof}

\backmatter

\end{document}